\documentclass[a4paper,leqno,11pt]{amsart}

\usepackage{amsfonts,amssymb,verbatim,amsmath,amsthm,latexsym,textcomp,amscd}
\usepackage{latexsym,amsfonts,amssymb,epsfig,verbatim}
\usepackage{amsmath,amsthm,amssymb,latexsym,graphics,textcomp}
\usepackage{mathtools}
\usepackage{paralist}
\usepackage{quiver}
\usepackage{graphicx}
\usepackage{color}
\usepackage{url}
\usepackage{enumerate}
\usepackage[mathscr]{euscript}
\usepackage{tikz-cd}
\usetikzlibrary{matrix,arrows}
\usetikzlibrary{decorations.pathmorphing}
\usetikzlibrary{shapes.geometric}
\usepackage{dsfont}
\usetikzlibrary{matrix}
\usepackage{hyperref}
\usepackage{multirow}
\usepackage{comment}
\usepackage{enumitem}
\usepackage{adjustbox}

\input xy

\xyoption{all}

\theoremstyle{definition}
\newtheorem{theorem}{Theorem}[section]
\newtheorem{prop}[theorem]{Proposition}
\newtheorem{lem}[theorem]{Lemma}

\newtheorem{defn}[theorem]{Definition}

\newtheorem{rmk}[theorem]{Remark}

\newtheorem{ex}[theorem]{Example}

\newtheorem{thm}[theorem]{Theorem}
\newtheorem{notation}[theorem]{Notation}
\newtheorem{subsec}[theorem]{}
\newtheorem{cons}[theorem]{Construction}

\theoremstyle{plain}
\newtheorem*{thma}{Theorem A}
\newtheorem*{thmb}{Theorem B}
\newtheorem*{thmc}{Theorem C}
\newtheorem*{thmd}{Theorem D}

\theoremstyle{remark}
{\swapnumbers
   \newtheorem{ack}[theorem]{Acknowledgements} }
   
\newenvironment{myeq}[1][]
{\stepcounter{theorem}\begin{equation}\tag{\thetheorem}{#1}}
{\end{equation}}

\newenvironment{mysubsection}[2][]
{\begin{subsec}\begin{upshape}\begin{bfseries}{#2.}
			\end{bfseries}{#1}}
		{\end{upshape}\end{subsec}}

\newcommand{\C}{{\mathbb C}}

\newcommand{\Hyp}{{\mathbb H}}
\newcommand{\D}{{\mathbb D}}
\newcommand{\Z}{{\mathbb{Z}}}
\newcommand{\R}{{\mathbb R}}

\newcommand{\E}{{\mathbb E}}

\newcommand{\upi}{\underline{\pi}}

\newcommand{\Fib}{\mbox{Fib}}
\newcommand{\uF}{\underline{\F_2}}
\newcommand{\Cut}{\mbox{Cut}}
\newcommand{\Ass}{\mbox{Ass}}
\newcommand{\pt}{\mbox{pt}}
\newcommand{\THH}{\mbox{THH}}

\newcommand{\Top}{{\mbox{Top}}}

\newcommand{\Sph}{\mathbb{S}}

\newcommand{\F}{\mathbb F}

\newcommand\DD{{\mathcal D}}

\newcommand\FF{{\mathcal F}}
\newcommand\GG{{\mathcal G}}

\newcommand\LL{{\mathcal L}}
\newcommand\MM{{\mathcal M}}

\newcommand\PP{{\mathcal P}}

\newcommand{\THR}{\mbox{THR}}
\newcommand{\Set}{\mbox{Set}}

\newcommand\PMF{{\PP\kern-2pt\MM\FF}}

\newcommand\PML{{\PP\kern-2pt\MM\LL}}

\newcommand\tr{\operatorname{tr}}

\newcommand{\Esgmop}{\mathbb{E}_{\sigma}^{\otimes}}
\newcommand{\Esgm}{\mathbb{E}_{\sigma}}
\newcommand{\Topc}{\underline{\text{Top}} ^{C_2}}
\newcommand{\Spc}{\underline{\text{Sp}} ^{C_2}}

\newcommand{\fsubd}{\mathrel{{\scriptstyle\searrow}\kern-1ex^d\kern0.5ex}}
\newcommand{\bsubd}{\mathrel{{\scriptstyle\swarrow}\kern-1.6ex^d\kern0.8ex}}
\newcommand{\fsubeq}{\mathrel{\raise-.7ex\hbox{$\overset{\searrow}{=}$}}}
\newcommand{\bsubeq}{\mathrel{\raise-.7ex\hbox{$\overset{\swarrow}{=}$}}}

\newcommand{\tsh}[1]{\left\{\kern-.9ex\left\{#1\right\}\kern-.9ex\right\}}

\newcommand{\uM}{\underline{M}}

\newcommand{\Map}{\mathit{Map}}

\newcommand{\hofib}{\mbox{hofib}}
\newcommand{\colim}{\mbox{colim}}

\newcommand{\Th}{\mathit{Th}}

\newcommand{\res}{\mbox{res}}

\newcommand{\uZ}{\underline{\Z}}
\newcommand{\KR}{K\mathbb{R}}
\newcommand{\kR}{k\mathbb{R}}
\newcommand{\KRmod}{K\mathbb{R} \text{-mod}}
\newcommand{\HZ}{H\underline{\mathbb{Z}}}
\newcommand{\FII}{\underline{\mathbb{F}_2}}
\newcommand{\HFII}{H\underline{\mathbb{F}_2}}
\newcommand{\CII}{C_{2}}
\newcommand{\dsv}{DSV_\mathbb{C}^{C_2}}
\newcommand{\z}{\mathcal{Z}}
\newcommand{\gran}{Gr_{n} (\mathbb{C}^{\infty})_{\tau}}
\newcommand{\fbeta}{f_{\beta}}
\newcommand{\HP}{\mathbb{H} P}

\newcommand{\Obcat}{\mathcal{O} _{C_2} ^{op}}
\newcommand{\finC}{\underline{\mbox{Fin}}_{*} ^{C_2}}
\newcommand{\PicR}{\underline{\Pic}(R)}

\newcommand{\Res}{\mbox{Res}}
\newcommand{\Pic}{\mbox{Pic}}
\newcommand{\Psh}{\mbox{Psh}}
\newcommand{\Alg}{\mbox{Alg}}

\makeatletter
\@namedef{subjclassname@2020}{%
  \textup{2020} Mathematics Subject Classification}
\makeatother

\title[Twisted $R$-algebras]{Equivariant twisted $R$-algebras via Thom spectra}

\author[S. Basu]{Samik Basu}
\address{Stat Math Unit, Indian Statistical Institute Kolkata 700108, India}
\email{samik.basu2@gmail.com; samikbasu@isical.ac.in}

\author[A.Das]{Abhinandan Das}
\address{Stat Math Unit, Indian Statistical Institute Kolkata 700108, India}
\email{abhinandandas1111@gmail.com}

\hypersetup{
	colorlinks,
	citecolor=red,
	filecolor=black,
	linkcolor=blue,
	urlcolor=black
}


\begin{document}

\begin{abstract}
	For a $C_2$-commutative ring spectrum $R$, a twisted $R$-algebra is an $R$-module with a multiplication whose order is switched by the $C_2$-action. In this paper, we construct various quotients of $R$ as twisted $R$-algebras, when $R$ is an even real commutative ring spectrum. These are constructed as Thom spectra of maps out of  suitable $C_2$-actions on $S^1$ and $U(n)$. One such example is given by $\KR/2$ which is endowed with a twisted $\KR$-algebra structure. Other examples include quotients such as $M\R/(2,x_1,\dots, x_{n-1})$ over the real bordism spectrum $M\R$, and the real $2$-periodic Morava $K$-theories as modules over the real Morava $E$-theory spectra. In the context of twisted $R$-algebras, one may consider the real topological Hochschild homology, and for Thom spectra, one has a nice formula again as a Thom spectrum. We use this to obtain computations for the real topological Hochschild homology of $\KR/2$ as a twisted $\KR$-algebra. The computation also involves a splitting of the units spectrum $gl_1\KR$, which is an analogue of the classical splitting of the units of $K$-theory.  
\end{abstract}
\maketitle%
\thispagestyle{empty}%
\tableofcontents

\section{Introduction}
Real algebraic $K$-theory was defined in \cite{HM15} for categories with duality, and in this context, the trace map lands in real topological Hochschild homology, which is a genuine $C_2$-spectrum. The real topological Hochschild homology ($\THR$) was defined in \cite{DMP21} for {\it rings with anti-involution}, and the homotopy type of $\THR(\F_p)$ and $\THR(\Z)$ was determined. In the latter case, one needs to localize away from $2$.  Subsequently, there have been many computations of $\THR$ in a wide range of contexts (see \cite{HP25}, \cite{DMP24}, \cite{LRZ25}, for example). 

We recall from \cite{DMP21} that a ring spectrum with anti-involution is a $C_2$-spectrum with an underlying ring spectrum $A$ such that the involution is a map of ring spectra $A^{op} \to A$. Such a structure is described as an algebra over the $\E_\sigma$-operad, the operad of little $1$-disks thought of as residing inside the sign representation $\sigma$ of $C_2$. An example of a ring spectrum with anti-involution is a $C_2$-commutative ring spectrum, which is a $\E_\sigma$-algebra via the usual map of operads from $\E_\sigma$ to the $C_2$-$\E_\infty$-operad. An alternate example \cite{DMP21} comes from a commutative algebra possessing anti-involutions, such as matrix rings (with transpose serving as the involution). In this paper, we are interested in the $R$-module case, over a $C_2$-commutative ring spectrum $R$, constructing quotients of $R$  as $\E_\sigma$-$R$-algebras.

\begin{mysubsection}{Thom spectra and algebra structures} 
The Thom spectrum functor has traditionally been very useful for constructing algebra objects in the category of spectra. For an $n$-fold loop map $f$ from a space $X$ to $BG$, the classifying space of spherical fibrations, the Thom spectrum $\Th(f)$, constructed from the associated spherical fibration, has an $\E_n$-algebra structure \cite{LMS86}. This theorem implies that the spectra which arise out of cobordism theories via infinite loop maps to $BG$ are commutative ring spectra. Examples include $MU$, the spectrum for complex cobordism, $MO$, the spectrum for unoriented cobordism, and $MSO$, the spectrum for oriented cobordism. 

For a commutative ring spectrum $R$, there is an $R$-module Thom spectrum functor
\[\begin{tikzcd}[cramped]
\Top_{/ BGL_1 R} &	{\text{Top} _{ / {\tiny \Pic(R \text{-mod})}}}  && {R \text{-mod}}
	\arrow[from=1-1, to=1-2]
	\arrow["{{\Th^{R}(-)}}"', curve={height=18pt}, from=1-1, to=1-4]
	\arrow[dashed, from=1-2, to=1-4]
\end{tikzcd}\]
The construction of this functor first appeared in \cite{ABG14, ABGHR14}. The analogue in the equivariant case has been constructed in \cite{nonabelianpoincare}. The monoidality of these Thom spectra are also discussed therein (in fact, \cite{BSS20} in the non-equivariant case). The space $GL_1R$, which classifies the units in $0^{th}$ $R$-cohomology (that is, $R^0(-)^\times$), is an infinite loop space for a commutative ring spectrum $R$, and in this case, there is a spectrum $gl_1R$ such that $\Omega^\infty gl_1 R \simeq GL_1 R$. This implies the monoidality of the category $\Top_{/ BGL_1R}$. From the perspective of twisted $R$-algebra structures, we note the following corollary from this analysis (see Proposition \ref{twalgth}). 
\begin{prop}
For a commutative $C_2$-ring spectrum $R$, the $C_2$-Thom spectrum functor carries an $\mathbb{E}_{\sigma}$-map $f: X \rightarrow BGL_1R$ to an $\mathbb{E}_{\sigma}$-$R$-algebra. 
\end{prop}
\end{mysubsection}

\vspace*{0.2cm}

\begin{mysubsection}{Quotients as twisted algebras} 
In the study of commutative ring spectra and the category of modules over them, a natural question is whether one can perform the usual constructions which are familiar for commutative algebras. However, such questions are involved in general. The quotient $R/x$ defined for $x \in \pi_k R$, does not generally possess a ring structure even if $R$ is a commutative ring spectrum. The mod $2$ Moore spectrum $\Sph/2$ is such an example. The quotient $R/(x_1,\dots, x_k)$ may be proved to be a ring spectrum under some hypothesis on $\pi_\ast R$ and on the degree of the $x_i$ \cite{EKMM97, Ang08, BSS20}. One of the methods of proving such a result involves the $R$-module Thom spectrum functor $\Th^R$. We adapt this approach in this paper to the $C_2$-equivariant case. 

The first example of a quotient as a twisted algebra that we provide is $\KR/2$, the spectrum of mod $2$ real $K$-theory, in the category of modules over the real $K$-theory spectrum $\KR$ \cite{Ati66} (see Proposition \ref{kuR/2 as kuR-module}). This may be viewed as a $C_2$-analogue of the construction in \cite{Bas17}. We construct a $\Omega^\sigma$-map $S^1 \to BGL_1 \KR$ whose Thom spectrum is $\KR/2$. The space $S^1$ has trivial $C_2$-action, and one identifies $S^1 \simeq \Omega^\sigma \C P^\infty_\tau$, where $C_2$ acts on $\C P^\infty_\tau$ by complex conjugation. We also consider the connective cover $\kR$, and show that there is a $\Omega^\sigma$-map $S^{\rho+1} \to BGL_1 \kR$ whose Thom spectrum is $H\uZ$ (see Proposition \ref{hzmodkr}). Here, $\rho$ stands for the regular representation of $C_2$. The space $S^{\rho+1} \simeq \Omega^\sigma \Hyp P^\infty_\tau$ with the $C_2$-action on $\Hyp P^\infty_\tau$ as conjugation by $i\in \C$.   

We also discuss examples of quotients by a sequence of elements. In this case, as in \cite{BSS20}, we use the unitary group $U(n)$. The twisted monoid structure on $U(n)$ is given by the $C_2$-action $A\mapsto A^T$, the transpose of $A$. We prove that $U(n) \simeq \Omega^\sigma Gr_n(\C^\infty)_\tau$ (Proposition \ref{untw}), where $Gr_n(\C^\infty)\simeq \underset{m}{\varinjlim}~ Gr_n(\C^m)$ is the usual model for $BU(n)$. The $C_2$-action is by complex conjugation that sends an $n$-dimensional subspace $W$ to $\bar{W}$, its image under complex conjugation. 

As in the non-equivariant case, finite quotients by more than one element are proved to be ring spectra primarily in the even case. In the $C_2$-equivariant case, the evenness that is useful in the context of the paper is defined in \cite[Definition 3.1]{hillmeier} : A $C_2$-spectrum $X$ is said to be {\it even} if the $RO(C_2)$-graded homotopy groups satisfy $\upi_{m\rho - 1}(X)=0$ for all integers $m$. An example of an even spectrum is the spectrum for real $K$-theory \cite{Ati66}. From \cite{hukriz} one notes that the real complex cobordism spectrum $M\R$, and the real Brown Peterson spectrum $BP\R$ are even. The real Morava $E$-theory spectra $E\R_n$ are also examples of even spectra \cite{HS20}. An even ring spectrum $E\R$ is real oriented in the sense of \cite{hukriz} and one has an isomorphism 
\[E\R^\bigstar(\C P^\infty_\tau) \cong E\R^\bigstar(\pt)[[x]], \quad |x| = \rho. \]  
A $C_2$-spectrum $X$ is said to be {\it cofree } if the map $X \to F({EC_2}_+, X)$ is an equivalence. The spectra $M\R$, $BP\R$, and $E\R_n$ satisfy this property \cite[Theorem 4.1]{hukriz}. We prove the construction of quotients as twisted algebras in this setting. (see Theorem \ref{quottwralg})
\begin{thma}
Suppose $R$ is an even $C_2$-commutative ring spectrum which is cofree. Let $u_i \in \pi_{i\rho}^{C_2}(R)$ for $i=0,\dots, n-1$, such that $1+u_0$ is a unit in $\pi_0R$. Then, there is an $\Omega^\sigma$-map $U(n)_\tau \xrightarrow{f} BGL_1 R$, such that $\Th^R(f) \simeq R/(u_0,\dots, u_{n-1})$ as $R$-modules. As $f$ is an $\Omega^\sigma$-map, the Thom spectrum $R/(u_0,\cdots, u_{n-1})$ has the structure of a twisted $R$-algebra.   
\end{thma}

\vspace*{0.2cm}

\begin{ex}
The real bordism spectrum $M\R$ is a $C_2$-commutative ring spectrum which is both even and cofree. From \cite[Theorem 2.28]{hukriz} we have 
\[\pi^{C_2}_{\ast\rho} (M\R) \cong \Z[x_1,x_2,\dots ], \quad |x_i|=i\rho.\]
Using Theorem A, we now conclude that there are twisted $M\R$-algebra structures on the $M\R$-module $M\R/(2,x_1,\dots, x_{n-1})$. The real Morava $E$-theories $E\R_n$ are also even and cofree \cite[Theorem 1.9]{HS20},  and they satisfy 
\[\pi^{C_2}_{\ast\rho}(E\R_n) \cong W\F_{2^n}[[u_1,\dots,u_{n-1}]][u_n^{\pm}], \quad |u_i|=0, \quad |u_n|=2.\]
The real Morava $E$-theories are also commutative ring spectra, and so Theorem A applies in this situation. Thus, the ($2$-periodic) real Morava $K$-theories, which may be expressed as the $E\R_n$ modules $E\R_n/(2,u_1,\dots, u_{n-1})$, possess twisted $E\R_n$-algebra structures.  
\end{ex}

\end{mysubsection} 

\vspace*{0.2cm}

\begin{mysubsection}{$\THR$ of Thom spectra}
As the Thom spectrum functor  
\[\begin{tikzcd}[cramped]
\Top_{/ BGL_1 R} &	{\text{Top} _{ / {\tiny \Pic(R \text{-mod})}}}  && {R \text{-mod}}
	\arrow[from=1-1, to=1-2]
	\arrow["{{\Th^{R}(-)}}"', curve={height=18pt}, from=1-1, to=1-4]
	\arrow[dashed, from=1-2, to=1-4]
\end{tikzcd}\]
is symmetric monoidal, it is expected that there would be a formula for the topological Hochschild homology of $R$-module Thom spectra along the lines of \cite{BCS10}. This idea has been explained in the general context of equivariant factorization homology of equivariant Thom spectra  in \cite{nonabelianpoincare}, where the authors consider this formulation for the usual Thom spectrum functor 
\[ \Th: \Top_{/{\tiny \Pic(Sp^G)}} \to Sp^G.\]
Let $\Omega^V f : \Omega^V X \to \Pic(Sp^G)$ be a $V$-fold loop map, in which case $\Th(f)$ is a $\E_V$-algebra. Let $M$ be a $V$-framed $G$-manifold, which means that there is an equivariant isomorphism of tangent bundles $TM\cong M \times V$. In this situation, one has the formula \cite{nonabelianpoincare}
\[\int_M \Th(f) \simeq \Th\Big( \int_M \Omega^V X \to \int_M \Omega^V \Pic(Sp^G) \to \Pic(Sp^G)\Big).\]
Further, under suitable connectivity hypotheses on $X$, one has $\int_M \Omega^V X \simeq \Map_\ast(M_+, X)$. From \cite{horev}, we know that $\int_{S^\sigma}$ is equivalent to $\THR$, and for connected $X$, we also have 
\[\int_{S^\sigma} \Omega^\sigma X \simeq \Map_\ast(S^\sigma_+, X) \simeq L^\sigma X,\]
the $\sigma$-fold free loop space of $X$. In summary, the $\THR$ of the Thom spectrum of a $\Omega^\sigma$-map, is the Thom spectrum of a certain map out of the $\sigma$-fold free loop space.  In the $R$-module situation, we provide the following analogous result (see Theorem \ref{thrthom}). The key factor here is the explicit identification of the map $L^\sigma B^\sigma X \to BGL_1 R$ as $L^{\tilde{\eta}}(f)$, which is useful in our later computations.

\begin{thmb}
Let $f: X \rightarrow BGL_1(R)$ be a $\sigma$-fold loop map, that is, $f\simeq \Omega^\sigma B^\sigma f$, for a map $B^{\sigma}f: B^{\sigma}X \rightarrow B^{\rho}GL_1(R)$. In this situation, the real topological Hochschild homology of the Thom spectrum $\Th^R (f)$ admits the following equivalence as $R$-module
		$$\THR^{R} (\Th (f)) \simeq \Th (L^{\tilde{\eta}}f: L^{\sigma}B^{\sigma}(X) \rightarrow BGL_1 (R) ),$$
		where the map $L^{\tilde{\eta}}f$ is the following composition.
		$$\resizebox{\linewidth}{!}{$ L^{\sigma}B^{\sigma}X \xrightarrow{L^{\sigma}B^{\sigma}f} L^{\sigma}B^{\rho}GL_1(R) \xleftarrow{\simeq} BGL_1(R) \times B^{\rho}GL_1(R) \xrightarrow{id \times (-\tilde{\eta})^*} BGL_1(R) \times BGL_1(R) \xrightarrow{m} BGL_1(R) $}$$
\end{thmb}

\vspace*{0.2cm}

Stronger multiplicative properties on $f$ provide further simplifications for the formula in Theorem B. If $f$ is  a $\rho +1$-fold map, then \cite[Corollary 7.1.3]{nonabelianpoincare} adapted to this case implies that 
\[\THR^R(\Th^R(f)) \simeq \int_{S^\sigma} \Th^R(f) \simeq \int_{S^\sigma \times \R} \Th^R(f) \simeq \Th^R(f) \wedge \Sigma^\infty B^\sigma X_+.\]

\end{mysubsection}

\vspace*{0.2cm}

\begin{mysubsection}{$\THR^{K\R}(K\R/2)$} 
We compute the real topological Hochschild homology $\THR^{K\R}(K\R/2)$ for the mod $2$ real $K$-theory spectrum $K\R/2$ as a module over $K\R$. The computations are made using Theorem B for the twisted algebra structures that arise as Thom spectra. We then have 
\[\THR^{K\R}(K\R/2) \simeq \Th^{K\R}\Big(L^{\tilde{\eta}}f : L^\sigma \C P^\infty_\tau \to BGL_1 K\R\Big).\]  
We now describe a brief outline of the steps involved.  

\vspace*{0.2cm}

\noindent {\bf Step I : } The map $\tilde{\eta}^\ast$ in the formula for $\THR^{K\R}(K(\R/2)$ may be replaced by the projection. As a consequence, we have the following reduction (see Proposition \ref{THR from loop space})
\[\resizebox{\linewidth}{!}{$ \THR^{K\R}(K\R/2) \simeq \Th^{K\R}\Big(L^\sigma \C P^\infty_\tau \xrightarrow{L^\sigma B^\sigma f} L^\sigma B^\rho GL_1 K\R \simeq BGL_1 K\R \times B^\rho GL_1 K\R \xrightarrow{\pi_1} BGL_1 K\R \Big).$}\]
The proof relies on the fact that $\C P^\infty_\tau$ has a cellular filtration with cells of the type $\DD(n\rho)$ for $n\geq 0$. 

\vspace*{0.2cm} 

\noindent {\bf Step II : } As $\C P^\infty_\tau$ is a $C_2$-infinite loop space, we have a decomposition $L^\sigma \C P^\infty_\tau \simeq S^1 \times \C P^\infty_\tau$. We may thus express $L^\sigma \C P^\infty_\tau$ as a homotopy pushout of $\C P^\infty_\tau \leftarrow S^0 \times \C P^\infty_\tau \rightarrow \C P^\infty_\tau$. This implies that there is a cofiber sequence (see Proposition \ref{LES of THR}) 
\[K\R \wedge {\C P^\infty_\tau}_+ \xrightarrow{u-1} K\R \wedge {\C P^\infty_\tau}_+ \to \THR^{K\R}(K\R/2),\]
where $u\in K\R^0(\C P^\infty_\tau)^\times$. 

\vspace*{0.2cm} 

\noindent {\bf Step III : } The relation of the unit $u$ with the map $B^\sigma f : \C P^\infty_\tau \to B^\rho GL_1R$. We obtain a formula similar to \cite{Bas17}. (see Proposition \ref{identify u})

\vspace*{0.2cm} 

\noindent {\bf Step IV : } Provide an explicit form for $u$ in order to compute $\THR^{K\R}(K\R/2)$. For this purpose, we use a decomposition result for the units spectrum for real $K$-theory. 

\vspace*{0.2cm}

The steps I-IV allow us to complete the calculation of $\THR^{K\R}(K\R/2)$ for twisted algebra structures arising as Thom spectra. (see Theorem \ref{pi* calculation})
\begin{thmc}
The $RO(C_2)$-graded homotopy groups $\THR^{\KR}(\KR/2)$ are given by		
$$\pi _{k\rho +l} ^{C_2} (\THR^{\KR} (\KR /2)) \cong 
		\begin{cases}
			\Z /(2 ^{\infty}) \; & \text{ when } l \equiv 0,4 \, (\text{mod }8)\\
			\Z / {2} \;  \; & \text{ when } l \equiv 2,3 \, (\text{mod }8)\\
			0 \;  \; & \text{ when } l \equiv 1,5,6,7 \,(\text{mod }8).
		\end{cases}$$
			In fact, as a $\KR$-module  $\THR^{\KR} (\KR /2) \simeq \KR / (2^{\infty})$, where $\KR / (2^{\infty})$ is defined by the cofiber sequence
		$$\KR \rightarrow \KR [2^{-1}] \rightarrow \KR / (2^{\infty}).$$
\end{thmc}

\vspace*{0.2cm} 

The formulas in Theorem C show that for all the algebra structures that arise as Thom spectra, the $\THR$ computation is the same. One may compare this with $\THH^{KU}(KU/2)$ from \cite{BL04}, which is also constant over the moduli space of $A_\infty$-algebra structures. In comparison for odd primes $p$, $\THH^{KU}(KU/p)$ is not constant over the moduli space and has $p-1$ possible values \cite{Ang08}. The same is true for $\THH^{K_p^\wedge}(K/p)$ for algebra structures arising out of Thom spectra \cite{Bas17}. 

\end{mysubsection}

\vspace*{0.2cm} 

\begin{mysubsection}{A splitting of $gl_1 K\R$} 
In the computation of $\THR^{\KR}(\KR/2)$, a key ingredient is the structure of the units spectrum $gl_1\KR$. In the non-equivariant situation, the units spectrum for $KU$ splits as a wedge 
\[gl_1 KU \simeq \PP^2 gl_1 KU \vee \tau_{\geq 3} gl_1 KU.\]
The first summand is the second Postnikov section, and the second factor is the $2$-connective cover that may be identified as $bsu$. This splitting has been known classically though a proof in the literature is hard to find, for this formulation one may refer to \cite{BLM23}. The second Postnikov section has $\pi_0=\F_2$ and $\pi_2 = \Z$. The map 
\[ \C P^\infty \simeq \Omega^\infty \Sigma^2 H\Z \to \Omega^\infty \PP^2 gl_1 KU \to \Omega^\infty gl_1 KU \simeq GL_1 KU,\]
classifies the inclusion of line bundles as units in $K(X)$ ($=KU^0(X)$).  The spectrum $\PP^2gl_1KU$ is determined by a single $k$-invariant, which is computed as 
\[H\F_2 \xrightarrow{\beta \circ Sq^2} \Sigma^3 H\Z.\]
In the $C_2$-equivariant setting for real spectra, the slice tower plays a role analogous to the Postnikov tower. The result generalizes as the following theorem. (see Proposition \ref{ses of gl1P2kR} and Theorem \ref{splitglkr})
\begin{thmd}
The second slice section of $gl_1(K\R)$ splits off as a wedge summand, that is, 
\[gl_1(K\R) \simeq P^2(gl_1K\R) \vee \hat{K}.\]
Moreover, the slice section $P^2(gl_1K\R)$ is $\simeq \Fib(H\underline{\F_2} \xrightarrow{\beta_{C_2} \circ Sq^2_{C_2}} \Sigma^{2+\sigma} H\uZ)$. 
\end{thmd}
\end{mysubsection}

\vspace*{0.2cm}

\begin{notation} Throughout this paper, we use the following notations. 
\begin{itemize}
\item	$Sp$ refers to the $\infty$-category of orthogonal spectra. $Sp^G$ refers to the $\infty$-category of genuine $G$-spectra modelled via orthogonal spectra. $Sp^{BG}$ is the $\infty$-category of $G$-objects in $Sp$, called ``Borel-equivariant spectra" or ``naive $G$-spectra". 

\item For a $G$-representation $V$, $\DD(V)=$ the disc in $V$, $S(V)=$ the sphere in $V$, and $S^V=$ the $1$-point compactification of $V$. 

\item The notation $B^{cyc}(-)$ refers to the cyclic bar construction and $B^{di}(-)$ refers to the dihedral bar construction. 

\item $\E_n$ stands for the little $n$-cubes operad for $1\leq n \leq \infty$. If $V$ is a $G$-representation, the operad of little cubes inside $V$ is denoted by $\E_V$.  
	
\item Let $\mathbb{C}^{n} (\tau)$ be the complex $n$-plane $\mathbb{C}^{n}$ with $C_2$-action by complex conjugation. We denote $\C P^{n}_{\tau}$ as the $C_2$-space of complex lines in $\mathbb{C}^{n+1} (\tau)$ such that the underlying space in $\C P^{n}$ with $C_2$-action is given by complex conjugation. Also, we denote $\C P^{\infty}_{\tau}$ as the $C_2$-space whose underlying space is $\C P^{\infty}$ and $C_2$-action is given by complex conjugation.

\item The Thom spectrum for a spherical fibration classified by $X\xrightarrow{f} \Pic(Sp)$ is denoted by $\Th(f)$. In the $G$-equivariant case, the Thom spectrum of $X\xrightarrow{f} \Pic(Sp^G)$ is denoted by $\Th^G(f)$. In the $R$-module case, the $R$-module Thom spectrum is usually denoted as $\Th^R(f)$. 	It is sometimes convenient to use $X^{f}$ instead of $Th^{R}(f)$ for the Thom spectrum corresponding to a map $X \xrightarrow{f} \Pic(R)$.

\item The notation $P^n(-)$ refers to the $n^{th}$-slice section functor on $Sp^{C_2}$, and $\PP^n(-)$ refers to the $n^{th}$-Postnikov section functor on $Sp$. 
	\end{itemize}
\end{notation}

\vspace*{0.2cm}

\begin{mysubsection}{Organization}
In \S \ref{realcyc}, we discuss some preliminaries on real cyclic objects and the dihedral bar construction for twisted monoid objects. The preliminaries on equivariant stable homotopy theory are discussed in \S \ref{eqsthtpy}. We also recall the formulation of equivariant Thom spectra and their monoidal properties. 
The construction of quotients as twisted algebra Thom spectra is carried out in \S \ref{twralgconst}. We recall the formulation of real topological Hochschild homology in \S \ref{thrthL}, and prove a formula for the real topological Hochschild homology of $R$-module Thom spectra. In \S \ref{splitting of units}, we prove a splitting of the units of $K\R$. In \S \ref{thrcalc}, we compute the real topological Hochschild homology of $K\R/2$. 
\end{mysubsection}

\vspace*{0.2cm}

\begin{ack}
The research of the first author was supported by ANRF grant ARG/2025/000434/MS.
\end{ack}

\vspace*{0.2cm}

\section{Real cyclic objects and the dihedral bar construction} \label{realcyc} 

The definition of real topological Hochschild homology involves real analogues of cyclic sets and the cyclic bar construction. In this section, we recall their formulation. The primary reference for this is \cite{DMP21}. 
Our exposition follows the language of $\infty$-categories, in particular Joyal's quasi-categories, developed in \cite{HTT}, \cite{HA}, \cite{landbook21}, \cite{groth10}. This allows us to employ a formalism different from that of \cite{DMP21} and \cite{HM15}, without recourse to model-categorical language. 
	
	Moreover, we rely significantly on the theory of parametrized higher categories, developed by Barwick–Dotto–Glasman–Nardin–Shah, whose references include \cite{shah23}, \cite{nardin17}, \cite{nardin22}, \cite{barwick2016parametrized}, \cite{horev}. This framework generalizes the usual theory of higher categories by considering higher categories fibered over a base $\infty$-category. Our main interest lies in the situation where the base $\infty$-category is $\mathcal{O}^{op} _{G}$, the opposite category of $G$-orbits for a finite group $G$; providing us with a formalism for genuine equivariant homotopy theory for all subgroups of $G$, incorporating the appropriate restriction maps and Hill–Hopkins–Ravenel norms.

\begin{mysubsection}{Real simplicial and Real cyclic objects}
	Let $\Delta$ be the standard simplex category, i.e. the skeleton of the category of non-empty finite totally ordered sets and order preserving set maps. The involution  $\omega: \Delta \rightarrow \Delta$ is defined as $\omega([n]) = [n]$ and for a morphism $\alpha : [n] \rightarrow [k]$, 
	\begin{myeq}  \label{RSRC} 
\omega(\alpha)(i) = k - \alpha(n - i).
\end{myeq}
	This also induces an involution on the opposite category $\Delta^{op}$. 
	
	\begin{defn} \label{ex: Real simplicial object}
		A \textit{real simplicial object} in a category  $\mathscr{C}$ is a simplicial object $X: \Delta^{op} \rightarrow \mathscr{C}$ together with a natural transformation $\tau: X \rightarrow X\circ \omega$ such that the composition
		$X \xrightarrow{\tau} X \circ \omega \xrightarrow{\tau} X \circ {\omega ^2} = X$ is the identity natural transformation. Thus, this consists of a simplicial object $X$ in $\mathscr{C}$ with an involution $\tau_n : X_n \rightarrow X_n$, for every $n$, such that for every map $\alpha : [n] \rightarrow [k]$, 
		$$\alpha^* \circ \tau_k = \tau_n \circ (\omega(\alpha))^*.$$
	\end{defn}	

\vspace*{0.2cm}
	
		Connes' cyclic category $\Lambda$ has objects of the form $[n]$, same as that of $\Delta$, with morphisms generated by face maps $\delta_i : [n-1] \rightarrow [n], \; i= 0,1,...,n$, degeneracy maps $\sigma_j: [n+1] \rightarrow [n], \; j= 0,1,...,n$ and cyclic operators $\tau_n: [n] \rightarrow [n]$; satisfying the identities laid out in \cite[Definition 6.1.1]{LodayBook}. The cyclic category $\Lambda$ can also be endowed with an involution, similar in formula to $\Delta$.
		\begin{defn}
			A \textit{real cyclic object} in a category $\mathscr{C}$ comprises a cyclic object $X: \Lambda ^{op} \rightarrow \mathscr{C}$, involutions $\tau_n: X_n \rightarrow X_n$, such that for every map $\alpha : [n] \rightarrow [k]$, $\alpha^* \circ \tau_k = \tau_n \circ (\omega(\alpha))^*$.
		\end{defn}

\vspace*{0.2cm}

	The cosimplicial object $|\Delta^{\bullet}| : \Delta \rightarrow Top$, sending $[n]$ to the topological $n$-simplex $|\Delta^n|$ has an additional involutive structure given by $\tau : |\Delta^n| \rightarrow |\Delta^n|$ as follows
	\begin{myeq}\label{rmk: cosimplicial structure}
		\tau(t_0, t_1, . . . , t_n) = (t_n, t_{n-1}, ... , t_0).
	\end{myeq} 
	Therefore, for a category $\mathscr{C}$, which is tensored over $Top$, we can define geometric realization of a real simplicial object $X$ to be the $C_2$-object of $\mathscr{C}$ given by coend 
	$$|X| = X \otimes_{\Delta} |\Delta^{\bullet}|,$$
	with diagonal $C_2$-action. 
		The inclusion $\Delta \rightarrow \Lambda$ induces a simplicial structure on a cyclic object. Thus, we can talk about realization of a cyclic object as the realization of the underlying simplicial object, in a category which is tensored over $\Top$.

\begin{ex} \label{Lem: S-sigma as cyclic}
	The assignment $X_n = Aut_{\Lambda ^{op}}([n])$ with involutions $\sigma_n (t_n ^k) = t_n ^{n+1-k}$ for $k = 1,2,..,n$ and $\sigma_n (id) = id$, defines a real cyclic set. The underlying simplicial set has two non-degenerate simplices $t_0$ and $t_1$ such that the realization is homeomorphic to $S^1$ \cite[\S 6.1.10]{LodayBook} and the $C_2$-action, following \eqref{rmk: cosimplicial structure}, is via reflection, thus as a $C_2$-space its realization is homeomorphic to $S^{\sigma}$.
\end{ex}
\end{mysubsection}

\vspace*{0.2cm}

\begin{mysubsection}{$S^\sigma$ action on Real cyclic objects}
		One observes that the classical result of \cite[Theorem 7.1.4, Theorem 7.3.11]{LodayBook} refines to this case, and results in an $S^{\sigma}$-action on realizations of real cyclic spaces. Note that the multiplication on $S^1$, the underlying space of $S^{\sigma}$, commutes with the reflection (or equivalently, complex conjugation) action on it as $\overline{z_1. z_2} = \overline{z_1} \, . \, \overline{z_2}$; making $S^\sigma$ a monoid in $Top^{C_2}$. Therefore, by $S^{\sigma}$ action on a $C_2$-space we mean the existence of a $C_2$-map $m_X:S^{\sigma} \times X \rightarrow X$ such that the following diagram commutes in $C_2$-spaces:
	\[\begin{tikzcd}[cramped]
		{S^{\sigma} \times S^{\sigma} \times X} && {S^{\sigma} \times X} \\
		{S^{\sigma} \times X} && X.
		\arrow["{m_{S^{\sigma}} \times id}", from=1-1, to=1-3]
		\arrow["{id \times m_X}"', from=1-1, to=2-1]
		\arrow["{m_X}", from=1-3, to=2-3]
		\arrow["{m_X}"', from=2-1, to=2-3]
	\end{tikzcd}\]

The $S^\sigma$-action on the realization of real cyclic objects follows readily from the way the involution interacts with the cyclic structure.	The unique factorization of morphisms of Connes' category, $\Lambda$, \cite[Theorem 6.1.4]{LodayBook} is well-behaved with respect to the involution on $\Lambda$. Explicitly, if we write the unique factorization of $\alpha \in  Hom_{\Lambda}([m],[n])$ as $\alpha= \beta \circ \gamma$ for $\gamma \in Hom_{\Lambda}([m],[m])$ and $\beta \in Hom_{\Delta}([m],[n])$; then $\omega(\alpha)=\omega(\beta) \circ \omega(\gamma) $ is the unique factorization of $\omega(\alpha)$. As a consequence, the functorial assignment of factorization of $g \circ f$, as the diagram depicts:
	\[\begin{tikzcd}[cramped]
		{[m]} & {[n]} \\
		{[m]} & {[n]}
		\arrow["f", from=1-1, to=1-2]
		\arrow["{f^{*}(g)}"', from=1-1, to=2-1]
		\arrow["g", from=1-2, to=2-2]
		\arrow["{g_{*}(f)}"', from=2-1, to=2-2]
	\end{tikzcd}\]
	results in the factorizations 
$$\omega(g \circ f) = \omega(f^{*}(g)) \circ \omega(g_{*}(f)) = {\omega (f)}^{*}(\omega (g)) \circ {\omega (g)}_{*}(\omega (f)),$$
 forcing the following identities due to uniqueness:
		\begin{myeq} \label{factorization eq}
			\omega(f^{*}(g)) = {\omega (f)}^{*}(\omega (g)) , \; \omega(g_{*}(f)) = {\omega (g)}_{*}(\omega (f)).
		\end{myeq}
One uses \eqref{factorization eq} to define a functor $F$ from real simplicial spaces to real cyclic spaces as follows:
	\begin{align*}
	&	F \, : \, \text{\{Real simplicial spaces\}} \longrightarrow  \,\text{\{Real cyclic spaces\}}  \\
	&	F(Y)_{n} = Aut_{\Lambda ^{op}}([n]) \times Y_n  \text{; for $Y$ a real simplicial space}\\
	&	f^{*}(g,y)  = (f^{*}(g), (g_{*}(f))^{*}(y))  \text{; for } f \in Hom_{\Delta}([m],[n]) \\
	&	h^{*}(g,y) = (gh,y)  \text{; for } h \in Aut_{\Lambda}([n])
	\end{align*}
	
	Adopting the arguments of \cite[Section 7.1]{LodayBook} using the fuctoriality of involution and the uniqueness of factorization, we attain the following results:
	
	\begin{itemize}
		\item The functor $F$ is left adjoint to the forgetful functor.
		\item For a real cyclic space $X$, we have a morphism of real cyclic spaces $ev: F(X) \rightarrow X$ given by evaluation $(g,y) \mapsto g_{*}(y)$, inducing a $C_2$-equivariant map on the realization.
		\item For any real simplicial space $X$, the map $(p_1,p_2): |F(X)| \rightarrow S^{\sigma} \times |X|$ induced by $p_1$, a map of real cyclic spaces given levelwise $(g,y) \mapsto g$ and $p_2: |F(X)| \rightarrow |X|$ given by $(g,y;u) \mapsto (y, g^{*}(u))$ for $(g,y;u) \in F(X)_n \times \Delta^n$; is a $C_2$-equivariant homeomorphism.
	\end{itemize}

	\begin{defn}({$S^{\sigma}$ action}) \label{s sigma action}
		For a real cyclic space $X$, we define an $S^{\sigma}$-action as the following composition:
		$$S^{\sigma} \times |X| \xrightarrow{(p_1,p_2)^{-1}} |F(X)| \xrightarrow{|ev|} |X|$$
		which is $C_2$-equivariant as a consequence of the above discussion.
	\end{defn}
	
\vspace*{0.2cm}

	\begin{ex}
		For the real cyclic set of Example \ref{Lem: S-sigma as cyclic}, the above map results in the usual $C_2$-monoid structure on $S^{\sigma}$. Moreover, the fact that map of Definition \ref{s sigma action} actually defines an action, can be realised using the map of real cyclic spaces $FF(X) \rightarrow F(X)$ given as $(g,h,x) \mapsto (gh,x)$. See \cite[\S 7.1.10, \S 7.1.11]{LodayBook} for the nonequivariant case. 
	\end{ex}

\vspace*{0.2cm}

\end{mysubsection}

\begin{mysubsection}{Dihedral bar construction and twisted free loop space}
The cyclic bar construction of topological monoids generalizes to twisted monoids as dihedral bar construction.
	
	\begin{defn} \label{twistedmonoid}
		A $C_2$-space $M$ is called a \textit{twisted monoid} if it is non-equivariantly a monoid such that $\tau (x y)= \tau(y) \tau(x)$ for all $x,y \in M$. A morphism between two twisted monoids is a $C_2$-map which is also a non-equivariant monoid map. 
	\end{defn}
	
\vspace*{0.2cm}

	\begin{rmk} \label{E sigma algebra in top}
		Twisted monoids can be identified as grouplike $\mathbb{E}_{\sigma}$-algebras in the category of $C_2$-spaces. Explicitly, they are algebras, in the category $\text{Top}^{C_2}$, over the $C_2$-operad $\Ass^{\sigma}$ \cite[Remark 2.3]{DMP21}, given by $\Ass^{\sigma}_{n} = \Sigma_n$ with $C_2$-action given by $\alpha \mapsto \tau_{n} \circ \alpha$ where $\tau_n$ is the permutation of $\{1,2,...,n\}$ which reverses the order. Moreover, one can observe that it is equivalent to the little $\sigma$-disks operad, whose genuine operadic nerve \cite{Bon19} is $\mathbb{E}_{\sigma}$, as a consequence of the fact that the mapping spaces of the $C_2$-$\infty$ operad of $\sigma$-framed representations are homotopically discrete.
	\end{rmk}

\vspace*{0.2cm}
	
	\begin{cons} \label{cons twisted bar}
		Given a twisted monoid $M$, the twisted bar construction is given as the following composition:
		$$B^{\sigma} (-) : \{\text{Twisted monoids} \} \xlongrightarrow{B^{\sigma} _{\bullet} (-)} \{\text{Real simplicial space} \} \xlongrightarrow{|-|} \text{Top}^{C_2}$$
		whose underlying simplicial space is same as the bar construction, explicitly $B^{\sigma} _{n} (M)= M^{\times n}$ with $C_2$-action as $\tau(m_1, m_2, ..., m_n) = (\tau(m_n), \tau(m_{n-1}) , ..., \tau(m_1))$. The face and degeneracy maps are the same as the non-equivariant case.
	\end{cons}
	
\vspace*{0.2cm}
\end{mysubsection}
	
	\begin{mysubsection}{$C_2$-fixed points of a $\sigma$-fold loop space} \label{fixed pts of loop}
		Let $(X,x_0)$ be a based $C_2$-space. We denote $\Omega^{\sigma}X$ as the $C_2$-space of all (non-equivariant) based maps from $S^{\sigma}$ to $X$ with $C_2$-action by conjugation, i.e. $\Omega^{\sigma}X = \text{Map}_{*}(S^{\sigma},X)$. The $C_2$-fixed points as a subspace of $\Omega^\sigma X$ is given by
		\begin{equation*}
			\begin{aligned}
				(\Omega^{\sigma}X)^{C_2}  & = \text{Map}_{*} ^{C_2}(S^{\sigma},X) \\
										& =  \{ \gamma : [0,1] \rightarrow X | \gamma(0)=x_0=\gamma(1) \text{ and } \gamma(1-t)= \tau \gamma(t)\}  \\
										& =  \{ \gamma : [0,1] \rightarrow X | \gamma(0)=x_0, \gamma(1) \in X^{C_2} \}  \\
										& \simeq \mbox{hofib} \, (X^{C_2} \xhookrightarrow{i} X).
			\end{aligned}	
		\end{equation*}
		
	\end{mysubsection}
	
	This formula is used to show that the twisted bar construction gives a $\sigma$-fold delooping.
	
	\begin{thm} \label{sigma fold deloop}
		Given a connected twisted monoid $M$, there exist a natural $C_2$-map $M \rightarrow \Omega^{\sigma} B^{\sigma} M$ which is a $C_2$-weak equivalence.
	\end{thm}
	
	We recall the main steps of the proof. The details are available in \cite{Liu20}.
	
	\begin{proof}
		Analogous to the non-equivariant case the map is given as follows:
		\begin{equation*}
			\begin{aligned}
				M \longrightarrow & \; \Omega^{\sigma} B^{\sigma} M \\
				m \longmapsto & \; (t \mapsto |m \times (t, 1-t)|) 
			\end{aligned}
		\end{equation*}
		which is evidently $C_2$-equivariant under conjugation action on the right. Now, the proof proceeds along the steps given below:
		
		\begin{itemize}
			\item The non-equivariant equivalence is a classical result.
			\item As per \ref{fixed pts of loop}, the $C_2$-fixed points on the right hand side is $\hofib \, ((B^{\sigma} M)^{C_2} \xhookrightarrow{i} B^{\sigma} M)$.
			\item The space $(B^{\sigma}M)^{C_2}$ is the quotient of the disjoint union of $(B^{\sigma} _{2k+1} M \times \Delta^{2k+1})^{C_2}$, consisting of points $\{(m_1,m_2,...,m_n, m, \tau m_n, \tau m_{n-1},..., \tau m_1) \times (v_0, v_1, ..., v_n, v_n, ..., v_0)\}$, where $a \in M^{C_2}$. Additionally, it is homeomorphic to $B(*, M, M^{C_2})$, with action of $m \in M$ on $M^{C_2}$ is given as: $m' \mapsto m m' \tau(m)$.
			\item The inclusion map $i: (B^{\sigma}M)^{C_2} \hookrightarrow B^{\sigma}M$ is  homotopic to the projection map: $B(*,M,M^{C_2}) \rightarrow B(*,M,*)$. Thus, the homotopy fiber is $\simeq M^{C_2}$.
		\end{itemize}	
	\end{proof}
	
		\begin{cons} \label{cons dihedral bar}
		Given a twisted monoid $A$, the dihedral bar construction is given by the following composition:
$$B^{di} (-) : \{\text{Twisted monoids} \} \xlongrightarrow{B^{di} _{\bullet} (-)} \{\text{Real cyclic space} \} \xlongrightarrow{|-|} \text{Top}^{C_2}.$$
		The underlying cyclic space is same as the cyclic bar construction, explicitly $B^{di} _{n} (A)= A \times A^{\times n}$ with $C_2$-action as $\tau(a_0,a_1, a_2, ..., a_n) = (\tau(a_0),\tau(a_n), \tau(a_{n-1}) , ..., \tau(a_1))$. The face and degeneracy maps are same as the non-equivariant case. As a consequence of \ref{s sigma action}, we get an $S^{\sigma}$-action on $B^{di}A$. Moreover, we have a projection map on the underlying real simplicial spaces $B^{di}_{\bullet} (A) \rightarrow B^{\sigma}_{\bullet} (A)$ given levelwise by $(a_0,a_1, a_2, ..., a_n) \mapsto (a_1, a_2, ..., a_n)$; inducing a $C_2$-map on realizations.
	\end{cons}

\vspace*{0.2cm}

	\begin{thm} \label{dihedral free loop}
		Given a (connected) twisted monoid $A$, there exists a canonical $C_2$-equivariant map $\gamma: B^{di} (A) \rightarrow L^{\sigma} B^{\sigma}A$, which is a $C_2$-weak equivalence.
	\end{thm}
	The proof is entirely analogous to the non-equivariant case. We briefly mention the details here.
	\begin{proof}
		The $C_2$-equivariant map $\gamma$ is the adjoint of the following composition of $S^{\sigma}$-action map \ref{s sigma action} and projection map \ref{cons dihedral bar}:
		$$S^{\sigma} \times B^{di}A \longrightarrow B^{di}A \longrightarrow B^{\sigma}A.$$
		The map fits into the commutative diagram:
		\begin{myeq}\label{diagram free loop}
 \begin{tikzcd}[cramped]
			A && {B^{di}A} && {B^{\sigma}A} \\
			{\Omega^{\sigma} B^{\sigma}A} && {L^{\sigma} B^{\sigma}A} && {B^{\sigma}A}
			\arrow[from=1-1, to=1-3]
			\arrow["\simeq"', from=1-1, to=2-1]
			\arrow[from=1-3, to=1-5]
			\arrow["\gamma"', from=1-3, to=2-3]
			\arrow[equals, from=1-5, to=2-5]
			\arrow[from=2-1, to=2-3]
			\arrow["ev", from=2-3, to=2-5]
		\end{tikzcd}
\end{myeq}
		where, by Theorem \ref{sigma fold deloop}, the left vertical equivalence holds; and the bottom row forms a fiber sequence in $\text{Top}^{C_2}$. Therefore, by five lemma the proof is concluded, provided we establish that the top row is a fiber sequence. As the non-equivariant case is addressed by classical theory, enough to show it is a fiber sequence on fixed points, i.e. $A^{C_2} \rightarrow (B^{di}A)^{C_2} \rightarrow (B^{\sigma}A)^{C_2}$ is a fiber sequence of spaces. This is a consequence of the arguments that follow:
				\begin{itemize}
					\item From Theorem \ref{sigma fold deloop}, $(B^{\sigma}A)^{C_2} \cong B(*,A,A^{C_2})$ with action of $a \in A$ on $A^{C_2}$ is given as: $a' \mapsto a a' \tau(a)$.
					\item The space $(B^{di}A)^{C_2}$ is the quotient of the disjoint union of $(B^{di} _{2k+1} A \times \Delta^{2k+1})^{C_2}$, consisting of points $\{(a', a_1,a_2,...,a_n, a, \tau a_n, \tau a_{n-1},..., \tau a_1) \times (v_0, v_1, ..., v_n, v_n, ..., v_0)\}$, where $a',a \in A^{C_2}$. Additionally, it is homeomorphic to $B(A^{C_2}, A, A^{C_2})$, with left-action of $a^{L} \in A$ on $A^{C_2}$ is given as: $a \mapsto  a^{L} a \tau(a^{L})$ and right-action of $a^{R} \in A$ on $A^{C_2}$ is given as: $a \mapsto \tau(a^{R}) a a^{R}$.
					\item This results in the fiber sequence: $A^{C_2} \rightarrow B(A^{C_2}, A, A^{C_2}) \rightarrow B(*, A, A^{C_2})$.
				\end{itemize}
	\end{proof}

\vspace*{0.2cm}

\section{Equivariant Stable Homotopy Theory}\label{eqsthtpy}

	In this section, we recall some preliminaries from equivariant homotopy theory, whose classical references include \cite{LMS86}, \cite{MM02}, \cite{HHR21}, \cite{schwede2010lectures}. We refer to \cite{NS18} for the $\infty$-categorical point of view.

	\subsection{$G$-spectra}
	The $\infty$-category of $G$-spectra, we shall be considering, is the one with $\pi_{*}$-isomorphisms as equivalences. We begin with the definition of the homotopy groups of orthogonal spectra with $G$-action.
\begin{defn}
\begin{enumerate}
\item For $X \in Fun(BG, Sp^{O})$ and $V$, a $G$-representation, we define the following based spaces:
			$$X(V) := L(\R ^{n}, V)_{+} \wedge _{O(n)} X_n$$ 
			(\cite[Equation 2.2]{schwede2010lectures}), where $L(\R ^{n}, V)$ is the space of linear isometries from $\R ^n$ to $V$, admitting a natural $O(n)$-action, and the space $X(V)$ admits $G$-action via diagonal action.\\
\item The homotopy groups of $X$ can be defined as 
	$$\pi ^{H} _{n} X := \underset{{W \in G \text{-rep}}}{\colim } [S^{ W+n}, X(W)]^{H} _{*}.$$
\end{enumerate}
\end{defn}

\vspace*{0.2cm}

One notes that a map of $G$-spectra is a $\underline{\pi}_{*}$-isomorphism if and only if all the geometric fixed points are weakly equivalent. 
We note the relevant pieces in the following definition.
\begin{defn}
		\begin{enumerate}
			\item Let $X \in Fun(BG, Sp^{O})$. One defines the \textit{geometric fixed points} as an orthogonal spectrum $(\Phi^{G} X)_n := X(\R^n \otimes \rho_{G})^G $ \cite[\S 7.3]{schwede2010lectures}.\\
			\item The $RO(G)$-graded homotopy groups of $X$ can be defined as follows:
			$$\pi ^{H} _{V} X := \underset{W \in G \text{-rep}}{\colim} [S^{V \oplus W}, X(W)]^{H} _{*}$$
			for $V \in RO(G)$. \\
			\item A map $f: X \rightarrow Y$ in $Fun(BG, Sp^{O})$ is called an equivalence if for all $H \leqslant G$, the map $\Phi^{H}(f): \Phi^{H} X \rightarrow \Phi^{H}Y$ is a stable equivalence of orthogonal spectra. This is equivalent to \cite[Theorem 7.12]{schwede2010lectures} :  $\pi ^{H} _{n} (f): \pi ^{H} _{n} X \rightarrow \pi ^{H} _{n} Y$ is an isomorphism for all $n \in \Z$. \\
			\item The $\infty$-category of \textbf{genuine $G$-spectra} $Sp^G$ is obtained from $N(Fun(BG, Sp^{O}))$ by inverting above equivalences using Dwyer-Kan localization, This admits a symmetric monoidal structure such that the geometric fixed point functor $\Phi^{H} : Sp^{G} \rightarrow Sp$ is symmetric monoidal.
		\end{enumerate}
	\end{defn}
	
\vspace*{0.2cm} 

	The level-wise fixed point functor is homotopically well-defined on $\Omega$-$G$-spectra \cite[\S 7.1, Definition 3.18]{NS18}, and becomes a lax symmetric monoidal functor $(-)^H: Sp^G \rightarrow Sp$. The embedding of the trivial representation into the regular representation $\rho_G$ induces a natural transformation of lax symmetric monoidal functors $(-)^H \rightarrow \Phi^H (-)$. 
	
	We also have a functor from spectra with $G$-action to the category of spectra as follows $Sp^{BG} \longrightarrow Sp$ given by $X\mapsto F(EG_+,X)^G$, which extends to the \textit{homotopy fixed point} functor
	$$(-)^{hG}: Sp^G \xrightarrow{\text{(forgetful) }i} Sp^{BG} \rightarrow Sp.$$
The homotopy fixed points functor $(-)^{hG}$ is also lax symmetric monoidal.	The map $EG_+ \to S^0$ induces a natural transformation of these two functors: $(-)^G \rightarrow (-)^{hG}$. Following  \cite[Theorem II.2.7, Corollary II.2.8]{NS18} one may regard $Sp^{BG}$ as a full subcategory of $Sp^G$ of \textbf{Borel complete} genuine $G$-spectra, i.e. all such $X \in Sp^G$ for which the natural map $X^H \rightarrow X^{hH}$ is an equivalence of spectra $\forall H \leqslant G$.
	
	\begin{thm}
		The forgetful functor fits into an adjunction
		$$\text{(forgetful) }i: Sp^G \rightleftarrows Sp^{BG}: B_{G}$$
		such that the right adjoint $B_{G}$ is fully-faithful with essential image as the subcategory of Borel complete objects. Moreover, the forgetful functor is symmetric monoidal, while its right adjoint $B_{G}$ is lax symmetic monoidal.
	\end{thm} 
	
\vspace*{0.2cm}

	\begin{mysubsection}{Stable homotopy groups of spheres}
		There have been computations of $C_2$-equivariant stable homotopy groups of spheres due to Bredon \cite{Bre01}, \cite{Bre02}, Araki, Iriye \cite{AI82}, and recent advances like \cite{GI2024}. We mention some homotopy classes relevant to our computations.
		\begin{itemize}
			\item $\pi^{C_2} _0 (S^0) \cong \Z\{C_2 / e _{+}\} \oplus \Z\{C_2 / C_2\}$, where the generators corresponds to the maps 
\[S^0 \xrightarrow[\text{Thom}]{\text{Pontrjagin}} C_2 / e _{+} \rightarrow S^0, \quad \mbox{and } S^0 \xrightarrow{id} S^0 \mbox{ respectively.}\]
			\item $\pi^{C_2} _0 (S^{\rho}) \cong \Z\{i\}$, where the generator is the inclusion of fixed points $i: S^0 \hookrightarrow S^{\sigma}$.
			\item $\pi^{C_2} _{\sigma} (S^0) \cong \Z \{ \tilde{\eta}\}$ with the generator being the hopf map $\tilde{\eta} : S^{2\rho -1} \rightarrow S^{\rho}$ under the identifications $\mathbb{CP}^1 _{\tau} \simeq S^{\rho}$ and $S^{2\rho -1}$ as unit quaternions with the $C_2$-action as conjugation by $i$ : $(a+bi+cj+dk) \mapsto (a+bi-cj-dk)$.
		\end{itemize}
		
	\end{mysubsection}
	
\vspace*{0.2cm}

	\subsection{Mackey Functors}
	The homotopy groups of a spectrum $X$ in $Sp^G$ has the structure of a Mackey functor. We briefly recall its definition following \cite{LMM81}, \cite{Dress73}, \cite{GM95}.
		\begin{defn}
			The \textit{Lindner Category} $\mathcal{B}_G$ is the category consisting of:
			\begin{itemize}
				\item \textit{Objects}: Finite $G$-sets
				\item \textit{Morphisms}: Equivalent classes of $G$-spans, i.e. a morphism $f: X \rightarrow Y$ is given by an equivalence class of maps $X \leftarrow Z \rightarrow Y$ where two such maps $X \leftarrow Z_i \rightarrow Y$, for $i= 1,2$ are equivalent if there is a $G$-set isomorphism $Z_1 \rightarrow Z_2$ making the following diagram commute
				\[\scalebox{0.8}{$
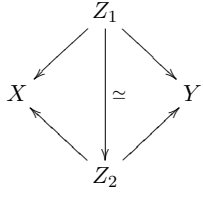
\xymatrix{
					& {Z_1}\ar[dd]^{\simeq} \ar[ld] \ar[rd] \\
					X && Y \\
					& {Z_2} \ar[lu] \ar[ru]}$}\]
				\item \textit{Compositions}: They are defined as pullbacks:
				
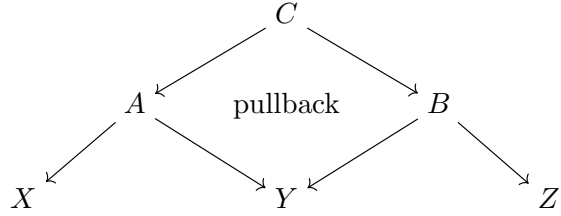
\[\begin{tikzcd}[cramped]
					&& C \\
					& A & {\text{pullback}} & B \\
					X && Y && Z
					\arrow[from=1-3, to=2-2]
					\arrow[from=1-3, to=2-4]
					\arrow[from=2-2, to=3-1]
					\arrow[from=2-2, to=3-3]
					\arrow[from=2-4, to=3-3]
					\arrow[from=2-4, to=3-5]
				\end{tikzcd}\]
			\end{itemize}
\end{defn} 

\vspace*{0.2cm}

			Note that the mapping sets $\mathcal{B}_G (X,Y)$ are commutative monoids under disjoint union. Then the Burnside category is defined as the group completion of these mapping sets.
		\begin{defn}
			The Burnside category $Burn_G$ is defined by taking formal differences in the mapping sets of the Lindner category; explicitly the objects are same as that of Lindner category and the morphism sets are $Burn_G (X,Y) := (\mathcal{B}_G (X,Y), \amalg)^{gp}$. 
		\end{defn}
		
\vspace*{0.2cm}

		This enables us to define Mackey functors as enriched functors taking values in abelian groups.
				\begin{defn}
			A Mackey functor is a contravariant additive functor $\underline{M}: Burn_G ^{op} \rightarrow \mathsf{Ab}$. Note that the additivity ensures it takes disjoint unions ($\amalg$) to direct sums ($\oplus$).
\end{defn}

\vspace*{0.2cm}

 Unravelling the definition, one can think of a Mackey functor $\underline{M}$ as an assignment of an abelian group $\underline{M} (G/H)$ for each orbit $G/H$ along with morphisms between orbits giving structure maps. These are generated by the following maps.
			\begin{itemize}
				\item \textit{Restrictions}: For $K \leqslant H \leqslant G$ subgroups of $G$ the map $G/K \xleftarrow{id} G/K \twoheadrightarrow G/H$ induces restriction map $Res^{H} _{K}: \underline{M}(G/H) \rightarrow \underline{M}(G/K)$.
				\item \textit{Transfers}: For $K \leqslant H \leqslant G$ subgroups of $G$ the map $G/H \twoheadleftarrow G/K \xrightarrow{id} G/K$ induces transfer map $tr^{H} _{K}: \underline{M}(G/K) \rightarrow \underline{M}(G/H)$.
				\item \textit{Weyl group actions}: For $H \leqslant G$ and an element $\gamma \in W_{G}(H):= N_{G}(H)/H$, the Weyl group of $H$ in $G$, the map $G/H \xleftarrow{\gamma} G/H \xrightarrow{id} G/H$ induces an action of $W_{G}(H)$ on $\underline{M}(G/H)$.
		\end{itemize}		
				The category of $G$-Mackey functors, with natural transformations as morphisms, is denoted as $\MM_G$.
			
		\vspace*{0.2cm}


		\begin{mysubsection}{$C_2$-Mackey Functors}
			A common way of describing Mackey functors is in terms of the Lewis diagram. For $G=C_2$, the  diagram depicting a $C_2$-Mackey functor is given by
			\[\begin{tikzcd}[cramped]
				&&&{\underline{M}(C_2/C_2)} \\
\uM	:  \\
			&&& {\underline{M}(C_2/e)}
				\arrow["{Res^{C_2} _e}"', shift right, from=1-4, to=3-4]
				\arrow["{tr^{C_2} _e}"', shift right, from=3-4, to=1-4]
				\arrow["{W_{C_2}(e) \cong C_2}"', shift left, from=3-4, to=3-4, loop, in=300, out=240, distance=5mm]
			\end{tikzcd}\]
		\end{mysubsection}
		
\vspace*{0.2cm}

		\begin{ex} We now note down some common $C_2$-Mackey functors. 
			\begin{enumerate}
				\item \textit{Constant Mackey functors (for $G = C_2$)}:
				\[\begin{tikzcd}[cramped]
					 &  && {\underline{\Z}:} & {\underline{\Z}(C_2/C_2)=\Z} && {\FII:} & {\FII(C_2/C_2)=\mathbb{F}_2} \\
					&  &&& {\underline{\Z}(C_2/e)=\Z} &&& {\FII(C_2/e)=\mathbb{F}_2}
					\arrow["id"', shift right, from=1-5, to=2-5]
					\arrow["id"', shift right, from=1-8, to=2-8]
					\arrow["{\times 2}"', shift right, from=2-5, to=1-5]
				\end{tikzcd}\]
				\item \textit{Burnside ring Mackey functor $\underline{A}$}: For finite group $G$, this is defined as a corepresentable functor $\underline{A}:= Burn_{G} (-, *)$. Explicitly, $$\underline{A}(G/H)= Burn_G (G/H,*) = Burn_H(*,*)= A(H) \text{ ; the Burnside ring of } H $$
				$$\Res^{H} _{K}: A(H) \rightarrow A(K) \text{ ; restriction of }H \text{-sets to } K \text{-sets}$$
				$$tr ^{H} _{K}: A(K) \rightarrow A(H) \text{ , given as } S \mapsto H \times_{K} S$$
				In case of $G=C_2$, we have the following description:
				\[\begin{tikzcd}[cramped]
					 &  && {\underline{A}:} & {\underline{A}(C_2/C_2)=A(C_2) \cong \Z \oplus \Z} \\
					&  &&& {\underline{A}(C_2/e)=A(e) \cong \Z}
					\arrow[shift right, from=1-5, to=2-5]
					\arrow[shift right, from=2-5, to=1-5]
				\end{tikzcd}\]
				with $tr^{C_2} _{e}= \begin{pmatrix} 1 \\ 0 \end{pmatrix}$ and $\Res^{C_2} _e = \begin{pmatrix} 1 & 2 \end{pmatrix}$.
				\item $\underline{\Z}_{-}$ is given as $\underline{\Z}_{-} (C_2/C_2) = 0$ and $\underline{\Z}_{-} (C_2/e)=\Z$, with $C_2$-action by $\{\pm 1\}$. The restriction and transfer maps are $0$.
			\end{enumerate}
		\end{ex}
	
\vspace*{0.2cm} 

		The following theorem states the existence of Eilenberg-MacLane $G$-spectra, which is characterized by the fact that its integral homotopy groups being concentrated at $0$. This appears in \cite[Theorem 5.3]{GM95}, \cite{LMM81}.
		\begin{thm} \label{Mackey map}
			For every Mackey functor $\underline{M}$, there exists a unique, upto weak equivalence, $G$-spectrum $H\underline{M}$ such that
			$\underline{\pi}_i (H\underline{M}) \cong 
			\begin{cases}
				\underline{M} & \text{  for } i = 0 \\
				0 & \text{  otherwise}
			\end{cases}$
			
			Moreover, for two Mackey functors $\underline{M}, \underline{M}'$; we have $[H\underline{M}, H\underline{M}'] \cong \MM_G (M,M')$.
		\end{thm}

\vspace*{0.2cm}
		
	The following identification is useful for our computations in the forthcoming sections. 		
		\begin{prop} \label{Mackey identify}
			We have a weak equivalence of spectra $\Sigma ^{1-\sigma} H\underline{\Z} \simeq H \underline{\Z}_{-}$.
		\end{prop} 
		
		\begin{proof}
			Using the cofiber sequence $C_2/e _{+} \rightarrow S^0 \rightarrow S^{\sigma}$, we arrive at the following long exact sequence:
			$$... \rightarrow \pi_i ^{C_2} (H\underline{\Z}) \rightarrow \pi_i ^{e} (H\underline{\Z}) \rightarrow \pi_i ^{C_2} (\Sigma^{1-\sigma} H\underline{\Z}) \rightarrow \pi_i ^{C_2} (\Sigma H\underline{\Z}) \rightarrow \pi_{i-1} ^{e} (H\underline{\Z}) \rightarrow ...$$
			The cases where one of the groups on either sides of $\pi_i ^{C_2} (\Sigma^{1-\sigma} H\underline{\Z})$ is non-zero, are $i=0,1$.
			Thus we have the following two sequences:
				$$... \rightarrow \pi_0 ^{C_2} (H\underline{\Z}) \cong \Z \xrightarrow{\simeq} \pi_0 ^{e} (H\underline{\Z}) \cong \Z \rightarrow \pi_0 ^{C_2} (\Sigma^{1-\sigma} H\underline{\Z}) \rightarrow \pi_0 ^{C_2} (\Sigma H\underline{\Z}) \cong 0 \rightarrow ...$$
				$$... \rightarrow \pi_1 ^{e} (H\underline{\Z}) \cong 0 \rightarrow \pi_1 ^{C_2} (\Sigma^{1-\sigma} H\underline{\Z}) \rightarrow \pi_1 ^{C_2} (\Sigma H\underline{\Z}) \cong \Z \xrightarrow{\simeq} \pi_{0} ^{e} (H\underline{\Z}) \cong \Z \rightarrow ...$$
				Therefore, we arrive at 
\[\underline{\pi} _i (\Sigma^{1-\sigma} H \underline{\Z}) \cong 
				\begin{cases}
					\underline{\Z}_{-} & \text{ for } i=0 \\
					0 & \text{ otherwise.}
				\end{cases}\]
The proof is now complete from Theorem \ref{Mackey map}.
		\end{proof}

\vspace*{0.2cm}
	
	\subsection{$C_2$-equivariant Steenrod operations}
	The $C_2$-equivariant Steenrod Algebra has been computed in \cite{hukriz}(see also \cite{Voe03}). To fix our notation, let $\mathcal{A}^{\bigstar} _{C_2}:= H\underline{\F_2} ^{\bigstar} H\underline{\F_2}$ denote the $C_2$-equivariant Steenrod Algebra, and the coefficient ring is denoted as 
		$$\mathcal{H}^{\bigstar}:= H^{\bigstar} _{C_2}(S^0; \underline{\F_2}) \cong \mathbb{F}_2[a_{\sigma}, u_{\sigma}] \oplus \Sigma  \mathbb{F}_2 \{a_{\sigma}^{-k} u_{\sigma}^{-l}\} \text{ for } k,l > 0 \text{; with } |a_{\sigma}|= \sigma, |u_{\sigma}|=\sigma -1$$
		Note that elements here are written in cohomological grading. The Steenrod algebra $\mathcal{A}^{\bigstar} _{C_2}$ is generated by equivariant Steenrod squares. We note some of their properties:
		
		\begin{enumerate}
			\item $|Sq^i _{C_2}| = \begin{cases}
				 k(1+\sigma) & \text{ for } i= 2k \\
				 k(1+\sigma)+1 & \text{ for } i=2k+1
			\end{cases}$
			\item $Sq^0 _{C_2} = id$ and $Sq^1 _{C_2}= \beta_{C_2}$, Bockstein homomorphism
			\item $\beta_{C_2} \circ Sq^{2i} _{C_2} = Sq^{2i+1} _{C_2}$
		\end{enumerate}
	
		The short exact sequence of Mackey Functors $\underline{\Z} \xrightarrow{\times 2} \underline{\Z} \rightarrow \underline{\F_2}$ induces the connecting homomorphism Bockstein $\beta_{C_2} : H\underline{\F_2} \rightarrow \Sigma H\underline{\Z}$, whose restriction gives the Bockstein $Res^{C_2} _{e}(\beta_{C_2}) = \beta : H\F_2 \rightarrow \Sigma H\Z$. We now make the following observation which we shall use later.
		\begin{prop} \label{beta sq2}
			The restriction map 
\[\Res^{C_2}_e : [\HFII, \Sigma^{1+\sigma} \HZ]^{C_2} \to [H\F_2, \Sigma^3 H\Z]\cong \F_2\{\beta \circ Sq^2\}\]
 is injective.
		\end{prop}
		
		\begin{proof}
			To prove this we use the cofiber sequence ${C_2 / e}_{+} \rightarrow S^0 \rightarrow S^{\sigma}$. This gives rise to the following diagram whose top row is exact. 
			\[\begin{adjustbox}{max width=\linewidth}
\begin{tikzcd}[cramped]
				{...} & {[\Sigma^{\sigma}\HFII,\Sigma^{2+\sigma}\HZ]^{C_2}} & {[\HFII,\Sigma^{2+\sigma}\HZ]^{C_2}} & {[\HFII \wedge {C_2 / e}_{+},\Sigma^{2+\sigma}\HZ]^{C_2}} & {...} \\
				& {[\HFII,\Sigma^{2}\HZ]^{C_2} \cong 0} && {[H\F_2,\Sigma^{3}H\Z]}
				\arrow[from=1-1, to=1-2]
				\arrow[from=1-2, to=1-3]
				\arrow["\cong", from=1-2, to=2-2]
				\arrow[from=1-3, to=1-4]
				\arrow["{res^{C_2}_e}"', from=1-3, to=2-4]
				\arrow[from=1-4, to=1-5]
				\arrow["\cong", shorten <=2pt, shorten >=2pt, from=1-4, to=2-4]
			\end{tikzcd}
\end{adjustbox}\]
			The computation of $[\HFII,\Sigma^{2}\HZ]^{C_2}$ in the above diagram proceeds via the diagram below in which the top row is again exact 		
			\[\begin{tikzcd}[cramped]
				{...} & {[\HFII,\Sigma^{1-\sigma}\HZ]^{C_2}} & {[\HFII,\Sigma^{2}\HZ]^{C_2}} & {[\HFII \wedge {C_2 / e}_{+},\Sigma^{2}\HZ]^{C_2}} & {...} \\
				&&& {[H\F_2,\Sigma^{2}H\Z]\; \cong\; 0}
				\arrow[from=1-1, to=1-2]
				\arrow[from=1-2, to=1-3]
				\arrow[from=1-3, to=1-4]
				\arrow[from=1-4, to=1-5]
				\arrow[equals, from=1-4, to=2-4]
			\end{tikzcd}\]
			and the subsequent set of isomorphisms	
			$$[\HFII,\Sigma^{1-\sigma}\HZ]^{C_2} \; \cong \; [\HFII,\HZ_{-}]^{C_2} \; \cong \;  \MM_{C_2}(\underline{\F}_2,\underline{\Z}_{-}) \; \cong \; 0$$
 where the first isomorphism comes from \ref{Mackey identify}, and the second isomorphism comes from \ref{Mackey map}.
		\end{proof}
		
		\vspace*{0.2cm}
		
\subsection{Equivariant Thom spectra}
We now recall the $C_2$-equivariant Thom spectrum functor and its properties. The details are discussed in \cite{nonabelianpoincare}, \cite{BD18}, \cite{HW20}.
First recall the non-equivariant $R$-module Thom spectrum construction associated to a stable $R$-fibration $E$ over $X$ for an $\mathbb{E}_{\infty}$-ring spectrum $R$ \cite[Definition 2.20]{ABG14}. This is defined as the colimit
$$X \xrightarrow{E} \Pic(R \text{-mod}) \hookrightarrow R \text{-mod}.$$
Here, $\Pic(R \text{-mod})$ is the $\infty$-groupoid of invertible $R$-modules and $X \rightarrow \Pic(R \text{-mod})$ is the classifying map for the $R$-fibration over $X$. This results in functorial construction of the colimit preserving functor, $R$-module Thom spectrum, as a left Kan extension:
\[\begin{tikzcd}[cramped]
	& {\Pic(R \text{-mod})} \\
	{\text{Top} _{ / {\tiny \Pic(R \text{-mod})}}} & {\Psh(\Pic(R \text{-mod}))} && {R \text{-mod}}
	\arrow[hook, from=1-2, to=2-2]
	\arrow[from=1-2, to=2-4]
	\arrow["\simeq", from=2-1, to=2-2]
	\arrow["{{\Th^{R}(-)}}"', curve={height=18pt}, from=2-1, to=2-4]
	\arrow[dashed, from=2-2, to=2-4]
\end{tikzcd}\]
A similar approach is taken to construct equivariant Thom spectrum functor.  Recall some definitions from \cite{shah23}, \cite{nardin17}, \cite{nardin22}, \cite{barwick2016parametrized}, \cite{horev} on parametrized higher category theory.
	A $G$-$\infty$-category is a coCartesian fibration over $\Obcat$, i.e. $\underline{\mathcal{C}} \rightarrow \Obcat$.

\begin{ex}
	These examples are from \cite[Example 1.4, Defn. 2.14]{nardin17}
	\begin{itemize}
		\item The $G$-$\infty$-category $\underline{\text{Top}}^{C_2}$ consists of the following data \\		
		(1) For each coset the corresponding fiber takes the from $\underline{\text{Top}}^{C_2} _{[C_2/H] }\simeq \text{Top}^H$. \\
		(2) The coCartesian lift of $\phi: C_2/C_2 \rightarrow C_2/e$ induces the restriction map:
		$$\res: \text{Top}^{C_2} \simeq \underline{\text{Top}}^{C_2} _{[C_2/C_2]} \rightarrow \underline{\text{Top}}^{C_2} _{[C_2/e]} \simeq \text{Top}$$
		$\underline{Sp}^{C_2}$ has similar description.
		
		\item The (discrete) $C_2$-$\infty$-category of finite $C_2$-sets is denoted as $\finC$, with objects as $C_2$-maps $U \rightarrow O$, where $O$ is a $C_2$-orbit, and morphisms as the triple
		\[\begin{tikzcd}[cramped]
			{U_1} & {\phi ^{*} U_1} & U & {U_2} \\
			{O_1} & {O_2} & {O_2} & {O_2}
			\arrow[from=1-1, to=2-1]
			\arrow[from=1-2, to=1-1]
			\arrow[from=1-2, to=2-2]
			\arrow[hook', from=1-3, to=1-2]
			\arrow["f", from=1-3, to=1-4]
			\arrow[from=1-3, to=2-3]
			\arrow[from=1-4, to=2-4]
			\arrow["\phi", from=2-2, to=2-1]
			\arrow[equals, from=2-2, to=2-3]
			\arrow[equals, from=2-3, to=2-4]
		\end{tikzcd}\]
		
	\end{itemize}
	
\end{ex}

\vspace*{0.2cm}

\begin{defn}
	A $C_2$-symmetric monoidal $C_2$-$\infty$-category is a cocartesian fibration $p: \underline{\mathcal{C}}^{\otimes} \rightarrow \finC$ with \textbf{$G$-Segal condition} of \cite[Defn. 2.2.3, Prop. 2.2.6]{nardin22}. The underlying $C_2$-$\infty$-category is the pullback under the functor $\Obcat \rightarrow \finC$ sending $O$ to $O = O$ and $O_1 \xleftarrow{\phi} O_2$ to 
	\[\begin{tikzcd}[cramped]
		{O_1} & {O_2} & {O_2} & {O_2} \\
		{O_1} & {O_2} & {O_2} & {O_2}
		\arrow[from=1-1, to=2-1]
		\arrow["\phi", from=1-2, to=1-1]
		\arrow["{=}", from=1-2, to=2-2]
		\arrow[equals, from=1-3, to=1-2]
		\arrow[equals, from=1-3, to=1-4]
		\arrow["{=}", from=1-3, to=2-3]
		\arrow["{=}", from=1-4, to=2-4]
		\arrow["\phi", from=2-2, to=2-1]
		\arrow[equals, from=2-2, to=2-3]
		\arrow[equals, from=2-3, to=2-4]
	\end{tikzcd}\]
\end{defn}

\vspace*{0.2cm}

\begin{ex}
	The $G$-$\infty$-category $\underline{\text{Top}}^{C_2}$ and $\underline{Sp}^{C_2}$ admits $C_2$-symmetric monoidal structures ${\underline{\text{Top}}^{C_2}}^{\otimes}$ and ${\underline{Sp}^{C_2}}^{\otimes}$. See \cite[Example 2.1.10]{nonabelianpoincare} for more details. Moreover, there is an essentially unique $C_2$-symmetric monoidal structure on $\underline{Sp}^{C_2}$ with sphere spectrum as unit by \cite[Cor. 3.28]{nardin17}.
\end{ex}

\vspace*{0.2cm}

We consider the case when $R$ is a $C_2$-$\mathbb{E}_{\infty}$-ring spectrum, which makes the (parametrized) module category $\underline{R \text{-mod}}$ a $C_2$-symmetric monoidal $\infty$-category. Denote this symmetric monoidal category as the coCartesian fibration $\underline{R \text{-mod}}^{\otimes} \twoheadrightarrow \underline{\text{Fin}}_{*} ^{C_2}$, with the underlying $C_2$-$\infty$-category as $\underline{R \text{-mod}} \twoheadrightarrow \Obcat$.

\vspace*{0.2cm}

\begin{mysubsection}{The Picard $C_2$-space $\PicR$}
	The Picard $C_2$-space is defined as the maximal $C_2$-$\infty$-groupoid of $\underline{R \text{-mod}}$ spanned by invertible $R$-modules, explicitly $\PicR \subseteq \underline{R \text{-mod}}$ is the $\infty$-subgroupoid spanned by invertible objects and coCartesian morphisms. Note that $\PicR$, being an $\infty$-groupoid, the map $\PicR \rightarrow \Obcat$ is a left fibration. Hence, it  corresponds to a $C_2$-space, $\PicR \in \text{Top}^{C_2}$, with $C_2$-fixed points given as $\PicR_{[C_2 / C_2]} \simeq \text{Pic}(\underline{R \text{-mod}}_{[C_2 / C_2]})$. In \cite{nonabelianpoincare}, authors described $C_2$-symmetric monoidal structure on Picard space.  
\end{mysubsection}

\vspace*{0.2cm}

\begin{mysubsection}{$C_2$-Thom spectrum functor} \label{C2 thom functor}
We now recall the construction and some properties of the equivariant Thom spectrum functor, based on \cite[Thm 5.0.2, Def 5.2.1, Prop 5.2.7]{nonabelianpoincare}.	Given a $C_2$-$\mathbb{E}_{\infty}$-ring spectrum $R$, the $R$-module Thom spectrum $C_2$-functor 
\[\Th^R: \underline{\text{Top}^{C_2}}_{/ {\tiny \PicR}} \simeq \underline{\text{Psh}}_{C_2}(\PicR) \rightarrow \underline{R\text{-mod}}\]
 is given as $C_2$-left Kan extension of the inclusion $\PicR \hookrightarrow \underline{R\mbox{-mod}}$ along the parametrized Yoneda embedding $j: \PicR \rightarrow \underline{\text{Psh}}_{C_2}(\PicR)$. Explicitly
	\[\begin{tikzcd}[cramped]
		& \PicR \\
		{\underline{\text{Top}^{C_2}}_{/{\tiny \PicR}}} & {\underline{\text{Psh}}_{C_2}(\PicR)} && {\underline{R\text{-mod.}}}
		\arrow["j"', hook, from=1-2, to=2-2]
		\arrow[from=1-2, to=2-4]
		\arrow["\simeq", from=2-1, to=2-2]
		\arrow["{{{\Th^R(-)}}}"', curve={height=18pt}, from=2-1, to=2-4]
		\arrow[dashed, from=2-2, to=2-4]
	\end{tikzcd}\]
The Thom spectrum satisfies the following:
	\begin{itemize}
		\item $Th$ \textit{strongly} preserves $C_2$-colimits.
		\item It accepts a $C_2$-symmetric monoidal refinement:
		$\Th^{\otimes} : (\underline{{\text{Top}}}^{C_2}_{/{\tiny \PicR}})^{\otimes} \rightarrow ({\underline{R\text{-mod}}})^{\otimes}$
		\item For an invertible $R$-module $E$ and a $C_2$-map $f: X \rightarrow \PicR$ which is $C_2$-homotopic to a constant map with value $E$, we get $\Th(f) \simeq E \wedge \Sigma ^{\infty} _{+} X$.
	\end{itemize}
\end{mysubsection}

\vspace*{0.2cm}

We now demonstrate that twisted monoids, under Thom spectrum functor, goes to $\mathbb{E}_{\sigma}$-algebras in $R$-modules. 
\begin{prop}\label{twalgth}
	The $C_2$-Thom spectrum functor carries an $\mathbb{E}_{\sigma}$-map $f: X \rightarrow \PicR$ to an $\mathbb{E}_{\sigma}$-$R$-algebra. 
\end{prop}

\begin{proof}
	The parametrized slice category $\underline{{\text{Top}}}^{C_2}_{/ {\tiny \PicR}}$ admits a $C_2$-symmetric monoidal $\infty$-category structure $(\underline{{\text{Top}}}^{C_2}_{/ {\tiny \PicR}})^{\otimes}$, following \cite[Appendix A.5]{nonabelianpoincare}. An $\mathbb{E}_{\sigma}$-map $f: X \rightarrow \PicR$ gives rise to an $\mathbb{E}_{\sigma}$-algebra:
	\[\begin{tikzcd}[cramped]
		{\mathbb{E}_{\sigma} ^{\otimes}} && {(\underline{{\text{Top}}}^{C_2}_{/ {\tiny \PicR}})^{\otimes}} \\
		& {\underline{\mbox{Fin}}_{*} ^{C_2}}
		\arrow["f", from=1-1, to=1-3]
		\arrow[from=1-1, to=2-2]
		\arrow[from=1-3, to=2-2]
	\end{tikzcd}\]
	Composing with the symmetric monoidal refinement $\Th^{\otimes}$ \ref{C2 thom functor}, we get the following diagram which sends coCartesian lifts of inert morphisms to coCartesian lifts. 
	\[\begin{tikzcd}[cramped]
		{\mathbb{E}_{\sigma} ^{\otimes}} & {(\underline{{\text{Top}}}^{C_2}_{/ {\tiny \PicR}})^{\otimes}} & {({\underline{R\text{-mod}}})^{\otimes}} \\
		& {\underline{\mbox{Fin}}_{*} ^{C_2}}
		\arrow["f", from=1-1, to=1-2]
		\arrow[from=1-1, to=2-2]
		\arrow["{\Th^{\otimes}}", from=1-2, to=1-3]
		\arrow[from=1-2, to=2-2]
		\arrow[from=1-3, to=2-2]
	\end{tikzcd}\]
This gives an $\mathbb{E}_{\sigma}$-algebra structure on $\Th^R(f)$.
\end{proof}

\vspace*{0.2cm}

\begin{mysubsection}{Units of $C_2$-$\mathbb{E}_{\infty}$-ring spectra}
Given a $C_2$-$\mathbb{E}_{\infty}$-ring spectra $R$, the definition of the $C_2$-$\mathbb{E}_{\infty}$-space of units of $R$, $GL_1(R)$, mimics the construction in the non-equivariant case described in \cite{ABG14}, \cite{Sch04}. We follow the definition of the space of units in \cite{Rek11}.  
\begin{defn} \label{grouplike space}
		A $C_2$-$\mathbb{E}_{\infty}$-space $X \in \Alg _{\mathbb{E}_{\infty}}(\underline{\Top}^{C_2})$ is called \textit{group-like} if both $\pi_0(X^{C_2})$ and $\pi_0(X)$ are groups under the multiplicative structure induced from the $\mathbb{E}_{\infty}$-structure of $X$. We denote the full subcategory of group-like $C_2$-$\mathbb{E}_{\infty}$-spaces as $\Alg^{gp} _{\mathbb{E}_{\infty}}(\underline{\Top}^{C_2})$.
\end{defn}

\vspace*{0.2cm}

\begin{prop} \label{adjoint GL1}
	The inclusion functor $\Alg^{gp} _{\mathbb{E}_{\infty}}(\underline{\Top}^{C_2}) \hookrightarrow \Alg _{\mathbb{E}_{\infty}}(\underline{\Top}^{C_2})$ admits a right adjoint, denoted as $gr$.
\end{prop}

\begin{proof}
	A $\mathbb{E}_{\infty}$-algebra in $C_2$-sets $A \in \Alg _{\mathbb{E}_{\infty}}(\underline{\Set}^{C_2})$ is a functor into commutative monoids, that is,  both $A$ and its fixed points $A^{C_2}$ are commutative monoids, and the map between them is a map of monoids. One can look at the subgroup of all invertible elements $gr(A)$ as the functor:
	$$\begin{aligned}
		\Obcat & \xlongrightarrow{gr(A)} \Set \\
		C_2/e & \mapsto \text{invertible elements of  }A \\
		C_2 / C_2 & \mapsto \text{invertible elements of } A^{C_2} \text{ which restrict to invertible elements in }A
	\end{aligned}$$
	The $\mathbb{E}_{\infty}$-structure restricts to group-like $C_2$-$\mathbb{E}_{\infty}$-structure on $gr(A)$. Moreover, any $\mathbb{E}_{\infty}$-map from a group-like one to $A$ factors through $gr(A)$ by construction. 
	Now, for any $C_2$-$\mathbb{E}_{\infty}$-space $X$, we can construct the space $gr(X) \in \Alg^{gp} _{\mathbb{E}_{\infty}}(\underline{\Top}^{C_2})$ as the following pullback in $\underline{\Top}^{C_2}$:
	\[\begin{tikzcd}[cramped]
		{gr(X)} && X \\
		{gr(\pi_0(X))} && {\pi_0(X)}
		\arrow[from=1-1, to=1-3]
		\arrow[from=1-1, to=2-1]
		\arrow["{\pi_0}", from=1-3, to=2-3]
		\arrow[from=2-1, to=2-3]
	\end{tikzcd}\]
	Note that the functor 
\[\pi_0 : \Alg^{gp} _{\mathbb{E}_{\infty}}(\underline{\Top}^{C_2}) \rightarrow \Alg^{gp} _{\mathbb{E}_{\infty}}(\underline{\Set}^{C_2})\]
is given by $\pi_0(X) _{[C_2/e]} = \pi_0(X)$ and $\pi_0(X) _{[C_2/C_2]} = \pi_0(X^{C_2})$. The fact that this is the right adjoint follows from the construction.
\end{proof}

\vspace*{0.2cm}

\begin{defn}
	For a $C_2$-$\mathbb{E}_{\infty}$-ring spectrum $R$, we define the units $GL_1(R)$ as the group-like $C_2$-$\mathbb{E}_{\infty}$-space given by the composite functor:
	$$GL_1:\Alg _{\mathbb{E}_{\infty}}(\underline{Sp}^{C_2}) \xrightarrow{\Omega^{\infty}} \Alg _{\mathbb{E}_{\infty}}(\underline{\Top}^{C_2}) \xrightarrow{gr} \Alg^{gp} _{\mathbb{E}_{\infty}}(\underline{\Top}^{C_2}).$$
The $\E_\infty$-structure on $\Omega^\infty R$ considered here is the one induced by the $\E_\infty$-ring structure on $R$.	
\end{defn}

\vspace*{0.2cm}

Expanding the definition further, one can write the space of units $GL_1(R)$ as the following pullback in $\underline{\Top}^{C_2}$:
	\begin{myeq} \label{pullback units}
\begin{tikzcd}[cramped]
		{GL_1(R)} && \Omega^{\infty}R \\
		{\underline{\pi}_0(R) ^{\times}} && {\underline{\pi}_0(X)}
		\arrow[from=1-1, to=1-3]
		\arrow[from=1-1, to=2-1]
		\arrow["{\underline{\pi}_0}", from=1-3, to=2-3]
		\arrow[from=2-1, to=2-3]
	\end{tikzcd}
\end{myeq}
	Note that the level-wise commutative monoid structure on $\underline{\pi}_0(R)$, comes from the multiplicative structure of $R$. As $R$ is a $C_2$-$\E_\infty$-ring spectrum, $\upi_0(R)$ is a commutative Green functor. Recall that this carries the following data : 
	\begin{itemize}
		\item A commutative ring structure on $\pi^e _0(R)$ and $\pi^{C_2} _0(R)$.
		\item The restriction map $\Res^{C_2} _e : \pi^{C_2} _0(R) \rightarrow \pi^e _0(R)$ is a ring homomorphism.
		\item The involution $\tau: \pi^e _0(R) \rightarrow \pi^e _0(R)$ is a ring homomorphism.
		\item Frobenius reciprocity relations: $Tr(x Res(y))= Tr(x)y$ and $Tr(Res(y) x)= y Tr(x)$. 
 	\end{itemize}
 	Therefore, we have $\underline{\pi}_0(R) ^{\times} _{[C_2 / C_2]} = (\pi^{C_2} _0(R))^{\times}$, as the units of $\pi^{C_2} _0(R)$ restricts to units in $\pi^e _0(R)$ under the ring map $\Res^{C_2} _e$. In particular, $(GL_1(R))^{C_2}$ is just collection of connected components of $(\Omega^{\infty}R)^{C_2}$ corresponding to $(\pi^{C_2} _0(R))^{\times}$, resulting $\pi^{C_2} _0(GL_1(R)) \cong (\pi^{C_2} _0(R))^{\times}$, and for any $C_2$-representation $V$ with $V^{C_2} \neq 0$, we get $\pi^{C_2} _V(GL_1(R)) \cong \pi^{C_2} _V(R)$.

	The equivariant recognition principle of \cite{GM17}, \cite{CHL24} establishes an equivalence between group-like $C_2$-$\mathbb{E}_{\infty}$-spaces and connective $C_2$-spectra. Thus, we get a connective spectrum $gl_1(R)$ corresponding to the space of units $GL_1(R)$. The following proposition observes the standard adjunction such a construction satisfies. 
	\begin{prop}
		The functor $gl_1(-): \Alg _{\mathbb{E}_{\infty}}(\underline{Sp}^{C_2}) \rightarrow \underline{Sp}^{C_2} _{\geq 0}$ admits a left adjoint given by $\Sigma^{\infty} _{+} \Omega^{\infty}$.
	\end{prop}

	\begin{proof}
		This is a composite of the following adjunctions : 
		$$\Sigma^{\infty} _{+} \Omega^{\infty} : \underline{Sp}^{C_2} _{\geq 0} \rightleftarrows \Alg^{gp} _{\mathbb{E}_{\infty}}(\underline{\Top}^{C_2}) \rightleftarrows \Alg _{\mathbb{E}_{\infty}}(\underline{\Top}^{C_2}) \rightleftarrows \Alg _{\mathbb{E}_{\infty}}(\underline{Sp}^{C_2}): gl_1$$
		The leftmost adjunction is from \cite{GM17}, \cite{CHL24}, the second adjunction from the left comes from Proposition \ref{adjoint GL1}, and the rightmost one is the adjunction $\Sigma_{+} ^{\infty} \perp \Omega^{\infty}$.
	\end{proof}

\vspace*{0.2cm}

We now note that $BGL_1(R)$ admits an $C_2$-$\mathbb{E}_{\infty}$-map to $\PicR$ \cite{HW20}. In the $C_2$-$\infty$-groupoid $\PicR$, restricting to the $\infty$-subgroupoid of invertible objects which are equivalent to $R$ yields the $C_2$-space $BGL_1(R)$. To see this one may identify the $C_2$-space $\Omega^{\infty}R$ as the space of endomorphisms of $R$ in $R$-mod, i.e. 
$$\Omega^{\infty}R \simeq F_{R} (R,R) \simeq F(S^0 , R).$$
As a consequence, the pullback diagram \eqref{pullback units} exhibits $GL_1(R)$ as the space of $R$-linear automorphisms of $R$. Therefore, the maximal $\infty$-subgroupoid of invertible objects which are equivalent to $R$ in $\underline{R \text{-mod}}$ yields the $C_2$-space $BGL_1(R)$. 
Since this subcategory is closed under norms and smash products, it inherits the structure of grouplike $C_2$-$\mathbb{E}_{\infty}$-space together with a $C_2$-$\mathbb{E}_{\infty}$-map:
	$$BGL_1(R) \hookrightarrow \PicR$$	
	The $C_2$-$\mathbb{E}_{\infty}$-structure of both these spaces induces $C_2$-spectra $\underline{pic}(R)$ and $gl_1(R)$ such that the zeroth spaces corresponding to $\underline{pic}(R)$ and $\Sigma gl_1(R)$ are $\PicR$ and $BGL_1(R)$ respectively. Moreover, the construction of Thom spectra now yields a Thom spectrum construction for a $C_2$-map $X \to BGL_1R$. 
\end{mysubsection}

\vspace*{0.2cm}

\section{Constructing twisted $R$-algebras}\label{twralgconst}

We show that certain quotients give examples of twisted $R$-algebras. Our method involves the identification of these quotients as Thom spectra extending previous constructions in \cite{Bas17,Bas18,BSS20}. More precisely, we show that these quotients are of the form $\Th(\Omega^\sigma f)$, and hence these are twisted $R$-algebras. The underlying equivariant ring spectrum $R$ is a $C_2$-equivariant commutative ring spectrum, and our usual examples are $R=K\R$, the real $K$-theory spectrum, and $R=MU_\R$, the real complex cobordism spectrum. 

\subsection{$\KR/2$ as a twisted $K\R$-algebra}
We first construct $\KR/2$ as the Thom spectrum of a map $S^1 \to BGL_1 \KR$. The argument is analogous to \cite{Bas17}. 
\begin{mysubsection} {Units of $\KR$}
Recall that the homotopy groups of $\KR$ in the degrees $\ast(1+\sigma)$ are non-zero only when $\ast$ is even. That is,	 for each $n,\ \pi_{(n-1)+n\sigma}^{C_2}(\KR) = 0$, and
	$$\pi_{*(1+\sigma)}^{C_2} (\KR) \ \cong  \Z[\bar{\beta}^{\pm}]$$
	where $|\bar{\beta}|$ = $1+\sigma$. Moreover, $\underline{\pi}_0^{C_2} (\KR)$ is the constant Mackey functor $\uZ$. Therefore, $GL_1 (K\mathbb{R})$	can be identified as the pullback
	\begin{myeq} \label{gl1krpb}
\begin{tikzcd}[cramped]
		{GL_1 (\KR)} && {\Omega^{\infty}\KR} \\
		\\
		{\Z^{\times}} && {\pi_0^{C_2} (\KR)\cong \Z.}
		\arrow[from=1-1, to=1-3]
		\arrow[from=1-1, to=3-1]
		\arrow[from=1-3, to=3-3]
		\arrow[from=3-1, to=3-3]
	\end{tikzcd}
\end{myeq}	
\end{mysubsection}
	
	For a based $C_2$-space $X$, a $C_2$ map $S^{1} \wedge X \xrightarrow{f} BGL_1 (\KR)$, corresponds to a unit $\hat{f} \in \KR ^{0} (X)$, via 
\[[S^{1} \wedge X, BGL_1 \KR]^{C_2}  \cong  [X, GL_1 \KR]^{C_2}	 \cong \KR ^{0} (X)^{\times}.\] 
The unit $\hat{f}$ induces a $\KR$-module map $K\R \wedge X_{+} \xrightarrow{u_{f}} \KR$. 	

\begin{prop} \label{Thom colimit form}
		Suppose that $X$ is a $C_2$-space and $f$ be a $C_2$-map $f:S^{1} \wedge X \rightarrow BGL_1 (\KR) \hookrightarrow \underline{Pic}(\KR)$. Then, the $C_2$-Thom spectrum is equivalent to the homotopy pushout of $(\KR \xleftarrow{\pi_1} \KR \wedge X_{+} \xrightarrow{u_{f}} \KR)$ in the category of $\KR$-modules.
	
\end{prop}

\begin{proof}
	As $S^{1} \wedge X$ is $C_2$-colimit \cite[Section 2.3]{horev} of the constant diagram $(* \leftarrow X \rightarrow *)$ in $\underline{\Top}^{C_2}$, and $C_2$-Thom spectrum functor strongly preserves $C_2$-colimits \cite[Theorem 5.0.2]{nonabelianpoincare}, we get a homotopy pushout  of $K\mathbb{R}$-modules	
	\[\begin{tikzcd}[cramped]
		{\Th^{C_2}(\tilde{f})\ \simeq K\mathbb{R} \wedge X_{+}} &&& {K\mathbb{R}} \\
		\\
		{K\mathbb{R}} &&& {\Th^{C_2}(f)}
		\arrow[from=1-1, to=1-4]
		\arrow[from=1-1, to=3-1]
		\arrow[from=1-4, to=3-4]
		\arrow[from=3-1, to=3-4]
	\end{tikzcd}\]
where $\tilde{f}$ is the composition $X \xhookrightarrow{incl} S^{1} \wedge X \xrightarrow{f} \underline{Pic}(K\mathbb{R})$. As the $GL_1 (K\mathbb{R})$-bundle over $\Sigma X$, restricts to trivial bundles over two copies of cones over X, and on the intersection $X$ the bundle is given by $\tilde{f}$; the two maps from $K\mathbb{R} \wedge X_{+}$ to $K\mathbb{R}$ are given by projection and $u_{f}$.	
\end{proof}

\vspace*{0.2cm}

We now put $X=S^1$ and compute the Thom spectrum to obtain $\KR/2$ as a $\KR$-module. 
\begin{prop} \label{kuR/2 as kuR-module}
	Let $f: S^{1} \rightarrow BGL_1 (K\mathbb{R}) \hookrightarrow \underline{Pic}(K\mathbb{R})$ be the map corresponding to the unit $(-1) \in \pi_{0} ^{C_2}(GL_1K\mathbb{R}) = \pi_{1} ^{C_2}(BGL_1 (K\mathbb{R}))$. Then $\Th^{K\R}(f) \simeq K\mathbb{R}/2$ as $K\mathbb{R}$-modules.
\end{prop}

\begin{proof}
	The pushout square of $K\mathbb{R}$-module spectra following Proposition \ref{Thom colimit form}, induces a cofiber sequence of $K\mathbb{R}$-modules 	
	$$K\mathbb{R} \vee K\mathbb{R} \xrightarrow
	{\begin{pmatrix}
			1 & 1\\
			1 & -1
	\end{pmatrix}}
	K\mathbb{R} \vee K\mathbb{R} \longrightarrow \Th^{K\R}(f)$$
	which can be rewritten as $K\mathbb{R} \xrightarrow{2} K\mathbb{R} \rightarrow \Th^{K\R}(f)$.
\end{proof}

\vspace*{0.2cm}

We now prove that the map $f$ of Proposition \ref{kuR/2 as kuR-module} is $\simeq \Omega^{\sigma} g$. For this we first identify $S^1$ as the $\sigma$-fold loops on $\C P^\infty$ on which $C_2$ acts by complex conjugation. We call this space $\C P^\infty_\tau$. 

\begin{mysubsection}{Cellular filtration on $\C P^\infty_\tau$} \label{CP_infity cell}
The usual CW complex structure on $\C P^\infty$ induces an equivariant cell-complex structure on $\C P^\infty_\tau$ on which the $(2n)$-cell corresponds to a disk in $\C^n$ on which $C_2$ acts by complex conjugation. Thus we have a filtration	
\begin{myeq} \label{cellcp}
Z_1 = S^{\rho} \rightarrow Z_2 \rightarrow Z_3 \rightarrow \dots
\end{myeq}
	 where $Z_n$ is complex lines through origin in $\mathbb{C}^{n+1} (\tau)$ and the attaching map is given by the pushout square of $C_2$-spaces
	\begin{myeq} \label{attcp}
\begin{tikzcd}[cramped]
		{S(\mathbb{C}^{n} (\tau))} && {D(\mathbb{C}^{n} (\tau))} \\
		{Z_{n-1}} && {Z_n}
		\arrow[hook, from=1-1, to=1-3]
		\arrow[from=1-1, to=2-1]
		\arrow[from=1-3, to=2-3]
		\arrow[from=2-1, to=2-3]
	\end{tikzcd}
\end{myeq}
Identifying $S(\mathbb{C}^{n} (\tau))$ as $S^{n\rho -1}$ such that $(1,0,0,...,0)$ corresponds to the point at infinity, we obtain the cofiber sequence	
	\[\begin{tikzcd}[cramped]
		{S^{n\rho-1}} & {Z_{n-1}} & {Z_n.}
		\arrow[from=1-1, to=1-2]
		\arrow[from=1-2, to=1-3]
	\end{tikzcd}\]
	Now, the fixed point inclusion map 
\[\R P^{\infty} \cong ({\C P^{\infty}_{\tau}})^{C_2} \hookrightarrow {\C P^{\infty}_{\tau}}\]
 is the map $BC_2 \rightarrow BS^1$. As by \ref{fixed pts of loop}
\[(\Omega^{\sigma}{\C P^{\infty}_{\tau}})^{C_2} \simeq \hofib(({\C P^{\infty}_{\tau}})^{C_2} \hookrightarrow {\C P^{\infty}_{\tau}}),\]
 we get 
\[(\Omega^{\sigma}{\C P^{\infty}_{\tau}})^{C_2} \simeq (S^1 / C_2) \cong S^1,\]
 inducing a $C_2$-equivalence $S^1 \xlongrightarrow{\sim} \Omega^{\sigma}{\C P^{\infty}_{\tau}}$ via the adjoint of $Z_1 = S^{\rho} \rightarrow {\C P^{\infty}_{\tau}}$. 	
\end{mysubsection}

\begin{prop} \label{kuR/2 as twisted monoid}
	$K\mathbb{R}/2$ is a twisted algebra in the category of $K\mathbb{R}$-modules, that is, $K\mathbb{R}/2 \in \Alg_{\E_\sigma}(\KRmod)$.
\end{prop}

\begin{proof}
	A map $f: X \longrightarrow BGL_1 (K\mathbb{R})$, for a grouplike $\mathbb{A}_{\infty}$-space with involution (twisted monoid) $X$, can be realised as a twisted monoid map $\Omega^{\sigma}\zeta \simeq f$ for some $\zeta: B^{\sigma}X \longrightarrow  B^{\rho}GL_1 (K\mathbb{R})$, if the map 
	$$\rho \circ \Sigma^{\sigma} f : \Sigma^{\sigma} X \longrightarrow B^{\rho}GL_1(K\mathbb{R})$$
	\noindent extensds over $B^{\sigma}X$
		\[\begin{tikzcd}[cramped]
		{\Sigma^{\sigma}X} &&& {\Sigma^{\sigma}BGL_1(K\mathbb{R})} \\
		\\
		{B^{\sigma}X} &&& {B^{\rho}GL_1(K\mathbb{R})}
		\arrow["{\Sigma^{\sigma}f}", from=1-1, to=1-4]
		\arrow["\chi"', from=1-1, to=3-1]
		\arrow["\chi"', from=1-4, to=3-4]
		\arrow["\zeta", from=3-1, to=3-4]
	\end{tikzcd}\]
	\noindent where $\chi$ is the adjoint to $C_2$-equivalances $X \xrightarrow{\sim} \Omega^{\sigma}B^{\sigma}X$ and $GL_1(K\mathbb{R}) \xrightarrow{\sim} \Omega^{\sigma}B^{\sigma}GL_1(K\mathbb{R})$.
	In that case the resulting Thom spectrum is a twisted $\KR$-algebra by Proposition \ref{twalgth}, therefore  $\Th^{K\R}(f) \simeq \Th^{K\R}(\Omega^{\sigma}\zeta) \in \Alg_{\E_\sigma}(\KRmod)$.
	
	Following Proposition \ref{kuR/2 as kuR-module}, we have the diagram:
	\[\begin{tikzcd}[cramped]
		{\Sigma^{\sigma}S^1} &&& {\Sigma^{\sigma}BGL_1(K\mathbb{R})} \\
		\\
		{B^{\sigma}S^1 \simeq {\C P^{\infty}_{\tau}}} &&& {B^{\rho}GL_1(K\mathbb{R})}
		\arrow["{\Sigma^{\sigma}(-1)}", from=1-1, to=1-4]
		\arrow["\chi"', from=1-1, to=3-1]
		\arrow["\chi"', from=1-4, to=3-4]
		\arrow["{want\ a\ \zeta}", dashed, from=3-1, to=3-4]
	\end{tikzcd}\]
 The $C_2$-cell structure of ${\C P^{\infty}_{\tau}}$ from \eqref{cellcp},\eqref{attcp} implies that the obstructions to extending the map $\Sigma^{\sigma}(-1)$ to ${\C P^{\infty}_{\tau}}$ lies in the groups
	\begin{equation*}
		\begin{aligned}
		 [S^{(n-1)+n\sigma},B^{\rho}GL_1(K\mathbb{R})]^{C_2}
			& \cong \pi_{(n-2)+(n-1)\sigma}^{C_2}(GL_1(K\mathbb{R})) \\
			& \cong \pi_{(n-2)+(n-1)\sigma}^{C_2}(K\mathbb{R}) \\ 
			& \cong 0.
		\end{aligned}
	\end{equation*} 
The proof is now complete by Proposition \ref{twalgth}.
\end{proof}

\vspace*{0.2cm}

\subsection{$H\uZ$ as a twisted $\kR$-algebra}
We now extend the construction of twisted algebras to other examples.   From \cite[Section 6]{Dug05}, we can identify the constant Mackey functor $\HZ$ as a quotient of $\kR$ as follows
     $$\Sigma^\rho \kR \xrightarrow{\beta} \kR \rightarrow \HZ.$$
     \begin{prop} 
     	The $C_2$-Thom spectrum of the map
\[S^{\rho +1} \xrightarrow{\fbeta} BGL_1(\kR) \hookrightarrow \underline{Pic}(K\mathbb{R}),\]
 corresponding to $\beta \in \pi _{\rho} ^{C_2}(\kR)$, is weakly equivalent to $\HZ$ as $\kR$-module.  
     \end{prop}
     
     \begin{proof}
     	By invoking Proposition \ref{Thom colimit form}, we get the cofiber sequence
   	    	$$k\mathbb{R} \vee \Sigma^{\rho} k\mathbb{R} \xrightarrow
     	{\begin{pmatrix}
     			1 & 1 \\
     			\beta & 0
     	\end{pmatrix}}
     	k\mathbb{R} \vee k\mathbb{R} \longrightarrow \Th^{k\R}(\fbeta)$$
     	 which can be rewitten as $\Sigma^\rho \kR \xrightarrow{\beta} \kR \rightarrow \Th^{k\R}(\fbeta)$. 
     \end{proof}
     
\vspace*{0.2cm}
     
\begin{mysubsection}{Cellular filtration on $\HP^\infty$}     
We now write $S^{\rho+1}\simeq \Omega^\sigma \HP^\infty_\tau$, and recall the cell structure of $\HP^\infty_\tau$ from  \cite[Proposition 4.2]{HS20}.
 Let $\HP^\infty _{\tau}$ be  the infinite dimensional quaternionic projective space, with $C_2$-action as conjugation by $i$. Explicitly, 
\[\tau  [z_0 : z_1 : z_2 : ...] = [i z_0 i^{-1} : i z_1 i^{-1} : i z_2 i^{-1} : ...],\]
 where $i (a+ bi + cj + dk) i^{-1}= a +bi -cj -dk$. This constructs a $C_2$-cell structure on $\HP^\infty _{\tau}$ as follows     	
     	\begin{myeq} \label{cellhp} 
Y_1 = S^{2\rho} \rightarrow Y_2 \rightarrow Y_3 \rightarrow \dots
     	\end{myeq}
     	with the following pushout square of $C_2$-spaces
     	\begin{myeq}\label{atthp}
\begin{tikzcd}[cramped]
     		{S^{2n\rho -1}} & {Y_{n-1}} \\
     		{D^{2n\rho}} & {Y_n.}
     		\arrow[from=1-1, to=1-2]
     		\arrow[from=1-1, to=2-1]
     		\arrow[from=1-2, to=2-2]
     		\arrow[from=2-1, to=2-2]
     	\end{tikzcd}
\end{myeq}
\end{mysubsection}

     This allows us to identify $\HZ$ as a twisted monoid.
     
     \begin{prop} \label{hzmodkr}
     	The map $\fbeta: S^{\rho +1} \rightarrow BGL_1(\kR) \hookrightarrow \underline{Pic}(K\mathbb{R})$ can be identified as a $\sigma$-fold loop map of $\HP^\infty _{\tau} \rightarrow B^{\rho}GL_1(\kR)$, establishing $\HZ \in \Alg_{\E_\sigma} (\kR)$.
     \end{prop}
     
     \begin{proof}
     	Similar to Proposition \ref{kuR/2 as twisted monoid}, it is the obstruction problem related to the diagram of $C_2$-spaces     	
     	\[\begin{tikzcd}[cramped]
     		{\Sigma^{\sigma}S^{\rho +1}} &&& {\Sigma^{\sigma}BGL_1(K\mathbb{R})} \\
     		\\
     		{B^{\sigma}S^{\rho+1} \simeq {\HP^\infty_{\tau}}} &&& {B^{\rho}GL_1(K\mathbb{R})}
     		\arrow["{{\Sigma^{\sigma}\fbeta}}", from=1-1, to=1-4]
     		\arrow["\chi"', from=1-1, to=3-1]
     		\arrow["\chi"', from=1-4, to=3-4]
     		\arrow["{{want\ a\ \phi}}", dashed, from=3-1, to=3-4]
     	\end{tikzcd}\]
     	     	As a consequence of \eqref{cellhp}, \eqref{atthp}, we get that the obstructions lie in the groups
     	     	\begin{equation*}
     		\begin{aligned}
     			[S^{2n \rho -1},B^{\rho}GL_1(k\mathbb{R})]^{C_2}
     			& \cong \pi_{(2n-1) \rho -1}^{C_2}(GL_1(k\mathbb{R})) \\
     			& \cong \pi_{(2n-1) \rho -1}^{C_2}(k\mathbb{R}) \\ 
     			& \cong 0.
     		\end{aligned}
     	\end{equation*}
     	This completes the proof.
     \end{proof}

\vspace*{0.2cm}

\subsection{Finite quotients as twisted algebras}
	This section focuses on the case where $R$ is a $C_2$-$\mathbb{E}_{\infty}$-ring spectrum that is $\textit{even}$ in the sense that $\underline{\pi}_{k\rho -1} = 0$ for $k \in \Z$. Our objective is to construct quotients by finitely many homotopy classes as Thom spectra. We shall establish for classes $u_i \in \pi^{C_2}_{i \rho}(R)$ for $0\leq i \leq n-1$, the quotient $R / (u_0, u_1, \dots , u_{n-1})$ has a twisted $R$-algebra structure via the Thom spectrum associated to a twisted monoid map $U(n) \rightarrow BGL_1(R)$.

	\begin{mysubsection}{Unitary groups as twisted monoids} \label{untw}
		The twisted monoid structure on $U(n)$, denoted as $U(n)_{\tau}$, has underlying multiplication as matrix multiplication and $C_2$-action is via transposition of matrices : $\tau (A) = A^{T}$. Let $Gr_{n} (\mathbb{C}^{\infty})_{\tau}$ be the $C_2$-space whose underlying non-equivariant space is $Gr_{n} (\mathbb{C}^{\infty})$ and $C_2$-action is given via complex conjugation on $n$-planes; explicitly $\tau(W) = \bar{W}$ for $W \in Gr_{n} (\mathbb{C}^{\infty})$. We show that $Gr_{n} (\mathbb{C}^{\infty})_{\tau}$ is a $\sigma$-fold delooping of $U(n)_{\tau}$.
	\begin{prop} \label{deloop of unitary}
As $C_2$-spaces, $\Omega^{\sigma} \gran \simeq U(n)_{\tau}$.
	\end{prop}
	
	\begin{proof}
		Since $\Omega Gr_n(\C^\infty) \simeq U(n)$, the equivalence above holds in the non-equivariant setting, and so, it suffices to establish the equivalence on $C_2$-fixed points. Therefore, it boils down to checking 
\begin{myeq} \label{unfix} 
U(n)_{\tau} ^{C_2} \simeq \mbox{hofib} \,( \gran ^{C_2} \hookrightarrow \gran).
\end{myeq} 
 We observe that the $C_2$-fixed points of $U(n)_{\tau}$ is the collection of $\textit{unitary symmetric matrices}$, explicitly $\{ A \in GL_n(\mathbb{C}) | \, A^{T}=A, A^* A=A A^*=I\}$. On the other hand, the homotopy fiber is homotopic to the following pullback of complex Stiefel manifold via the inclusion of $C_2$-fixed points		
		\[\begin{tikzcd}[cramped]
			{\tilde{V}_n (\mathbb{C}^{\infty})} && {V_n (\mathbb{C}^{\infty})} \\
			{Gr_n (\mathbb{C} ^{\infty})_{\tau} ^{C_2}} && {Gr_n (\mathbb{C}^{\infty}).}
			\arrow[from=1-1, to=1-3]
			\arrow[from=1-1, to=2-1]
			\arrow[two heads, from=1-3, to=2-3]
			\arrow[hook, from=2-1, to=2-3]
		\end{tikzcd}\]
		More explicitly, 
\[\tilde{V}_n (\mathbb{C}^{\infty}) \cong \{ (v_1, v_2, ..., v_n) \in (\mathbb{C}^{\infty} )^{n} \, | \, \langle v_i,v_j \rangle = \delta_{ij}, \,  span(v_1, v_2, ..., v_n) = span(\bar{v_1}, \bar{v_2}, ..., \bar{v_n}) \}.\]
 Taking into account that 
\[\exists \, ! \, A \in U(n)_{\tau} ^{C_2} \mbox{ such that } A \, [v_1\, v_2\, ...\, v_n] = [\bar{v_1} \, \bar{v_2} \, ...\, \bar{v_n}] \mbox{ for } (v_1, v_2, ..., v_n) \in \tilde{V}_n (\mathbb{C}^{\infty}),\]
 we obtain a map $\phi: \, \tilde{V}_n (\mathbb{C}^{\infty}) \rightarrow U(n)_{\tau} ^{C_2}$ via $(v_1, v_2, ..., v_n) \mapsto A$. Now, the proof of \eqref{unfix} is complete by Lemma \ref{unfib}.
			\end{proof}
	
\vspace*{0.2cm}
	
	\begin{lem}\label{unfib}
	The map	$\phi: \, \tilde{V}_n (\mathbb{C}^{\infty}) \rightarrow U(n)_{\tau} ^{C_2}$ is a $\textit{fiber bundle}$ with fiber $V_n (\mathbb{R}^{\infty})$.
	\end{lem}
	
	\begin{proof}
		Since unitary symmetric matrices have eigenspaces spanned by real eigenvectors, by the spectral theorem for normal matrices, they can be diagonalised by real orthogonal matrices.  Explicitly for $A \in U(n)_{\tau} ^{C_2}$, there exists $O \in O(n)$ and $D = \text{diag}(\lambda_1, \lambda_2, ..., \lambda_n)$ such that $A= ODO^{T}$ with $\lambda_i \in S^1$ being the eigenvalues. Noting that, if $(v_1, v_2, ..., v_n) \in \tilde{V}_n (\mathbb{C}^{\infty})$ is in the inverse image of $A \in U(n)_{\tau} ^{C_2}$ under $\phi$, then by construction 
		\begin{align*}
			& A \, [v_1\, v_2\, ...\, v_n] = [\bar{v_1} \, \bar{v_2} \, ...\, \bar{v_n}] \\
			\implies & DO^{T}\, [v_1\, v_2\, ...\, v_n] = O^{T}\, [\bar{v_1} \, \bar{v_2} \, ...\, \bar{v_n}] \; \; \; (\text{for a fixed choice of } D \text{ and }  O) \\
			\implies & D \, [w_1\, w_2\, ...\, w_n] = [\bar{w_1} \, \bar{w_2} \, ...\, \bar{w_n}] \; \text{; where } [w_1\, w_2\, ...\, w_n] = O^{T}\, [v_1\, v_2\, ...\, v_n]\\
			\implies & e^{i \theta_j /2} w_j = \overline{e^{i \theta_j /2} w_j} \; \text{; where } \lambda_j= e^{i \theta_j} \text{ for } \theta_j \in [0,2\pi) \\
			\implies & (e^{i \theta_j /2} w_j)_{j=1} ^{n} \in V_n (\mathbb{R}^{\infty})
		\end{align*}
				Therefore, each of the fiber is homeomorphic to $V_n (\mathbb{R}^{\infty})$. As a result, obtaining continuous assignment of such a $D$ and $O$ on small enough open sets establishes the lemma. One makes the following observations that the orbit space under conjugation action of $O(n)$ on $\textit{unitary symmetric matrices}$ $U(n)^{C_2} _{\tau}$ is homeomorphic to the $\textit{n-th symmetric power of } S^1$, $SP^n(S^1)$. Moreover, $SP^n(S^1)$ fits into the fiber bundle
 $$\Delta ^{n-1} \rightarrow SP^n (S^1) \rightarrow S^1.$$ 
				Thus, if we write $\D^1_+=S^1\setminus \{1\}$ and $\D^1_{-}= S^1 \setminus \{-1\}$, 
 $$ U(n)^{C_2} _{\tau} \xlongrightarrow{q} U(n)^{C_2} _{\tau} / O(n) \cong SP^n (S^1) \cong \Delta^{n-1} \times \mathbb{D}^1_+ \cup \Delta^{n-1} \times \mathbb{D}^1_-.$$
Note that $q^{-1}(\Delta^{n-1}\times \D^1_{\pm})$ are open. Therefore, given any $A \in  U(n)^{C_2} _{\tau} $, we can choose an open neighbourhood $U_A$ of $A$ for which the eigenvalues $\lambda_j = e^{i\theta_j}$ where $\theta_j : U_A \to \R$ are continuous for $1\leq j \leq n$. This is because the choice of the $\theta_j$ can be made over $q^{-1}(\Delta^{n-1}\times \D^1_{\pm})$.
Concentrating on the open set $X := q^{-1} (\Delta^{n-1} \times \mathbb{D}^1)$, we get a finite stratification 
$$\phi = Z_0 \subseteq Z_1 \subseteq ... \subseteq Z_n = X$$ 
induced under $q^{-1}$ from the skeletal stratification of 
$$\Delta^{n-1} \cong \{(t_0, t_1, ..., t_{n-1}) \, |  \, 0 \leq t_0 \leq t_1 \leq ... \leq t_{n-1} \leq 1\, , \Sigma t_i = 1 \}.$$
Explicitly this may be described as 
$$\phi = \Delta^{n-1}_{(0)} \subseteq \Delta^{n-1}_{(1)} \subseteq ... \subseteq \Delta^{n-1}_{(n)} = \Delta^{n-1};$$
 where
				\begin{align*}
			\Delta^{n-1}_{(1)} = & {(1/n, 1/n, ..., 1/n)} \\
			\Delta^{n-1}_{(2)} = & \{(t_0, t_1, ..., t_{n-1}) \, |  \, 0 \leq t_0 \lneqq  t_1 = ... = t_{n-1} \leq 1\, , \Sigma t_i = 1 \} \\
			& \cup \{(t_0, t_1, ..., t_{n-1}) \, |  \, 0 \leq t_0 = t_1 \lneqq t_2 =... = t_{n-1} \leq 1\, , \Sigma t_i = 1 \} \\ 
			& \, \cup \\ & \;... \\ 
			& \cup \{(t_0, t_1, ..., t_{n-1}) \, |  \, 0 \leq t_0 = t_1 = ...= t_{n-2} \lneqq t_{n-1} \leq 1\, , \Sigma t_i = 1 \} \\
			 \text{and so on.}
		\end{align*}
		The attaching of stratum are done following the pushout square
		\[\begin{tikzcd}[cramped]
			{\Delta^{n-1}_{(i)} \times \mathbb{D}^1 \times ( O(n) / F_{i+1})} && {\Delta^{n-1}_{(i+1)} \times \mathbb{D}^1 \times ( O(n) / F_i)} \\
			{\Delta^{n-1}_{(i)} \times \mathbb{D}^1 \times ( O(n) / F_i)} \\
			{Z_i} && {Z_{i+1}}
			\arrow[hook, from=1-1, to=1-3]
			\arrow["{\alpha_{i+1,i}}"', two heads, from=1-1, to=2-1]
			\arrow["q", from=1-3, to=3-3]
			\arrow["q"', from=2-1, to=3-1]
			\arrow[from=3-1, to=3-3]
		\end{tikzcd}\]
		where $F_i\subset O(n)$ are isotropy groups of the matrices in $Z_i - Z_{i-1}$ under conjugation, resulting in the sequence of subgroups 
$$O(n) = F_1 > F_2 > ... > F_n = \{diag(\pm 1, \pm 1, ..., \pm 1)\}.$$
 Now, the fact that all the $\alpha_{i,j}: O(n)/F_{i+j} \to O(n)/F_i$ are all fiber bundles; one can ensure (possibly by taking finite intersections) that $A \in Z_i - Z_{i-1}$ has an open neighbourhood $V_A$ in $Z_n$ such that $\pi_3\big(q^{-1}(V_A \cap Z_{m})\big) \subset O(n)/F(m)$ is a trivializing open set for every fiber bundle $\alpha_{m,j}$, $m\geq i$. Going to a smaller $V_A$ if necessary, we may assume that $V_A \cap Z_i$ is homeomorphic to an Euclidean disk (as a subset of the manifold with corners $\big(\Delta^{n-1}_{(i)}\setminus \Delta^{n-1}_{(i-1)}\big) \times \D^1 \times O(n)/F_i $). The choice of $O$, can now be made over $V_A$ by starting with a section on $V_A\cap Z_i \to O(n)$ of $O(n) \to O(n)/F_i$, and then extending it to all of $V_A=V_A\cap Z_n$ using the local triviality of each $\alpha_{m,j}$ over $\pi_3 \big(q^{-1}(V_A \cap Z_m)\big)$ for $m\geq i$. This gives a continuous choice of $O$ on $V_A$.  Therefore, by taking the intersection of $V_A$ and $U_A$, we get a continuous choice of $O$ and $D$ on an open set containing $A$, completing the proof.
	\end{proof}

\vspace*{0.2cm}
	\end{mysubsection}

\begin{mysubsection}{Assumptions on $R$}
We construct quotients of the $\E_\infty$-ring spectrum $R$ as twisted $R$-algebras under the two hypotheses given below. \\
(1) The spectrum $R$ is \textit{even} \cite[Definition 3.1]{hillmeier}. This means that $\upi_{n\rho-1}(R)=0$ for all integers $n$. \\
(2) The spectrum $R$ is \textit{cofree}\cite[Definition 17.3.67]{Hil20}. This means that the map $R \to F({EC_2}_+,R)$ is an equivalence in $Sp^{C_2}$. 

We shall use the following property of modules over cofree ring spectra. 
\begin{prop}\label{modcofree}
Suppose that $R$ is a cofree ring spectrum. Then, any compact $R$-module $M$ is also cofree, that is, $M \to F({EC_2}_+,M)$ is an equivalence. 
\end{prop} 

\begin{proof}
We apply the Tate diagram \cite[Diagram (C)]{GM95a}
\[\xymatrix{ {EC_2}_+ \wedge E \ar[r] \ar[d]^{\simeq} & E \ar[r] \ar[d] & \widetilde{EC_2}\wedge E \ar[d]\\ 
{EC_2}_+\wedge F({EC_2}_+,E) \ar[r] & F({EC_2}_+,E) \ar[r] & \widetilde{EC_2} \wedge F({EC_2}_+,E),}\]
for $E=R$ and $E=M$. As the rows are cofiber sequences and the left vertical arrow is an equivalence, it follows that $E$ is cofree if and only if $\widetilde{EC_2}\wedge E \to \widetilde{EC_2} \wedge F({EC_2}_+, E)$ is an equivalence. Assuming this is true for $R$, we may write 
\[ \widetilde{EC_2} \wedge M \simeq \big(\widetilde{EC_2}\wedge R\big) \wedge_R M, \]
for an $R$-module $M$, and 
\begin{align*}
\widetilde{EC_2} \wedge F({EC_2}_+, M) &\simeq \widetilde{EC_2} \wedge F_{R\mbox{\tiny -mod}}({EC_2}_+\wedge R, M) \\
&\simeq \widetilde{EC_2}\wedge \big(F_{R\mbox{\tiny -mod}}({EC_2}_+\wedge R)\wedge_R M\big) \\
&\simeq \big(\widetilde{EC_2} \wedge F({EC_2}_+,R)\big)\wedge_R M, 
\end{align*}
for a compact $R$-module $M$. Therefore, it follows that $M$ is cofree.
\end{proof}

\vspace*{0.2cm}
\end{mysubsection}
	
		As $R$ is an $\textit{even} \; C_2$-ring spectrum, by \cite[Lemma 3.3]{hillmeier} $R$ is real orientable. Therefore, \cite[Theorem 2.10]{hukriz} implies that 
\[R^{\bigstar} _{C_2} (\C P^{n-1}) \cong \pi^{C_2} _{-\bigstar} (R) [x] / (x^n)\]
 as $\pi^{C_2} _{\bigstar} (R)$-modules for a suitable generator $x \in R^{\rho} _{C_2} (\C P^{n-1})$. We also have a $C_2$-map (with the $C_2$-action  on $\C P^n$ by complex conjugation, which is denoted by $\C P^n_\tau$)
$$\Sigma \C P^{n-1}_\tau\longrightarrow U(n)_{\tau}$$ 
given by 
\begin{myeq}\label{cpformun}
(z, \begin{bmatrix}
			x \\
			\zeta
		\end{bmatrix})  \mapsto 
		\begin{bmatrix}
			(z-1)x \bar{x} +1  &  (z-1)x[\zeta]^{*}  \\
			(z-1)\bar{x} [\zeta]  &  I + (z-1) [\zeta] [\zeta]^{*}
		\end{bmatrix}
\end{myeq} 
		where $x \in \C$ and $[\zeta] \in \C ^{n-1}$ such that $|x|^2 + |\zeta|^2 =1$.
	Now, composing with an equivariant map $U(n)_{\tau} \rightarrow BGL_1 (R)$, we obtain a homotopy class
		$$\Sigma \C P^{n-1}_\tau \rightarrow U(n)_{\tau} \rightarrow BGL_1 (R)$$
		whose adjoint determines an element in cohomology:
		$u \in [\C P^{n-1}_\tau, \Omega BGL_1(R)] \cong R^0 _{C_2} (\C P^{n-1}_\tau) ^{\times}$. Under the above mentioned $\pi^{C_2} _{\star} (R)$-module isomorphism this cohomology class take the following form:
		\begin{myeq} \label{identifyelement}
	u(\C P^{n-1}) = 1+u_0 + u_1 x + u_2 x^2 + ... + u_{n-1} x^{n-1} \mbox{ where } u_i \in \pi_{i \rho} (R).
\end{myeq}
 The following proposition helps us understand $C_2$-Thom spectra for maps out of $U(n)_{\tau}$.
		
\begin{prop} \label{cofiberofthom}
		There is a cofiber sequence of $R$-modules 
		$$\Sigma ^{m\rho} \Th^{R} (U(m)_{\tau}) \xrightarrow{u_m}  \Th^{R} (U(m)_{\tau}) \rightarrow  \Th^{R} (U(m+1)_{\tau})$$
for $1 \leq m < n$.
	\end{prop}
	
	\begin{proof}
		From the construction of the map $\Sigma \C P^{n-1} _{\tau} \rightarrow U(n)_{\tau}$ \eqref{cpformun}, we have the following commutative square of $C_2$-spaces over $BGL_1(R)$
\[\xymatrix{ \Sigma \C P^{m-1} _{\tau} \ar[d] \ar[r] & \Sigma \C P^m _{\tau} \ar[d] \\
U(m)_\tau \ar[r] & U(m+1)_\tau \ar[r] & BGL_1 R}\]
and hence a commutative square on Thom spectra 
\begin{myeq}\label{commthum}
\xymatrix{ \Th^{R}(\Sigma \C P^{m-1} _{\tau}) \ar[d] \ar[r] & \Th^{R}(\Sigma \C P^m _{\tau}) \ar[d] \\
\Th^{R}(U(m)_\tau) \ar[r] & \Th^{R}(U(m+1)_\tau). }
\end{myeq}
The Thom spectrum $ \Th^{R}(\Sigma \C P^m _{\tau})$ fits into a cofiber
\begin{myeq}\label{thcpm}
R \wedge \C P^m_\tau \xrightarrow{u(\C P^m)-1} R \wedge \C P^m_\tau \to  \Th^{R}(\Sigma \C P^m _{\tau}) 
\end{myeq}
as in Proposition \ref{Thom colimit form}. Now we have the homotopy pushout 
\[\xymatrix{\Sigma S^{m\rho -1} \ar[r] \ar[d] & \ast \ar[d] \\ 
\Sigma \C P^{m-1} _{\tau}  \ar[r] & \Sigma \C P^m _{\tau}  }\]
which implies the cofiber 
\[R\wedge S^{m\rho} \to  \Th^{R}(\Sigma \C P^{m-1} _{\tau}) \to \Th^{R}(\Sigma \C P^m _{\tau}).\]		
The formula \eqref{thcpm} and the formula of $u(\C P^m)$ in \eqref{identifyelement} directly implies that the left arrow above equals $u_m$ under the isomorphism 
\[ [R\wedge S^{m\rho}, \Th^{R}(\Sigma \C P^{m-1} _{\tau})]_{R\mbox{\tiny -mod}} \cong  [ S^{m\rho}, \Th^{R}(\Sigma \C P^{m-1} _{\tau})] = \pi_{m\rho}^{C_2} \Big(\Th^{R}(\Sigma \C P^{m-1} _{\tau})\Big).\]		
Now we return to the diagram \eqref{commthum}, where we now deduce from the above cofiber sequence that under the map 
\[\Th^{R}(U(m)_\tau) \to \Th^{R}(U(m+1)_\tau),\]
the class $u_m$ goes to $0$. 
		Thus, we arrive at the sequence 
		$$\Sigma ^{m\rho} \Th^{R} (U(m)_{\tau}) \xrightarrow{u_m}  \Th^{R} (U(m)_{\tau}) \rightarrow  \Th^{R} (U(m+1)_{\tau})$$
		where the first map is defined using the $R$-module structure on the Thom spectra, and the composite is null. From \cite[Proposition 5.2]{BSS20}, we have that it is a cofiber sequence on the non-equivariant spectra. We claim that it follows that this induces an equivalence $\mbox{cof}(u_m) \simeq \Th^{R} (U(m+1)_{\tau})$ of $R$-modules.

As the unitary groups $U(m)_\tau$ are finite complexes, it follows that $\Th^{R}(U(m)_\tau)$ are compact $R$-modules, which are cofree by Proposition \ref{modcofree}. Hence, it follows from \cite[Proposition 1.1]{GM95a} that  $\mbox{cof}(u_m) \simeq \Th^{R} (U(m+1)_{\tau})$ is an equivalence. 
	\end{proof}
	
	\vspace*{0.2cm}

In the $m=0$ case, one may also compute $\Th^R(f)$ for $f:S^1 \to BGL_1R$ directly as $R/(u_0)$ by using an argument analogous to the proof of Proposition \ref{Thom colimit form}. 

\vspace*{0.2cm}

	\begin{mysubsection}{Slices of even ring spectra} \label{sliceofspectra}
		The slices of an $\textit{even}$ ring spectrum $R$ are described in \cite[Proposition 4.20, Lemma 4.23]{HHR21}, which  states that the odd slices of such spectrum vanish. If $\underline{M}$ is a $C_2$-Mackey functor then, we denote $P^0 \underline{M}$ as the maximal quotient Mackey functor of $\underline{M}$ having injective restriction map. From \cite[Proposition 2.13]{hillmeier}, we have that the even slices are determined as $P^{2n} _{2n} (R) \simeq \Sigma ^{n\rho} HP^{0} \underline{\pi} _{n\rho} (R)$.
	\end{mysubsection}

\vspace*{0.2cm}

	\begin{thm} \label{quottwralg}
		Given homotopy classes $u_i \in \pi _{i\rho} ^{C_2} (R)$ for $i = 0,\dots ,n-1$ such that $1+u_0\in \pi_0^{C_2}(R)$ is a unit, there exists a $\Omega ^{\sigma}$-fold twisted monoid map $U(n)_{\tau} \xrightarrow{f} BGL_1 (R)$ such that $\Th^{R} (f) \simeq  R/ (u_0, u_1, ..., u_{n-1})$ as $R$-modules.
	\end{thm}

	\begin{proof}
		As in \eqref{identifyelement} for given homotopy classes $u_i$, we get a cohomology class in $R^0 _{C_2} (\C P^{n-1}_\tau) ^{\times}$ and corresponding to that class a map: $\Sigma \C P^{n-1}_\tau \rightarrow BGL_1(R)$. Now, having a twisted monoid map essentially boils down to the extension problem		
		\[\begin{tikzcd}[cramped]
			{\Sigma^{\sigma}(\Sigma \C P^{n-1}_\tau )} && {\Sigma^{\sigma} BGL_1(R)} && {B^{\rho}GL_1(R)} \\
			{\Sigma^{\sigma} U(n)_{\tau}} \\
			{Gr_n(\mathbb{C}^{\infty})_{\tau}}
			\arrow["u", from=1-1, to=1-3]
			\arrow[from=1-1, to=2-1]
			\arrow["\chi", from=1-3, to=1-5]
			\arrow["\chi"', from=2-1, to=3-1]
			\arrow[dashed, from=3-1, to=1-5]
		\end{tikzcd}\]
			Using the slice tower for $\Sigma^{\rho} gl_1(R)$ and the description of slices from \S \ref{sliceofspectra}, we see the obstructions lie in the cohomology groups $H^{k\rho +1} _{C_2} (Gr_n (\C^\infty)_{\tau}, \Sigma^{\rho}\C P^{n-1}_\tau; P^0 \underline{\pi}_{(k-1)\rho} R)$ (as the odd slices are already zero); which vanish by Lemma \ref{cohzero}. Now Proposition \ref{cofiberofthom} completes the proof.
	\end{proof}

\vspace*{0.2cm}
	
	\begin{lem}\label{cohzero}
		Let $\underline{M}$ be a $C_2$-Mackey functor with injective restriction map. Then 
$$H^{k\rho +1} _{C_2} (Gr_n (\C^\infty)_{\tau}, \Sigma^{\rho}\C P^{n-1}_\tau; \underline{M}) \cong 0.$$
	\end{lem}
	
	\begin{proof}
		The usual cell structure on $\C P^{n-1}$ and the Schubert cell structure on $Gr_n (\C ^{\infty})$ refines to a $C_2$-cell structure (analogous to \S \ref{CP_infity cell}). For a Mackey functor $\underline{M}$ with injective restriction maps, we observe that the cellular attaching maps are $0$ when we smash with $H\uM$. Consequently, we have a splitting of $\C P^{n-1}_\tau \wedge H\underline{M}$ and $Gr_n (\C ^{\infty})_\tau \wedge H\underline{M}$ as $H\underline{M}$-modules:
		\begin{equation*}
			\C P^{n-1}_\tau \wedge H\underline{M} \simeq \bigvee _{i=1} ^{n-1} S^{i\rho} \wedge H\underline{M},
		\end{equation*}
		\begin{equation*}
			Gr_n (\C ^{\infty})_\tau \wedge H\underline{M} \simeq \bigvee _{\text{Schubert cells } \alpha} S^{|\alpha| \rho} \wedge H\underline{M},
		\end{equation*}
and the Lemma follows from \cite[Theorem 6.1]{ROC2EMspectra}
\[\pi_{-k\rho - 1}(S^{m\rho} \wedge H\uM) \cong \pi_{(m-k)\rho -1}H\uM \cong 0.\]
 For the attaching maps, using induction in case of $\C P^{n-1}_\tau$, assuming that the splitting above holds for $\C P^m_\tau$ for $m<n-1$, the cellular attachment takes the form 		
		$$S^{(n-1)\rho -1} \wedge H\underline{M} \longrightarrow \bigvee _{i=0} ^{n-2} S^{i\rho} \wedge H\underline{M} \; (\simeq \C P^{n-2}_\tau \wedge H\underline{M}),$$
		which up to homotopy lies in the group $\bigoplus _{i=1} ^{n-1} \pi ^{C_2} _{i\rho -1} (H\underline{M} )=0$ by  \cite[Theorem 6.1]{ROC2EMspectra}. The same proof also works for the Grassmannian. Finally as $\Sigma^\rho \C P^{n-1}_\tau \wedge H\uM \to Gr_n(\C^\infty)_\tau \wedge H\uM$ is the inclusion of a summand, we see that $\frac{Gr_n(\C^\infty)_\tau}{\Sigma^\rho \C P^{n-1}_\tau} \wedge H\uM$ is also a wedge of copies of $S^{m\rho}\wedge H\uM$. Now the result follows as described above.
	\end{proof}

\vspace*{0.2cm}

\section{Real Topological Hochschild Homology of $C_2$-Thom Spectra} \label{thrthL}

We now switch our focus to the real topological Hochschild homology of Thom spectra. Conceptually, the existence of such a formula follows from \cite{nonabelianpoincare}. We provide an explicit expression which is useful for our later computations. 

\subsection{Real Topological Hochschild Homology} We recall the details of the construction of real topological Hochschild homology providing the background for the arguments for Thom spectra. 

\begin{mysubsection}{Description of $\Esgmop$ operad}
	Consider the (discrete) category $\Esgmop$ as follows :\\ 
	\noindent The objects are same as that of $\finC$, represented as $ (U \rightarrow O)$, where $U$ is a finite $C_2$-set and $O$ is a $C_2$-orbit. Recall that a morphism in $\finC$ is given by $(\phi, f, U)$ as depicted below:
	\[\begin{tikzcd}[cramped]
		{U_1} & {\phi ^{*} U_1} & U & {U_2} \\
		{O_1} & {O_2} & {O_2} & {O_2}
		\arrow[from=1-1, to=2-1]
		\arrow[from=1-2, to=1-1]
		\arrow[from=1-2, to=2-2]
		\arrow[hook', from=1-3, to=1-2]
		\arrow["f", from=1-3, to=1-4]
		\arrow[from=1-3, to=2-3]
		\arrow[from=1-4, to=2-4]
		\arrow["\phi", from=2-2, to=2-1]
		\arrow[equals, from=2-2, to=2-3]
		\arrow[equals, from=2-3, to=2-4]
	\end{tikzcd}\]
	A morphism in $\Esgmop$ between $U_1 \to O_1$ and $U_2 \to O_2$ is a morphism $(\phi,f,U)$ in $\finC$ together with the data  given as follows for every $x \in U_2$:
	$$ \begin{cases}
		\text{ if } \mbox{Stab}(x)= \{e\}: & \text{a morphism of } \finC \text{ along with a linear ordering on } f^{-1}(x)  \\
		\text{ if } \mbox{Stab}(x)=C_2 : & \text{a morphism of } \finC \text{ along with a linear ordering on } J \\
									& \text{ for } f^{-1}(x) = C_2 / C_2 \coprod (\amalg _{J} C_2/e ) \text{ or } f^{-1}(x) =\coprod_{J} C_2/e
	\end{cases} $$
Under the evident functor $\Esgmop \twoheadrightarrow \underline{\mbox{Fin}}_{*} ^{C_2}$, it gets the structure of a $C_2$-$\infty$-operad. Note that, it is the genuine operadic nerve \cite{Bon19} of the little $\sigma$-disks operad as the mapping spaces of the $\sigma$-framed representations are homotopically discrete.
\end{mysubsection}

\vspace*{0.2cm}

\begin{defn}
	A \textit{twisted monoid} $A$ in a $C_2$-symmetric monoidal $C_2 \text{-} \infty$-category $\underline{\mathscr{C}}$ is an $\Esgm$-algebra in $\underline{\mathscr{C}}$, that is, a $C_2$-$\infty$-operad map $A^{\otimes}: \Esgmop \rightarrow \underline{\mathscr{C}}^{\otimes}$.
\end{defn}

\vspace*{0.2cm}

\begin{ex}
	The topological twisted monoids of Definition \ref{twistedmonoid} is a strict model of twisted monoids in the $C_2$-$\infty$-category $\Topc$, where the underlying $C_2$-space is the image of $C_2 / C_2$; using the  equivariant recognition principle of \cite{GM17}, \cite{CHL24} under the condition of being grouplike. We shall focus on the twisted monoids in the category $\Spc$. In \cite{DMP21}, these were called spectra with anti-involution.  
\end{ex}

\vspace*{0.2cm}

We now  refine the notion of $C_2$-diagrams from \cite{DMP21}, \cite{DM16} by constructing a $C_2$-$\infty$-category $\tilde{\Delta}$ such that $C_2$-functors from $\tilde{\Delta}^{op}$ to another $C_2$-category are \textit{real simplicial objects}.
\begin{cons}
Consider $\tilde{\Delta} ^{op}$ as the coCartesian fibration $\tilde{\Delta} ^{op} \twoheadrightarrow \Obcat$ obtained by unstraightening the functor
	$$\begin{aligned}
		&\Obcat  \longrightarrow Cat \xrightarrow{\text{Nerve}} Cat_{\infty}\\
		&C_2/C_2   \mapsto  (\Delta^{op}) ^{C_2} \text{ defined as the fixed points of the involution }  \eqref{RSRC} \\
		&C_2 / e  \mapsto  \Delta^{op}
	\end{aligned}$$
	The ordinary category $\tilde{\Delta} ^{op}$ is described as:
	\begin{align*}
		\mbox{Objects}: & \; (C_2/C_2, [m]) \text{ and } (C_2/e, [n]) \text{ for } [m], [n] \in \Delta\\
		\mbox{Morphisms}: & \; \resizebox{13.5cm}{!}{$\text{Hom}((C_2/e, [m]),(C_2/e, [n])) = \{(id, \alpha): \alpha \in \Delta([n],[m])\} \cup \{(\omega, \beta) : \beta \in \Delta([n],[m])\}$} \\
		& \; \text{Hom}((C_2/C_2, [m]),(C_2/e, [n]))= \text{Hom}_{\Delta}([n],[m]) \\
		& \; \text{Hom}((C_2/C_2, [m]),(C_2/C_2, [n]))= \text{Hom}_{\Delta^{C_2}}([n],[m])\\
		& \text{ along with the relation } (id, \omega(\alpha))= (\tau, id) \circ (id, \alpha) \circ (\tau, id)
	\end{align*}
	
\end{cons}

\vspace*{0.2cm}

\begin{defn}
	Given a $C_2$-$\infty$-category $\underline{\mathscr{C}}$, a Real simplicial object $X$ is defined to be a $C_2$-functor $X: \widetilde{\Delta} ^{op} \rightarrow \underline{\mathscr{C}} $.
\end{defn}

\vspace*{0.2cm}

For the category of $C_2$-spaces $\underline{\Top}^{C_2}$, such a $C_2$-functor amounts to the data of a sequence of $C_2$-spaces $X_n = X_{[C_2/C_2]} ([n])$ with the relation $\alpha^* \circ \tau_k \simeq \tau_n \circ (\omega(\alpha))^*$ of Definition \ref{ex: Real simplicial object}.
Next, one defines a $C_2$-functor $\Cut^{C_2}: \tilde{\Delta} ^{op} \rightarrow \Esgmop $, which is an extension of the functor $\Cut: \Delta ^{op} \rightarrow \Ass^{\otimes}$ of \cite{NS18}. It suffices to define it on the $1$-category level as both sides are discrete. For that, we need a closer look at the morphisms of the $C_2$-fixed category $\Delta ^{C_2}$, which the next proposition sheds light into. Recall from \eqref{RSRC}, a morphism $\alpha: [n] \rightarrow [m]$ of $\Delta ^{C_2}$ has to satisfy $\alpha = \omega(\alpha)$, i.e.
\begin{myeq} \label{fixedrelation}
	\alpha(i)= \omega(\alpha) (i)= m- \alpha(n-i)  \text{ for all } i \in [n].
\end{myeq}

\begin{notation}
	For two morphisms $\alpha_i: [n_i] \rightarrow [m_i], \, i=1,2$, their concatenation is denoted as $\alpha_1 * \alpha_2$ which is a morphism from $[n_1 + n_2 + 1] $ to $[m_1+m_2 +1]$ given as $\alpha_1$ from $\{0,\dots, n_1\}$ to $\{0,\dots, m_1\}$ and $\alpha_2$ from $\{ n_1+1,\dots, n_1+n_2+1\}$ to $\{m_1+1,\dots, m_1+m_2+1\}$. 
\end{notation}

\vspace*{0.2cm}

In terms of this notation the following proposition follows readily. 
\begin{prop} \label{identifymaps}
	A morphism in $\Delta^{C_2}$ can be one of the following:
		\begin{align*}
			(1) \;&\alpha: [2n+1] \rightarrow [2m+1] \text{ is of the form } \alpha= \beta * \omega(\beta) \text{ where } \beta = \alpha _{|\{0,1,..,n\}}\\
			(2)\; &\text{Hom}_{\Delta^{C_2}} ([2n],[2m+1])= \phi \\
			(3)\; &\alpha: [2n] \rightarrow [2m] \text{ is of the form } \alpha= \beta * ([2k] \rightarrow [0]) * \omega(\beta) \text{ where } \beta = \alpha _{|\{0,1,..,i\}}, i \leq n-1\\
			(4)\; &\alpha: [2n+1] \rightarrow [2m] \text{ is of the form } \alpha= \beta * ([2k+1] \rightarrow [0]) * \omega(\beta) \\
           & \text{ where } \beta = \alpha _{|\{0,1,..,i\}}, i \leq n-1,  \text{ or } \alpha= \delta_m \circ (\beta * \omega(\beta)) \text{ where } \beta = \alpha _{|\{0,1,..,n\}}
		\end{align*}
\end{prop}

\vspace*{0.2cm}
	
Now we define the functor $\Cut^{C_2}: \tilde{\Delta} ^{op} \rightarrow \Esgmop $, which gives us the dihedral bar construction of \ref{cons dihedral bar} in this parametrized higher category setup. We first recall that the functor $\Cut: \Delta ^{op} \rightarrow \Ass^{\otimes}$ sends any finite ordered set to the set of cuts with the two trivial cuts identified.
\begin{cons} \label{DBCpara}
	We define the functor $\Cut^{C_2}: \tilde{\Delta} ^{op} \rightarrow \Esgmop$ to be the set of cuts of the underlying set, with $C_2$-action given on the cuts. On objects this is given by 
	\begin{align*}
	&	(C_2/e, [n]) \longmapsto  \;(C_2/e \times [n]) \rightarrow C_2/e, \\
	&	(C_2/C_2, [2m]) \longmapsto  \; (C_2/C_2 \amalg (C_2/e \times [m-1])) \rightarrow C_2/C_2, \\
	&	(C_2/C_2, [2m+1]) \longmapsto  \; (C_2/C_2 \amalg C_2/C_2 \amalg (C_2/e \times [m-1])) \rightarrow C_2/C_2,	
	\end{align*}
	where the first copy of $C_2/C_2$ in the second and the third assignment are the trivial cuts. The second $C_2/C_2$ summand of the third one is the cut which is fixed under $C_2$-action.
	On morphisms the functor is given by the following formulas. 
$$(C_2/e, [m]) \xrightarrow{(id,\alpha)} (C_2/e, [n])  \longmapsto 
		\begin{tikzcd}[cramped]
			{C_2/e \times [m]} & {C_2/e \times [m]} && {C_2/e \times [n]} \\
			{C_2/e} & {C_2/e} && {C_2/e}
			\arrow[from=1-1, to=2-1]
			\arrow["id"', from=1-2, to=1-1]
			\arrow["{{id \times {\tiny \Cut}(\alpha)}}", from=1-2, to=1-4]
			\arrow[from=1-2, to=2-2]
			\arrow[from=1-4, to=2-4]
			\arrow["id", from=2-2, to=2-1]
			\arrow[between={0.1}{0.9}, equals, from=2-2, to=2-4]
		\end{tikzcd}$$
		$$(C_2/e, [m]) \xrightarrow{(\tau,\alpha)} (C_2/e, [n])  \longmapsto
		\begin{tikzcd}[cramped]
			{C_2/e \times [m]} & {C_2/e \times [m]^{op}} && {C_2/e \times [n]} \\
			{C_2/e} & {C_2/e} && {C_2/e}
			\arrow[from=1-1, to=2-1]
			\arrow["\tau", from=1-2, to=1-1]
			\arrow["{id \times \widetilde{{\tiny \Cut}(\alpha)}}", from=1-2, to=1-4]
			\arrow[from=1-2, to=2-2]
			\arrow[from=1-4, to=2-4]
			\arrow["\tau", from=2-2, to=2-1]
			\arrow[between={0.1}{0.9}, equals, from=2-2, to=2-4]
		\end{tikzcd}$$
The ordered set $[m]^{op}$ is the same set as $[m]$ with the reverse order, and $\widetilde{\Cut(\alpha)}$ is the same map as $\Cut(\alpha)$ but with reverse order on each inverse images.
		$$(C_2/C_2, [2m]) \xrightarrow{\alpha} (C_2/e, [n]) \longmapsto \begin{adjustbox}{max width=10cm}
		\begin{tikzcd}[cramped]
			{C_2/C_2 \amalg C_2/e \times [m-1]} & {C_2/e \times [2m]} && {C_2/e \times [n]} \\
			{C_2/C_2} & {C_2/e} && {C_2/e}
			\arrow[from=1-1, to=2-1]
			\arrow[from=1-2, to=1-1]
			\arrow["{{id \times {\tiny \Cut}(\alpha)}}", from=1-2, to=1-4]
			\arrow[from=1-2, to=2-2]
			\arrow[from=1-4, to=2-4]
			\arrow[from=2-2, to=2-1]
			\arrow[between={0.1}{0.9}, equals, from=2-2, to=2-4]
		\end{tikzcd}
\end{adjustbox}$$
	The order in the middle term of the top sequence is induced from the leftmost term (can be thought of as dictionary order). This is also the order in the next case.
		$$ \begin{adjustbox}{max width=\linewidth}	\begin{tikzcd}[cramped]
&  (C_2/C_2, [2m+1]) \xrightarrow{\alpha} (C_2/e, [n]) \\
\\
				{C_2/C_2 \amalg C_2/C_2 \amalg (C_2/e \times [m-1])} & {C_2/e \times [2m+1]} && {C_2/e \times [n]} \\
			{C_2/C_2} & {C_2/e} && {C_2/e}
\arrow[between={0.1}{0.9}, maps to, from=1-2, to=3-2]
			\arrow[from=3-1, to=4-1]
			\arrow[from=3-2, to=3-1]
			\arrow["{{id \times {\tiny \Cut}(\alpha)}}", from=3-2, to=3-4]
			\arrow[from=3-2, to=4-2]
			\arrow[from=3-4, to=4-4]
			\arrow[from=4-2, to=4-1]
			\arrow[between={0.1}{0.9}, equals, from=4-2, to=4-4]
		\end{tikzcd} 
\end{adjustbox}$$
		There are three types of morphisms remaining, by Proposition \ref{identifymaps}. The functor $\Cut^{C_2}$ operates on these in an analogous manner. We describe one of them as below.
		\[\begin{adjustbox}{max width=\linewidth}
\begin{tikzcd}[cramped]
			& {(C_2/C_2, [2m+1]) \xrightarrow{\alpha} (C_2/C_2, [2n+1])} && \\
			\\
			\begin{array}{c} {\begin{pmatrix} C_2/C_2 \amalg C_2/C_2 \\ \amalg (C_2/e \times [m-1])	\end{pmatrix}} \end{array} & \begin{array}{c} {\begin{pmatrix} C_2/C_2 \amalg C_2/C_2 \\ \amalg (C_2/e \times [m-1])	\end{pmatrix}} \end{array} && \begin{array}{c} {\begin{pmatrix} C_2/C_2 \amalg C_2/C_2 \\ \amalg (C_2/e \times [n-1])	\end{pmatrix}} \end{array} \\
			{C_2/C_2} & {C_2/e} && {C_2/e}
			\arrow[between={0.1}{0.9}, maps to, from=1-2, to=3-2]
			\arrow[from=3-1, to=4-1]
			\arrow["id"', from=3-2, to=3-1]
			\arrow["{{\tiny \Cut}^{C_2}(\alpha)}", from=3-2, to=3-4]
			\arrow[from=3-2, to=4-2]
			\arrow[from=3-4, to=4-4]
			\arrow[between={0.1}{0.9}, equals, from=4-2, to=4-1]
			\arrow[between={0.1}{0.9}, equals, from=4-2, to=4-4]
		\end{tikzcd}
\end{adjustbox}\]
		As $\alpha= \beta * \omega(\beta)$, here $\Cut^{C_2}(\alpha)$ is defined to be $\underline{\Cut}(\beta)$, where $\underline{\Cut}$ is defined to be all cuts (trivial or not), with two trivial cuts not identified; i.e. $\Cut = \underline{\Cut}_{ / \{\text{two trivial cuts identified}\} }$. In $\underline{\Cut}(\beta)$ the first trivial cut corresponds to the trivial cut of $[2n+1]$ and the second trivial cut corresponds to the $C_2$-fixed cut of $[2n+1]$. Explicitly,
		\begin{itemize}
			\item for the trivial cut $\Cut^{C_2}(\alpha)^{-1}(C_2/C_2) = \underline{\Cut}(\beta)^{-1} (\phi \amalg [n]) = C_2/C_2 \coprod (C_2/e \times T)$, where $C_2/C_2$ corresponds to the trivial cut and $T \subset [m-1]$,
			\item for the fixed cut $\Cut^{C_2}(\alpha)^{-1}(C_2/C_2) = \underline{\Cut}(\beta)^{-1} ([n] \amalg \phi) = (C_2/e \times S) \coprod  C_2/C_2$, where $C_2/C_2$ corresponds to the $C_2$-fixed cut and $S \subset [m-1]$,
			\item for any other cut the inverse image is just inverse image under $\underline{\Cut}(\beta)$.
		\end{itemize} 
	
	Noting that under the composition with the coCartesian fibration $\tilde{\Delta} ^{op} \xrightarrow{{\tiny\Cut}^{C_2}} \Esgmop \twoheadrightarrow \underline{\mbox{Fin}}_{*} ^{C_2}$, the image of the morphisms are all active morphisms in $\underline{\mbox{Fin}}_{*} ^{C_2}$. Therefore, we get a functor to the active part of the operad $\tilde{\Delta} ^{op} \xrightarrow{{\tiny\Cut}^{C_2}} (\Esgmop)_{\text{act}}$.
\end{cons} 
	
\vspace*{0.2cm}

	In the construction above, the $C_2$-sets $C_2/e \times [m]$ and $C_2/C_2 \times [n]$ are just $C_2/e \times \{0,1,..., m\}$ and $C_2/C_2 \times \{0,1,..., n\}$ respectively. The reason for writing it in the above way, is to indicate the order of the underlying sets and make it easier to demonstrate the order on inverse images.
\begin{defn}
	Given a $\Esgm$-algebra or twisted monoid $A$ in $\underline{Sp}^{C_2}$ (or in \underline{\text{R-Mod}} for a $C_2$-$\mathbb{E}_{\infty}$-ring spectrum $R$), we define the real topological Hochschild homology as the following $C_2$-colimits respectively:
	$$\THR(A) := C_2 \text{-colimit} \; \left( \tilde{\Delta} ^{op} \xrightarrow{{\tiny \Cut}^{C_2}} (\Esgmop)_{\text{act}} \xrightarrow{A^{\otimes}} ({\underline{Sp}^{C_2}}^{\otimes})_{\text{act}} \xrightarrow{\otimes} \underline{Sp}^{C_2} \right)$$
	$$\THR^{R} (A) := C_2 \text{-colimit} \; \left( \tilde{\Delta} ^{op} \xrightarrow{{\tiny \Cut}^{C_2}} (\Esgmop)_{\text{act}} \xrightarrow{A^{\otimes}} ({\underline{\text{R-Mod}}}^{\otimes})_{\text{act}} \xrightarrow{\otimes} \underline{\text{R-Mod}} \right)$$
	Note that such a $C_2$-colimit is a coCartesian section $s$ of $\underline{Sp}^{C_2} \twoheadrightarrow \Obcat$. We denote $\THR(A)$ as the genuine $C_2$-spectrum $s([C_2/C_2])$. 
\end{defn}
	
\vspace*{0.2cm}

	\begin{rmk} \label{rmk: Dihedral}
		We remark that the $C_2$-space $S^{\sigma}$ is the $C_2$-colimit of the following functor 
	$$ \tilde{\Delta} ^{op} \xrightarrow{{\tiny \Cut}^{C_2}} (\Esgmop)_{\text{act}} \twoheadrightarrow (\finC)_{\text{act}} \hookrightarrow {\underline{\Top}^{C_2}},$$
	where the last arrow is the inclusion of $C_2$-sets as discrete $C_2$-spaces. One way to realize this is to use the Elmendorf's theorem for G-diagrams \cite[Theorem 2.28]{DM16} and Example \ref{Lem: S-sigma as cyclic}. Thus, one can inspect that the following composition gives the dihedral bar construction of \ref{cons dihedral bar}
	$$ \tilde{\Delta} ^{op} \xrightarrow{{\tiny \Cut}^{C_2}} (\Esgmop)_{\text{act}} \xrightarrow{A^{\otimes}} ({\underline{\Top}^{C_2}}^{\otimes})_{\text{act}} \xrightarrow{\otimes} \underline{\Top}^{C_2}.$$
	\end{rmk}

\vspace*{0.2cm}
	
\subsection{Real Topological Hochschild Homology of Thom spectra}  
A formula for real topological Hochschild homology may be deduced from the general context of equivariant factorization homology of equivariant Thom spectra \cite{nonabelianpoincare}. We adapt a similar method to the specific case of $C_2$-Thom spectra which is applicable to real topological Hochschild homology, and obtain a simple formula for the resulting Thom spectrum that is used for computations in \S \ref{thrcalc}. 
We use the ideas from  $\Gamma$-$G$-spaces from \cite{Shi89}, \cite{Shi91}, \cite{Rek11},  adopted in this parametrized higher category setup. The following are the relevant bits in this context : 
	
	\begin{itemize}
		\item A group-like $\mathbb{E}_{\infty}$-$C_2$-space $\mathcal{G}$ gives rise to an endofunctor of $\underline{\Top}_{*} ^{C_2}$ by means of the  parametrized left Kan extension of \cite[Def 2.12]{nardin17}, \cite[10.4]{shah23}. 
			\[\begin{tikzcd}[cramped]
				\finC && {(\underline{\Top}_{*} ^{C_2})^{\otimes}} \\
				{(\underline{\Top}_{*} ^{C_2})^{\otimes}}
				\arrow["{\mathcal{G}^{\otimes}}", from=1-1, to=1-3]
				\arrow["{\Omega^{\infty}\mathbb{S}}"', from=1-1, to=2-1]
				\arrow[""{name=0, anchor=center, inner sep=0}, "{\widetilde{\mathcal{G}^{\otimes}} = Lan_{\Omega^{\infty}\mathbb{S}} \GG^{\otimes}}"', from=2-1, to=1-3]
				\arrow[between={0.3}{0.7}, Rightarrow, from=1-1, to=0]
			\end{tikzcd}\]
The underlying endofunctor $ \underline{\Top}_{*} ^{C_2} \rightarrow \underline{\Top}_{*} ^{C_2}$ is denoted by $\tilde{\GG}$.
		
		\item As the parametrized left Kan extension commutes with $C_2$-colimits, we may realize $\tilde{\GG} (S^{\sigma} _+)$ as a $C_2$-colimit of $\tilde{\Delta}^{op}$-indexed diagram, given by the dihedral bar construction of the underlying $\mathbb{E}_{\sigma}$-space of $\mathcal{G}$. (see Remark \ref{rmk: Dihedral})
			\begin{align*}
				\tilde{\GG} (S^{\sigma} _+)  & \simeq \tilde{\GG} \left( C_2\text{-colim} \left( \tilde{\Delta} ^{op} \xrightarrow{{\tiny \Cut}^{C_2}} (\Esgmop)_{\text{act}} \twoheadrightarrow (\finC)_{\text{act}} \hookrightarrow {\underline{\Top}^{C_2}} \xrightarrow{(-)_+} \underline{\Top}_* ^{C_2} \right) \right) \\
				& \simeq C_2\text{-colim} \left(  \tilde{\Delta} ^{op} \xrightarrow{{\tiny \Cut}^{C_2}} (\Esgmop)_{\text{act}} \twoheadrightarrow (\finC)_{\text{act}} \hookrightarrow {\underline{\Top}^{C_2}} \xrightarrow{(-)_+} \underline{\Top}_* ^{C_2} \xrightarrow{\widetilde{\mathcal{G}} } \underline{\Top}_* ^{C_2} \right) \\
				& \simeq C_2\text{-colim} \left( \tilde{\Delta} ^{op} \xrightarrow{{\tiny\Cut}^{C_2}} \Esgmop \xrightarrow{\mathcal{G}}  \underline{\Top}^{C_2} \right) \\
				& \simeq B^{di}(\mathcal{G})
			\end{align*}	
	The third equivalence comes from the formulation of the left Kan extension as a colimit over slice category. Moreover, one can also realize the $\sigma$-fold delooping of $B^{\sigma} \mathcal{G}$ of $\GG$, as $	\tilde{\mathcal{G}} (S^{\sigma}) $.  
			
			\item The induced endofunctor on $\underline{\Top}_{*} ^{C_2}$ also induces a genuine $C_2$-spectrum $g$ such that $\Omega^{\infty} (g \wedge X) \simeq \tilde{\mathcal{G}} (X) $ for any $X \in \Top_{*} ^{C_2}$.
	\end{itemize}
	
	\vspace*{0.2cm}
	
	We now state the main theorem of this section which establishes that the real topological Hochschild homology of a Thom spectrum, which comes from a twisted monoid map, is again a Thom spectrum.	
	\begin{thm}\label{thrthom}
		Given a twisted monoid map $f: X \rightarrow BGL_1(R)$, with its $\sigma$-fold delooping given by $B^{\sigma}f: B^{\sigma}X \rightarrow B^{\rho}GL_1(R)$, the real topological Hochschild homology of the Thom spectrum $\Th (f)$ admits the following equivalence as $R$-module
		$$\THR^{R} (\Th (f)) \simeq \Th (L^{\tilde{\eta}}f: L^{\sigma}B^{\sigma}(X) \rightarrow BGL_1 (R) ),$$
		where the map $L^{\tilde{\eta}}f$ is the following composition
		$$\resizebox{\linewidth}{!}{$L^{\sigma}B^{\sigma}X \xrightarrow{L^{\sigma}B^{\sigma}f} L^{\sigma}B^{\rho}GL_1(R) \xleftarrow{\simeq} BGL_1(R) \times B^{\rho}GL_1(R) \xrightarrow{id \times (-\tilde{\eta})^*} BGL_1(R) \times BGL_1(R) \xrightarrow{m} BGL_1(R)$}$$
	\end{thm}
	
	\begin{proof}
		As the $C_2$-Thom spectrum functor $\Th: \underline{\text{Top}^{C_2}}_{/ {\tiny BGL_1(R)}} \rightarrow \underline{R\text{-mod}}$ strongly preserves colimits and symmetric monoidal, it commutes with dihedral bar construction. Therefore, 
		$$\THR^{R}(\Th (f)) \simeq \Th \left(L^{\tilde{\eta}}f: B^{di}(X) \rightarrow B^{di} (BGL_1 (R)) \rightarrow BGL_1 (R) \right)$$
		where the last arrow is the augmentation map, induced by levelwise multiplication map of $BGL_1(R)$ in the dihedral bar construction.
	The dihedral bar constructions are weakly equivalent to $\sigma$-fold free loop space, by \ref{dihedral free loop}. Thus, we need to figure out the augmentation map 
\[L^{\sigma} B^{\rho}GL_1(R) \xleftarrow{\simeq} B^{di} (BGL_1 (R)) \rightarrow BGL_1 (R).\]
 As $BGL_1(R)$ is a $C_2$-$\mathbb{E}_{\infty}$-space, both the horizontal fiber sequences of \ref{diagram free loop} 
\[\begin{tikzcd}[cramped]
				BGL_1 R && {B^{di}BGL_1 R} && {B^{\rho} GL_1 R} \\
				{\Omega^{\sigma} B^{\sigma} BGL_1 R} && {L^{\sigma} B^{\sigma} BGL_1 R} && {B^{\rho}GL_1 R}
				\arrow[from=1-1, to=1-3]
				\arrow["\simeq"', from=1-1, to=2-1]
				\arrow[from=1-3, to=1-5]
				\arrow["\gamma"', from=1-3, to=2-3]
				\arrow[equals, from=1-5, to=2-5]
				\arrow[from=2-1, to=2-3]
				\arrow["ev", from=2-3, to=2-5]
			\end{tikzcd}\]
are split, the upper one by the map $r: S^{\sigma} _+ \rightarrow S^0$ under the identifications mentioned above, and the lower one by the inclusion of constant loops $B^{\rho} GL_1 (R) \rightarrow L^{\sigma} B^{\rho}GL_1(R)$. This splitting isn't compatible with the map of the fiber sequences; more explicitly, we claim the homotopy class of the composition 
\[B^{\rho} GL_1 (R) \rightarrow L^{\sigma} B^{\rho}GL_1(R) \xleftarrow{\simeq} B^{di} (BGL_1 (R)) \xrightarrow{r_*} BGL_1 (R),\] 
is equivalent to 
\[(-\tilde{\eta})^* : B^{\rho}GL_1(R) \simeq Map_{*} (S^{\rho}, B^{2\rho}GL_1(R)) \xrightarrow{(-\tilde{\eta})^*} Map_{*} (S^{2\rho -1}, B^{2\rho}GL_1(R)) \simeq BGL_1(R).\]
		
 Redraw the diagram \ref{diagram free loop} as follows
			\[\begin{adjustbox}{max width=\linewidth}
\begin{tikzcd}[cramped]
				{BGL_1(R)  \simeq \widetilde{BGL_1(R)} (S^{0})} & {B^{di}(BGL_1(R)) \simeq \widetilde{BGL_1(R)}(S^{\sigma} _+)} & {B^{\sigma}(BGL_1(R))  \simeq \widetilde{BGL_1(R)} (S^{\sigma})} \\
				{\Omega^{\sigma} B^{\rho}GL_1(R)} & {L^{\sigma} B^{\rho}GL_1(R)} & {B^{\rho}GL_1(R)}.
				\arrow[from=1-1, to=1-2]
				\arrow["\simeq"', from=1-1, to=2-1]
				\arrow["{r_*}"', shift left, curve={height=18pt}, from=1-2, to=1-1]
				\arrow[from=1-2, to=1-3]
				\arrow["\gamma", from=1-2, to=2-2]
				\arrow[equals, from=1-3, to=2-3]
				\arrow[from=2-1, to=2-2]
				\arrow["ev", from=2-2, to=2-3]
				\arrow["{\text{constant loops}}", shift right, curve={height=-18pt}, from=2-3, to=2-2]
			\end{tikzcd}
\end{adjustbox}\]
		As all the maps are $C_2$-$\mathbb{E}_{\infty}$, the above diagram comes from the spectrum level diagram below.
		\[\begin{tikzcd}[cramped]
			{bgl_1(R)} && {bgl_1(R) \wedge S^{\sigma}_+} && {bgl_1(R) \wedge S^{\sigma}} \\
			{F(S^{\sigma}, bgl_1(R) \wedge S^{\sigma})} && {F(S^{\sigma}_+, bgl_1(R) \wedge S^{\sigma})} && {bgl_1(R) \wedge S^{\sigma}}
			\arrow[from=1-1, to=1-3]
			\arrow["\simeq"', from=1-1, to=2-1]
			\arrow["{r_*}"', shift left, curve={height=18pt}, from=1-3, to=1-1]
			\arrow[from=1-3, to=1-5]
			\arrow["\gamma", from=1-3, to=2-3]
			\arrow[equals, from=1-5, to=2-5]
			\arrow[from=2-1, to=2-3]
			\arrow[from=2-3, to=2-5]
			\arrow["{r^*}", shift right, curve={height=-18pt}, from=2-5, to=2-3]
		\end{tikzcd}\]
		This is equivalent to the following simplified diagram.
		\[\begin{tikzcd}[cramped]
			{S^0} && {S^{\sigma}_+} && {S^{\sigma}} \\
			{F(S^{\sigma},S^{\sigma})} && {F(S^{\sigma}_+, S^{\sigma})} && {S^{\sigma}}
			\arrow[from=1-1, to=1-3]
			\arrow["\simeq"', from=1-1, to=2-1]
			\arrow["{r_*}"', shift left, curve={height=18pt}, from=1-3, to=1-1]
			\arrow[from=1-3, to=1-5]
			\arrow["\gamma", from=1-3, to=2-3]
			\arrow[equals, from=1-5, to=2-5]
			\arrow[from=2-1, to=2-3]
			\arrow[from=2-3, to=2-5]
			\arrow["{r^*}", shift right, curve={height=-18pt}, from=2-5, to=2-3]
		\end{tikzcd}\]
		Therefore, the map of interest is the following composition 
\[\Phi: S^{\sigma} \xrightarrow{r^*} F(S^{\sigma}_+, S^{\sigma}) \xrightarrow{\gamma ^{-1}} S^{\sigma}_+ \xrightarrow{r_*} S^0,\]
 whose homotopy class is $a. \tilde{\eta} \in \pi_{\sigma} ^{C_2} (S^0)$ for some $a \in \Z$. To prove our claim, it is enough to show $a=(-1)$. 
		
		Let $s: S^{\sigma} \rightarrow S^{\sigma}_+$ be the associated stable section. We can, instead, inspect the composition
\[S^{\sigma} \xrightarrow{s} S^{\sigma}_+ \xrightarrow{\gamma} F(S^{\sigma}_+, S^{\sigma}) \xrightarrow{s^*} F(S^{\sigma}, S^{\sigma}) ,\]
 whose adjoint is given as
		$$S^{\sigma} \wedge S^{\sigma} \xrightarrow{s \wedge s} S^{\sigma}_+ \wedge S^{\sigma}_+ \xrightarrow{m} S^{\sigma}_+ \rightarrow S^{\sigma}$$
		where $m$ is the multiplication of $S^{\sigma}$. The homotopy class of this composition is that of $\tilde{\eta}  \in \pi_{\sigma} ^{C_2} (S^0)$. This can be seen using the fact that the Hopf construction on the multiplication map $m: S^{\sigma} \times S^{\sigma} \rightarrow S^{\sigma}$ gives rise to $\tilde{\eta} : S^{2\sigma +1} \rightarrow \Sigma S^{\sigma} \simeq S^{\sigma +1}$. The join construction between two $C_2$-spaces naturally admits a $C_2$-action. Moreover, 
		\begin{align*}
			S^{\sigma} *  S^{\sigma} & \cong (C_2/C_2 \amalg C_2/C_2) * C_2/e * (C_2/C_2 \amalg C_2/C_2) * C_2/e \\
			& \cong (C_2/C_2 \amalg C_2/C_2) * (C_2/C_2 \amalg C_2/C_2) * C_2/e * C_2/e\\
			& \cong S^1 * C_2/e * C_2/e  \cong Z * C_2/e \cong S^{2\sigma +1},
		\end{align*}
		where $Z$ is the $C_2$-space with underlying space as $S^2$ with $C_2$-action flipping the hemispheres fixing the equator. This results in $\Phi \simeq -\tilde{\eta}$.
		Therefore, we have the description
		$$L^{\tilde{\eta}}f :{\resizebox{.95\hsize}{!}{$ L^{\sigma}B^{\sigma}X \xrightarrow{L^{\sigma}B^{\sigma}f} L^{\sigma}B^{\rho}GL_1(R) \xleftarrow{\simeq} BGL_1(R) \times B^{\rho}GL_1(R) \xrightarrow{id \times (-\tilde{\eta})^*} BGL_1(R) \times BGL_1(R) \xrightarrow{m} BGL_1(R),$}}$$
which completes the proof of the theorem.	
\end{proof}

\vspace*{0.2cm}

\section{A splitting of $gl_1K\R$} \label{splitting of units}
 	
		As the definition of the units of the ring spectrum $\KR$ depends solely on its connective cover, we work with $\kR$, the connective cover of $\KR$. Given that the spectrum $\kR$ is a $C_2$-$\E_{\infty}$-ring spectrum, it follows from \cite[Proposition 11.1.57]{HHR21} that the map from $\kR$ to its $2^{nd}$-slice section $\kR \rightarrow P^2(\kR)$ is a $C_2$-$\E_{\infty}$-ring map. Consequently, this induces a map on the units $gl_1(\kR) \to gl_1(P^2(\kR))$. In this section we demonstrate that units of $\kR$ are split by $gl_1(P^{2}(\kR))$, that is, $gl_1 \kR \simeq gl_1(P^2\kR)\vee \hat{K}$. Further $gl_1(P^2\kR)$ is determined by a single $k$-invariant as follows. 
		
	\begin{prop} \label{ses of gl1P2kR} 
We have an equivalence
\[gl_1(P^{2}(\kR)) \simeq \Fib\ (H\underline{\F_2} \xrightarrow{\beta_{C_2}\ \circ\ Sq^2_{C_2}} \Sigma^{2+\sigma}H\underline{\Z}),\]
 where $\beta_{C_2} : H\underline{\F_2} \to \Sigma H\underline{\Z}$ is the $C_2$-equivariant B\"{o}ckstein and $Sq^2_{C_2} : H\underline{\F_2} \to \Sigma^{1+\sigma}H\underline{\F_2}$ is the $C_2$-equivariant Steenrod square.
	\end{prop}

	\begin{proof}
		From the slice tower computations for $\kR$  described in \cite[Theorem 1.1]{Dug05}, there is a cofiber sequence 
\begin{myeq} \label{cofslgl1} 
\Sigma^{1+\sigma}\HZ \to P^{2}(\kR) \to  \HZ .
\end{myeq}
This means that $\Sigma^{1+\sigma}\HZ$ is the $0$-connected cover of $P^2(\kR)$. Accordingly, we arrive at the cofiber sequence
		$$\Sigma^{1+\sigma}\HZ \rightarrow gl_1(P^{2}(\kR)) \rightarrow \HFII.$$
		It remains to identify the connecting map $\HFII \xrightarrow{\phi} \Sigma^{2+\sigma}\HZ$ as $\beta_{C_2}\circ Sq^2_{C_2}$. We know from \cite[Theorem 3.32 and Corollary 3.34]{BLM23} that $\res^{\CII}_e (\phi) \simeq \beta \circ Sq^2=\res^{C_2}_e(\beta_{C_2}\circ Sq^2_{C_2})$. Hence, $\phi$ can be identified as $\beta_{C_2} \circ Sq^2_{C_2}$, using the proposition \ref{beta sq2}.		
	\end{proof}
	
\vspace*{0.2cm}

	For an $\E_{\infty}$-ring spectrum with $C_2$-action by $\mathbb{E}_{\infty}$-maps, the pullback definition of the space of units $GL_1(R)$ decends to a functor
	$$gl_1: \;(\E_{\infty} \text{-ring spectra})^{BC_2} \xlongrightarrow{GL_1} (\text{group-like } \E_{\infty} \text{-spaces})^{BC_2} \longrightarrow Sp^{BC_2}.$$
Here the functor $B_{C_2}: Sp^{BC_2} \rightarrow Sp^{C_2}$ is the right adjoint to the forgetful functor $i: Sp^{C_2} \rightarrow Sp^{BC_2}$ \cite[Theorem II.2.7] {NS18}, which is described as the cofree spectrum. The following Proposition relates this functor to the units of genuine equivariant ring spectra using the fact that the functor $B_{C_2}$ is lax monoidal \cite[Corollary II.2.8]{NS18}.	
	\begin{prop} \label{naive and genuine units}
		For $R \in (\E_{\infty} \text{-ring spectra})^{BC_2}$, we get a map of genuine $C_2$-spectra $B_{C_2}(gl_1(R)) \rightarrow gl_1(B_{C_2}(R))$.
	\end{prop}
	
	\begin{proof}
		It is enough to get a genuine $C_2$-spectrum map 
\[\Sigma_{+}^{\infty} \Omega^{\infty} B_{C_2}(gl_1(R)) \rightarrow B_{C_2}R\] of $\E_\infty$-ring spectra. Following \cite[II.2.7]{NS18}, we need to arrive at the adjoint map 
\[\Sigma_{+}^{\infty} \Omega^{\infty} B_{C_2}(gl_1(R)) \rightarrow R\]
in $(\E_{\infty} \text{-ring spectra})^{BC_2}$. Starting from the counit
\[i \big(B_{C_2}gl_1 R\big) \to gl_1R,\]
and then applying the usual adjunction for the unit spectrum, we obtain a map 
\[\Sigma_{+}^{\infty} \Omega^{\infty} B_{C_2}(gl_1(R)) \rightarrow R\]
in $(\E_{\infty} \text{-ring spectra})^{BC_2}$.
	\end{proof}
	
\vspace*{0.2cm}
	
		The $C_2$-action on complex vector bundles by complex conjugation leads to the spectrum with $C_2$ action $KU_{C_2} \in Sp^{BC_2}$. In light of the fact that $B_{C_2}(KU_{C_2}) \simeq \KR$ \cite[Theorem 5.2]{Faj95} and Proposition \ref{naive and genuine units}, we are going to leverage the non-equivariant splitting of units of $ku$ \cite{BLM23} in $Sp^{BC_2}$ to yield a splitting of  $gl_1(\KR)$ in $Sp^{C_2}$.

	\begin{defn} \label{groupoid of line bundle}
		Let $\z$ be the groupoid of pairs $(V,n)$ where $V$ is a one dimensional complex vector space and $n \in \{0,1\}$. The space of morphisms from $(V,n)$ to $(V',n')$ will be 
		\[\Map_{\z} ((V,n),(V',n')) =
		\begin{cases}
			S^1\; (\simeq \; \mathbb{C}\text{-linear isomorphisms from }V \text{ to }V') &\text{when } n = n' \\
			\phi &\text{otherwise}
					\end{cases}\]
Up to isomorphism, there are two objects in $\z$, namely, $(\C,0)$ and $(\C,1)$ and each of them has automorphisms $=S^1$. A symmetric monoidal structure on $\z$ can be defined by $(V,n) \otimes (V',n') := (V\otimes V' ,n+n')$, where the symmetry isomorphism $(V,n) \otimes (V',n') \rightarrow (V',n') \otimes (V,n)$ is given by $(z,w) \mapsto (w,(-1)^{nn'} z)$. 		
	\end{defn}

\vspace*{0.2cm}

 It is clear that $\z$ is  a group-like symmetric monoidal category.		
		Moreover, it has a monoidal action of $C_2$  which is identity on objects and acts as complex conjugation on morphism spaces. 	

Analogous to \cite{BLM23}, we construct a model of $KU_{C_2}$ via $\Z / 2\Z$-graded chain complexes of complex vector spaces, with $C_2$-action via complex conjugation.  
	\begin{cons}
		Consider $\dsv$ to be the groupoid 
		\[ \dsv :=
		\begin{cases}
			\text{obj($\dsv$)} : &\text{homotopy classes of $\Z / 2\Z$-graded chain complexes} \\ & \text{of complex vector spaces with $C_2$-action} \\
			& \text {via complex conjugation}  \\
			\text{Morphisms} : &\text{chain homotopy classes of homotopy equivalences} \\
			& \text{between them}	
		\end{cases}\]
The $C_2$-action on morphisms is induced by complex conjugation. The direct sum and tensor product of complexes makes $\dsv$ a $\textit{Ring Groupoid}$ \cite{Dri21, Lap72} with $C_2$-action. 
		 Let $gl_1(\dsv)$ be the corresponding Picard groupoid (under $\otimes$) with $C_2$-action. Replicating the arguments of \cite{BLM23}, we obtain that  the spectra with $C_2$-action corresponding to $\dsv$ and $gl_1(\dsv)$, are $KU_{C_2}$ and $gl_1(KU_{C_2})$ respectively.
	\end{cons}

\vspace*{0.2cm}

We now note that there is an evident map of symmetric monoidal categories $\z \to gl_1(\dsv)$ which commutes with the $C_2$-action on morphism spaces. This  provides a lift of the splitting map of $gl_1ku$ proved in \cite{BLM23} to $Sp^{BC_2}$. 
	\begin{rmk} \label{sptz}
		By taking nerves there is an inclusion of the $\infty$-category of groupoids into the $\infty$-category of spaces, $ Gpd \subset \mathcal{S}$, under which the group-like symmetric monoidal groupoids land in group-like $\mathbb{E}_{\infty}$-spaces. Therefore, a group-like symmetric monoidal groupoid corresponds to a connective spectrum. Additionally, a $C_2$-action via lax monoidal functors descends to the functor 
\[Spt:\; (\mbox{Group-like symmetric  monoidal groupoids})^{BC_2} \rightarrow Sp^{BC_2}.\] 
 Furthermore, the non-equivariant splitting map of \cite[Proposition 3.31, Corollary 3.34]{BLM23} $Spt(\z) \rightarrow gl_1(KU)$ results in a map of spectra $Spt(\z) \rightarrow gl_1(KU_{C_2})$ in $Sp^{BC_2}$.
	\end{rmk}

\vspace*{0.2cm}

Let $P^2k\R$ denote the second slice section of $k\R$.	The next couple of results demonstrates the identification of $Spt(\z)$ as  $i(gl_1(P^2 \kR))\in Sp^{BC_2}$ and provides a splitting map $gl_1(P^2 \kR) \rightarrow gl_1(\kR)$.
	
	\begin{prop} \label{sptz and gl1 equivalance}
		$Spt(\z)$ is equivalent to $i\big(gl_1(P^2 \kR)\big)$ in $Sp^{BC_2}$.
	\end{prop}
	
	\begin{proof}
	For the proof of the result, we require a map $Spt(\z)\to i\big(gl_1(P^2 \kR)\big)$ in $Sp^{BC_2}$ which is an isomorphism on $\pi_n^e$ for all $n\geq 0$. Let $X \mapsto \PP^2 X$ denote the functor which takes a spectrum to its $2^{nd}$ Postnikov section. We know from \cite[Corollary 3.34]{BLM23} that $Spt(\z) \simeq gl_1(KU_{C_2})$, and thus the homotopy groups of $Spt(\z)$ are given by 
 \[\pi_{n}^{e}Spt(\z) \cong
		\begin{cases}
			\F_2 & \mathrm{when }\; n=0 \\
			\begin{tikzcd}[cramped]
				\Z
				\arrow["{\times (\text{-}1)}", from=1-1, to=1-1, loop, in=240, out=300, distance=5mm]
			\end{tikzcd} & \mathrm{when}~ n=2 \\
			~ \ \; 0 & \mathrm{otherwise.} 
		\end{cases}\]
 Next we compute $\pi_n^e(gl_1P^2(\kR)$ via the following cofiber sequence of Proposition \ref{ses of gl1P2kR}
		$$\Sigma^{1+\sigma}\HZ \rightarrow gl_1(P^{2}(k\mathbb{R})) \rightarrow \HFII,$$
and the following  induced long exact sequence 		
		$$...\; \rightarrow \pi_{n-1-\sigma}^{e}(\HZ) \rightarrow \pi_{n}^{e}(gl_1(P^{2}(k\mathbb{R})) \rightarrow \pi_{n}^{e}(\HFII) \rightarrow \; ....$$
		Identification of the groups $\pi_{n-\sigma}^{e}(\HZ)$ is a consequence of the long exact sequence
 $$... \; \rightarrow \underline{\pi}_{n}(C_2/e _{+} \wedge \HZ) \rightarrow \underline{\pi}_{n}(\HZ) \rightarrow  \underline{\pi}_{n-\sigma}(\HZ) \rightarrow \; ...$$
		\begin{tabular}{p{0.05\textwidth} p{0.95\textwidth}}
			
		 \vspace*{0.25cm} \framebox[1.2cm]{$\mathbf{n=0}$}
		
		& \[\begin{adjustbox}{max width=\linewidth} \begin{tikzcd}[cramped]
			& {\underline{\pi}_{0}(C_2/e _{+} \wedge \HZ)} && {\underline{\pi}_{0}(\HZ)} && {\underline{\pi}_{-\sigma}(\HZ) } \\
			{C_2/C_2} & \Z && \Z && {\Z/2\Z} \\
			{C_2/e} & {\Z \oplus \Z} && \Z && 0
			\arrow[from=1-2, to=1-4]
			\arrow[from=1-4, to=1-6]
			\arrow["{res^{C_2}_e}"', shift right, from=2-1, to=3-1]
			\arrow["{\times 2}", from=2-2, to=2-4]
			\arrow["\Delta"', shift right, from=2-2, to=3-2]
			\arrow[from=2-4, to=2-6]
			\arrow["id"', shift right, from=2-4, to=3-4]
			\arrow[shift right, from=2-6, to=3-6]
			\arrow["{tr^{C_2}_e}"', shift right, from=3-1, to=2-1]
			\arrow["\nabla"', shift right, from=3-2, to=2-2]
			\arrow["\nabla", from=3-2, to=3-4]
			\arrow["{\times 2}"', shift right, from=3-4, to=2-4]
			\arrow[from=3-4, to=3-6]
			\arrow[shift right, from=3-6, to=2-6]
		\end{tikzcd}
\end{adjustbox}\]
		\end{tabular}
		\begin{tabular}{p{0.05\textwidth} p{0.95\textwidth}}
		
		\vspace*{0.25cm} \framebox[1.2cm]{$\mathbf{n=1}$}
		
		& \[\begin{adjustbox}{max width=13cm}
\begin{tikzcd}[cramped]
			&& {\underline{\pi}_{1-\sigma}(\HZ) } && {\underline{\pi}_{0}(C_2/e _{+} \wedge \HZ)} && {\underline{\pi}_{0}(\HZ)} \\
			{C_2/C_2} && 0 && \Z && \Z \\
			{C_2/e} && \Z && {\Z \oplus \Z} && \Z
			\arrow[from=1-3, to=1-5]
			\arrow[from=1-5, to=1-7]
			\arrow["{res^{C_2}_e}"', shift right, from=2-1, to=3-1]
			\arrow[from=2-3, to=2-5]
			\arrow[shift right, from=2-3, to=3-3]
			\arrow["{\times 2}", from=2-5, to=2-7]
			\arrow["\Delta"', shift right, from=2-5, to=3-5]
			\arrow["id"', shift right, from=2-7, to=3-7]
			\arrow["{tr^{C_2}_e}"', shift right, from=3-1, to=2-1]
			\arrow[from=3-1, to=3-1, loop, in=240, out=300, distance=5mm]
			\arrow[shift right, from=3-3, to=2-3]
			\arrow["{\times (\text{-}1)}", from=3-3, to=3-3, loop, in=240, out=300, distance=5mm]
			\arrow[from=3-3, to=3-5]
			\arrow["\nabla"', shift right, from=3-5, to=2-5]
			\arrow["{\text{switch}}", from=3-5, to=3-5, loop, in=240, out=300, distance=5mm]
			\arrow["\nabla", from=3-5, to=3-7]
			\arrow["{\times 2}"', shift right, from=3-7, to=2-7]
			\arrow[from=3-7, to=3-7, loop, in=240, out=300, distance=5mm]
		\end{tikzcd}
\end{adjustbox}\]

		\end{tabular}

\mbox{  }\\
		This results in 
\[\pi_{n}^{e}(gl_1(P^{2}(k\mathbb{R}))\; \cong
		\begin{cases}
			~ \ \mathbb{F}_2 & \mathrm{when }\; n=0 \\
			\begin{tikzcd}[cramped]
				\Z
				\arrow["{\times (\text{-}1)}", from=1-1, to=1-1, loop, in=240, out=300, distance=5mm]
			\end{tikzcd} & \mathrm{when}~ n=2 \\
			~ \ \; 0 & \mathrm{otherwise,} 
		\end{cases}\]
		which is the same in value to $\pi_{n}^{e}Spt(\z)$. Thus, it suffices to construct a map $Spt(\z) \to i\big(gl_1(P^2 \kR)\big)$ in $Sp^{BC_2}$ which is a lift of the non-equivariant equivalence. 
		This is given by the following composition of maps in $Sp^{BC_2}$
		$$Spt(\z) \xrightarrow{(\alpha)} gl_1(KU_{C_2}) \xrightarrow{(\beta)} \PP^2gl_1(KU_{C_2}) \xrightarrow{(\gamma)} i\big(gl_1(P^2 \kR)\big).$$
		The map $\alpha$ is constructed in Remark \ref{sptz}. The map $\beta$ is the extension of the $2^{nd}$ Postnikov section $\PP^2$ to $Sp^{BC_2}$, constructed by taking the homotopy limit of $\infty$-categories
		\[\begin{tikzcd}[cramped]
			Sp && {Sp_{\leq 2}} & {~} & {~} & {~} & {Sp^{BC_2}} && {Sp_{\leq 2}^{BC_2}.}
			\arrow["\begin{array}{c} \text{trivial} \\ C_2 \text{-action} \end{array}", from=1-1, to=1-1, loop, in=60, out=120, distance=5mm]
			\arrow["{\PP^2}", from=1-1, to=1-3]
			\arrow["\begin{array}{c} \text{trivial} \\ C_2 \text{-action} \end{array}", from=1-3, to=1-3, loop, in=60, out=120, distance=5mm]
			\arrow["holim", squiggly, from=1-4, to=1-6]
			\arrow["{\PP^2}", from=1-7, to=1-9]
		\end{tikzcd}\]
It is clear that $(\beta)\circ (\alpha)$ is an equivalence. 

We note that $\KR\simeq B_{C_2}KU_{C_2}$ implies  \cite[Theorem II.2.7]{NS18} that $i(\KR) \simeq KU_{C_2}$. From Proposition \ref{naive and genuine units}, we have 
\[ B_{C_2} (gl_1 KU_{C_2}) \to gl_1 (B_{C_2} KU_{C_2}) \simeq gl_1\KR, \]
which yields the map 
\[ gl_1KU_{C_2} \simeq i B_{C_2} (gl_1 KU_{C_2}) \to i(gl_1\KR) .\]
The composite is directly seen to be an isomorphism on $\pi_n^e$ by the computation of non-equivariant homotopy groups. We thus have the factorization 
\[ \xymatrix{ gl_1KU_{C_2} \ar[r]^{\simeq} \ar[dr] &   igl_1\KR \ar[r] & i(gl_1P^2\KR) \\
 & \PP^2(gl_1KU_{C_2}), \ar@{-->}[ru]^{(\gamma)}} \] 
which defines the $C_2$-map $(\gamma)$ via the equivalence $gl_1 \KR \simeq gl_1 \kR$. 
	\end{proof}
	
\vspace*{0.2cm}
	
We now put all the ingredients together to show that the $2^{nd}$-slice section of $gl_1\KR$ splits off as a wedge summand. 
	\begin{thm}\label{splitglkr}
		There is a map $gl_1(P^2 \kR) \rightarrow gl_1(\KR)$ such that the composite 
\[gl_1(P^2\kR) \to gl_1(\KR)\simeq gl_1(\kR) \to gl_1(P^2 \kR)\]
 is equivalent to the identity.
	\end{thm}
	
	\begin{proof}
		The splitting map is obtained as composition of the following maps:
		
	\[\begin{tikzcd}[cramped]
		{B_{C_2}\big(Spt(\z)\big)} & {B_{C_2}\Big(gl_1\big(\Sigma^{\infty}_{+} \Omega^{\infty}Spt(\z)\big)\Big)} && {gl_1(B_{C_2}\big(\Sigma^{\infty}_{+} \Omega^{\infty}Spt(\z))\big)} \\
		{B_{C_2}\Big(i\big(gl_1(P^2 \kR)\big)\Big)} &&& \begin{array}{c} gl_1(B_{C_2}(KU_{C_2})) \\ \simeq gl_1(\KR) \end{array} \\
		\\
		{gl_1(P^2 \kR)}
		\arrow["{(\alpha)}", from=1-1, to=1-2]
		\arrow["\simeq"', from=1-1, to=2-1]
		\arrow["{(\beta)}", from=1-2, to=1-4]
		\arrow["{(\gamma)}", from=1-4, to=2-4]
		\arrow["\begin{array}{c} \text{unit of}\\ \text{adjunction}\\ i \dashv B_{C_2} \end{array}", from=4-1, to=2-1]
		\arrow["\begin{array}{c} \text{splitting}\\ \text{map} \end{array}", curve={height=18pt}, dashed, from=4-1, to=2-4]
	\end{tikzcd}\]
		 where the top left vertical equivalence in $Sp^{C_2}$ is from Proposition \ref{sptz and gl1 equivalance}. The map $(\beta)$ is induced by Proposition \ref{naive and genuine units}. The map $(\alpha)$ is obtained applying the functor $B_{C_2}$ to 
\[Spt(\z) \to gl_1 \Sigma^{\infty}_{+} \Omega^{\infty}Spt(\z),\]
which is the adjoint of the identity map 
\[\Sigma^{\infty}_{+} \Omega^{\infty}Spt(\z) \xrightarrow{id} \Sigma^{\infty}_{+} \Omega^{\infty}Spt(\z).\]
The map  $Spt(\z) \rightarrow gl_1(KU_{C_2}) \in Sp^{BC_2}$ of Remark \ref{sptz}, provides an $\E_{\infty}$-map
\[\Sigma^{\infty}_{+} \Omega^{\infty} Spt(\z) \rightarrow KU_{C_2} \in (\E_{\infty}\text{-spectra})^{BC_2}.\]
 The map $(\gamma)$ arises from applying $gl_1 \circ B_{C_2}$ to this map. 
Denote the splitting map by $\tau$, and by $\phi$ the composite 
\[  {gl_1(P^2 \kR)} \xrightarrow{\tau} {gl_1\kR} \to {gl_1(P^2 \kR)}.\]
We know that $\phi$ is a non-equivariant equivalence, and using this fact we prove that $\phi$ is an equivalence in $Sp^{C_2}$, thus completing the proof of the theorem. Using the cofiber \eqref{cofslgl1}, we have a homotopy commutative diagram 
\[\xymatrix{\Sigma^{1+\sigma} H\uZ \ar[r] \ar[d]^{a} & {gl_1(P^2 \kR)} \ar[d]^{\phi} \ar[r] & H\underline{\F_2} \ar[d]^{b}\\
\Sigma^{1+\sigma} H\uZ \ar[r] & {gl_1(P^2 \kR)}  \ar[r] & H\underline{\F_2}.} \]
The reason $\Sigma^{1+\sigma} H\uZ \to {gl_1(P^2 \kR)}  \xrightarrow{\phi} {gl_1(P^2 \kR)} $ factors through $\Sigma^{1+\sigma} H\uZ$ is that 
\[ [\Sigma^{1+\sigma} H\uZ, H\underline{\F_2}]^{C_2} = 0 \; \mbox{ as } \Sigma^{1+\sigma} H\uZ \; \mbox{ is } 0\mbox{-connected}.\]
The fact that $\phi$ is a non-equivariant equivalence, implies that so are $a$ and $b$. However, for these Mackey functors, this implies $a$ and $b$ are equivariant equivalences, and thus, by $5$-lemma, $\phi$ is an equivariant equivalence. 
	\end{proof}

\vspace*{0.2cm}

\section{$\THR^{K\mathbb{R}}(K\mathbb{R}/2)$} \label{thrcalc}

We have constructed $\KR/2$ as a twisted $\KR$-algebra in Proposition \ref{kuR/2 as twisted monoid}. This involved \\
1) The identification of $\KR/2$ as a Thom spectrum of a map $f: S^1 \to BGL_1\KR$.\\
2) The identification of $f$ as $\Omega^\sigma \zeta$ for a map $\zeta : \C P^\infty_\tau \to B^\rho GL_1\KR$. \\
We call these twisted algebra structures as {\it Thomified} algebra structures, that is those algebras which arise as the Thom spectrum of a loop map. We now use the result that the real topological Hochschild homology is again a Thom spectrum from \S \ref{thrthL} to make explicit computations.

\subsection{$\THR^{\KR}(\KR/2)$ as a $C_2$-Thom spectrum}\label{thrthomkr}
We reduce the $C_2$-Thom spectrum 
\[\THR^{\KR}(\KR/2) \simeq \Th^{C_2}(L^\sigma \C P^\infty_\tau \to BGL_1\KR),\]
to a more computable cofiber sequence that takes the form 
\[ \KR \wedge \C P^\infty_\tau \to \KR \wedge \C P^\infty_\tau \to \THR^{\KR}(\KR/2).\] 
The evenness of $\KR$ implies the following proposition. 
\begin{prop} \label{Identify as null}
Any map from $\C P^\infty_\tau$ to $BGL_1\KR$ is null-homotopic, that is, 
$$[\C P^{\infty}_{\tau} , BGL_1(\KR)]^{C_2} \cong 0.$$ 
\end{prop}

\begin{proof}
	From the \S \ref{CP_infity cell}, we have $\C P^{\infty}_{\tau} \simeq \mbox{hocolim }  Z_n$, with a cofiber sequence 
\begin{myeq} \label{cofcp}
\dots \rightarrow S^{n\rho -1} \rightarrow Z_{n-1} \rightarrow Z_n \rightarrow S^{n\rho} \rightarrow \cdots .
\end{myeq}
	Thus, the desired group fits into the short exact sequence
	$$0 \rightarrow {\varprojlim}^{1}[\Sigma Z_{n},\Sigma gl_1 (\KR)]^{C_2} \rightarrow [\C P^{\infty}_{\tau} , BGL_1(\KR)]_\ast^{C_2} \rightarrow \varprojlim [Z_n, \Sigma gl_1 (\KR)]^{C_2} \rightarrow 0.$$
We check that for each $n$, $[Z_n,\Sigma gl_1\KR]^{C_2} = 0$ and 
$$[Z_n, gl_1\KR]^{C_2} \to [Z_{n-1},gl_1\KR]^{C_2}$$
 is surjective. It follows that  $\varprojlim^{1}$-term vanishes by the Mittag-Lefﬂer condition, and thus the proof is complete. The evenness of $\KR$ and the pullback description of $GL_1\KR \simeq \Omega^\infty gl_1\KR$ \eqref{gl1krpb} implies that 
\[ \pi_{n\rho - 1}^{C_2} (gl_1\KR)=0.\]
Now using the cofiber \eqref{cofcp} and taking maps into $gl_1 \KR$ we get the exact sequence 
\[  \pi_{n\rho}^{C_2} gl_1\KR \to [Z_n, gl_1\KR]^{C_2} \to [Z_{n-1}, gl_1\KR]^{C_2} \to 0,\]
which implies the surjectivity above.  
	We show, by induction on $n$, each of the terms of the limiting system is trivial. For $n=1$, 
$$[S^{\rho}, \Sigma gl_1 (\KR)]^{C_2} \cong [S^{\sigma}, GL_1 (\KR)]^{C_2}_\ast \cong \pi _{\sigma} ^{C_2} (\KR) \cong 0.$$
	 In the induction step, we turn to the long exact sequence, induced by \eqref{cofcp}	 
	 $$\dots \rightarrow [S^{n\rho}, \Sigma gl_1 (\KR)]^{C_2} \rightarrow [Z_n, \Sigma gl_1 (\KR)]^{C_2} \rightarrow [Z_{n-1}, \Sigma gl_1 (\KR)]^{C_2} \rightarrow \cdots .$$
	 The left group is $\pi _{n\rho-1} ^{C_2} (\KR) \cong 0$, and the right group is trivial by the induction hypothesis.		
\end{proof}

\vspace*{0.2cm}

The $\THR$ of $\Th^{R}(\Omega^\sigma \zeta)$ been identified as $\Th^{R}(L^{\tilde{\eta}} \zeta)$ in Theorem \ref{thrthom}. Proposition \ref{Identify as null} allows us to make a reduction which does not involve the Hopf map. 
\begin{prop} \label{THR from loop space}
The $\THR^{\KR}(\KR/2)$ for Thomified algebra structures as $\Th^{C_2}(\Omega^\sigma \zeta)$ may be computed using the formula	
	$$\THR^{\KR}(\KR/2)\ \simeq \Th^{K\R}(\hat{\zeta})$$
		where $\hat{\zeta}$ is the composition:
	$$L^{\sigma}{\C P^{\infty}_{\tau}} \xrightarrow{L^{\sigma}\zeta} L^{\sigma}B^{\rho}GL_1(\KR) \simeq BGL_1(\KR) \times B^{\rho}GL_1(\KR) \xrightarrow{\pi_{1}} BGL_1(\KR).$$
\end{prop}

\begin{proof}
From Theorem \ref{thrthom}, $\THR^{K\mathbb{R}}(K\mathbb{R}/2)\ \simeq \Th^{K\R}(L^{\tilde{\eta}}{\zeta})$, where $L^{\tilde{\eta}}{\zeta}$ is the composite:
	
	\[\begin{tikzcd}[cramped]
	{L^{\sigma} {\C P^{\infty}_{\tau}} } & {L^{\sigma}B^{\rho}GL_1(K\mathbb{R}) } & {BGL_1(K\mathbb{R}) \times B^{\rho}GL_1(K\mathbb{R}) } \\
	&& {BGL_1(K\mathbb{R}) \times BGL_1(K\mathbb{R}) } \\
	&& {BGL_1(K\mathbb{R}).}
	\arrow["{L^{\sigma}\zeta}", from=1-1, to=1-2]
	\arrow["{L^{\tilde{\eta}}{\zeta}}"', curve={height=18pt}, dashed, from=1-1, to=3-3]
	\arrow["\simeq"', from=1-3, to=1-2]
	\arrow["{id \times \tilde{\eta} ^{*}}", from=1-3, to=2-3]
	\arrow["m", from=2-3, to=3-3]
	\end{tikzcd}\]
	Observing that in the equivalence, the map 
$$L^{\sigma}B^{\rho}GL_1(K\mathbb{R}) \simeq BGL_1(K\mathbb{R}) \times B^{\rho}GL_1(K\mathbb{R}) \xrightarrow{\pi_2}  B^{\rho}GL_1(K\mathbb{R})$$
 is evaluation at the base point of $S^{\sigma}$, we evidently get a factorization of $L^{\sigma}\zeta \circ \pi_2$ through ${\C P^{\infty}_{\tau}}$. Hence, following Proposition \ref{Identify as null} we arrive at the following diagram: 
	\[\begin{tikzcd}[cramped]
		{L^{\sigma} {\C P^{\infty}_{\tau}} } && {L^{\sigma}B^{\rho}GL_1(K\mathbb{R}) } & {BGL_1(K\mathbb{R}) \times B^{\rho}GL_1(K\mathbb{R}) } \\
		{{\C P^{\infty}_{\tau}}} &&& {B^{\rho}GL_1(K\mathbb{R}) } \\
		&&& {BGL_1(K\mathbb{R}) }
		\arrow["{L^{\sigma}\zeta}", from=1-1, to=1-3]
		\arrow["ev"', from=1-1, to=2-1]
		\arrow["ev"', from=1-3, to=2-4]
		\arrow["\simeq"', from=1-4, to=1-3]
		\arrow["{\pi_2}", from=1-4, to=2-4]
		\arrow["\zeta", from=2-1, to=2-4]
		\arrow["{\simeq  ~ *}", curve={height=12pt}, from=2-1, to=3-4]
		\arrow["{\tilde{\eta} ^{*}}", from=2-4, to=3-4]
	\end{tikzcd}\]
	 Therefore, the equivalance $L^{\tilde{\eta}}{\zeta} \simeq \hat{\zeta}$ completes the proof.
\end{proof}

\vspace*{0.2cm}

We now proceed as in \cite[\S 3.2]{Bas17}, and obtain a cofiber sequence for $\THR$ using the form described in Proposition \ref{THR from loop space}.	Under the map $\hat{\zeta}$, the constant loops in $L^{\sigma} {\C P^{\infty}_{\tau}}$ goes to base point in $BGL_1(\KR)$ as described below	
	\[\begin{adjustbox}{max width=\linewidth}\begin{tikzcd}[cramped]
		{\C P^{\infty}_{\tau}} && {B^{\rho}GL_1(K\mathbb{R})} \\
		{L^{\sigma}{\C P}^{\infty}_{\tau}} && {L^{\sigma}B^{\rho}GL_1(K\mathbb{R})} & {BGL_1(K\mathbb{R}) \times B^{\rho}GL_1(K\mathbb{R})} & {BGL_1(K\mathbb{R}). }
		\arrow["\zeta", from=1-1, to=1-3]
		\arrow["\begin{array}{c} \mathrm{constant} \\ \mathrm{loops} \end{array}"', from=1-1, to=2-1]
		\arrow[from=1-3, to=2-3]
		\arrow["{*}", curve={height=-12pt}, from=1-3, to=2-5]
		\arrow["{L^{\sigma}\zeta}"', from=2-1, to=2-3]
		\arrow["\simeq", from=2-4, to=2-3]
		\arrow["{\pi_1}"', from=2-4, to=2-5]
	\end{tikzcd}
\end{adjustbox}\]
	Consequently, we get a unit $u\ \in (K\mathbb{R}^{0}({\C P^{\infty}_{\tau}}))^{\times}$ corresponding to the map:
	\begin{myeq} \label{unit identify}
\begin{tikzcd}[cramped]
		{\C P^{\infty}_{\tau}} && {L^{\sigma}{\C P}^{\infty}_{\tau} \simeq S^1 \times {\C P}^{\infty}_{\tau}} && {\Sigma{{\C P}^{\infty}_{\tau}}_{+}} \\
		&& {BGL_1{\KR}}
		\arrow["\begin{array}{c} \text{constant} \\ \text{loops} \end{array}", from=1-1, to=1-3]
		\arrow["\mbox{\tiny cofiber}", from=1-3, to=1-5]
		\arrow["{\hat{\zeta}}", from=1-3, to=2-3]
		\arrow["u", dashed, from=1-5, to=2-3]
	\end{tikzcd}
\end{myeq}
	
The following proposition constructs $\THR$ as a cofiber using the unit $u$. 
\begin{prop} \label{LES of THR}
There is a cofiber sequence
\[ \KR \wedge \C P^\infty_\tau \stackrel{u-1}{\to} \KR \wedge \C P^\infty_\tau \to \THR^{\KR}(\KR/2)\]
and thus on homotopy groups, we get a long exact sequence
	$$\dots \rightarrow K\mathbb{R}_{\star}({\C P^{\infty}_{\tau}}) \xrightarrow{u - 1} K\mathbb{R}_{\star}({\C P^{\infty}_{\tau}}) \rightarrow \pi^{C_2}_{\star}(\THR^{K\mathbb{R}}(K\mathbb{R}/2)) \rightarrow \cdots$$
	for the unit $u\ \in (K\mathbb{R}^{0}({\C P^{\infty}_{\tau}}))^{\times}$ as defined in \eqref{unit identify}.
\end{prop}

\begin{proof}
	As ${\C P^{\infty}_{\tau}} \simeq K(\underline{\mathbb{Z}},\rho)$ is a group-like $C_2$-$\mathbb{E}_{\infty}$-space, we've a splitting 
$$L^{\sigma}{\C P^{\infty}_{\tau}} \simeq \Omega^{\sigma}{\C P^{\infty}_{\tau}} \times {\C P^{\infty}_{\tau}}\  \simeq\  S^{1} \times {\C P^{\infty}_{\tau}}.$$
 Therefore, we get a $C_2$-(homotopy)-pushout square
	\[\begin{tikzcd}[cramped]
		{{\C P^{\infty}_{\tau}} \amalg {\C P^{\infty}_{\tau}}} &&& {{\C P^{\infty}_{\tau}}} \\
		\\
		{{\C P^{\infty}_{\tau}}} &&& {L^{\sigma}{\C P^{\infty}_{\tau}}},
		\arrow[from=1-1, to=1-4]
		\arrow[from=1-1, to=3-1]
		\arrow[from=1-4, to=3-4]
		\arrow[from=3-1, to=3-4]
	\end{tikzcd}\]
		where the maps ${\C P^{\infty}_{\tau}} \rightarrow L^{\sigma}{\C P^{\infty}_{\tau}}$ are inclusions of constant loops.
	As the $C_2$-Thom spectrum functor strongly preserves $C_2$-colimits, we also get a $C_2$-pushout square on the level of Thom spectra, (the notation $(-)^g$ denotes $\Th^{C_2}(g)$)	
	\[\begin{tikzcd}[cramped]
	{  K\mathbb{R} \wedge {\C P^{\infty}_{\tau}}_{+} \vee K\mathbb{R} \wedge {\C P^{\infty}_{\tau}}_{+}  \simeq ({\C P^{\infty}_{\tau}} \amalg {\C P^{\infty}_{\tau}})^{\hat{\zeta}}} && {({\C P^{\infty}_{\tau}})^{\hat{\zeta}} \simeq K\mathbb{R} \wedge {\C P^{\infty}_{\tau}}_{+}} \\
	\\
	{K\mathbb{R} \wedge {\C P^{\infty}_{\tau}}_{+} \simeq ({\C P^{\infty}_{\tau}})^{\hat{\zeta}}} && {\Th^{C_2}(\hat{\zeta}) \simeq \THR^{K\mathbb{R}}(K\mathbb{R}/2)}.
	\arrow[from=1-1, to=1-3]
	\arrow[from=1-1, to=3-1]
	\arrow[from=1-3, to=3-3]
	\arrow[from=3-1, to=3-3]
	\end{tikzcd}\]
		 Hence, we get the following long exact sequence on homotopy groups:
	$$\dots \rightarrow K\mathbb{R}_*({\C P^{\infty}_{\tau}}) \oplus K\mathbb{R}_*({\C P^{\infty}_{\tau}}) \rightarrow K\mathbb{R}_*({\C P^{\infty}_{\tau}}) \oplus K\mathbb{R}_*({\C P^{\infty}_{\tau}}) \rightarrow \pi^{C_2}_{*}(THR^{K\mathbb{R}}(K\mathbb{R}/2)) \rightarrow \cdots .$$
	
	Note that $\Th^{C_2}(\hat{\zeta})$ is obtained from identifying two trivial $GL_{1}(K\mathbb{R})$-bundles over ${\C P^{\infty}_{\tau}}$, where on the intersection ${\C P^{\infty}_{\tau}} \amalg {\C P^{\infty}_{\tau}}$ it is given by a map $\tilde{u}:\  {\C P^{\infty}_{\tau}} \amalg {\C P^{\infty}_{\tau}} \rightarrow GL_{1}(K\mathbb{R})$; whose adjoint is given by :	
	\[\begin{adjustbox}{max width=\linewidth}
\begin{tikzcd}[cramped]
		{{\C P^{\infty}_{\tau}} \amalg {\C P^{\infty}_{\tau}}} & {{\C P^{\infty}_{\tau}} \vee {\C P^{\infty}_{\tau}}} & {S^{1} \times {\C P^{\infty}_{\tau}} \simeq L^{\sigma}{\C P^{\infty}_{\tau}}} & {\Sigma{\C P^{\infty}_{\tau}}_{+} \vee \Sigma{\C P^{\infty}_{\tau}}_{+}} & {\cdots} \\
		&& {BGL_{1}(K\mathbb{R})}
		\arrow[from=1-1, to=1-2]
		\arrow[from=1-2, to=1-3]
		\arrow["0"', from=1-2, to=2-3]
		\arrow[from=1-3, to=1-4]
		\arrow[from=1-3, to=2-3]
		\arrow[from=1-4, to=1-5]
		\arrow["\tilde{u}", from=1-4, to=2-3]
	\end{tikzcd}
\end{adjustbox}\]
	 Therefore, the induced cofiber sequence from the pushout square takes the form:
	  $$ K\mathbb{R} \wedge {\C P^{\infty}_{\tau}}_{+}\ \vee\  K\mathbb{R} \wedge {\C P^{\infty}_{\tau}}_{+} \xrightarrow{\begin{pmatrix}
	 		1 & u\\
	 		u & 1
	 \end{pmatrix}}  K\mathbb{R} \wedge {\C P^{\infty}_{\tau}}_{+}\ \vee\  K\mathbb{R} \wedge {\C P^{\infty}_{\tau}}_{+} \rightarrow Th^{C_2}(\hat{\zeta})$$
	 Hence we get the desired long exact sequence on homotopy groups.
\end{proof}

\vspace*{0.2cm}

The following proposition further helps us understand the map $u$, and makes it efficient to compute it.

\begin{prop} \label{identify u}
	As a stable map $u:\Sigma{\C P^{\infty}_{\tau}}_{+} \rightarrow BGL_1(K\mathbb{R})$ lies in $[\C P^\infty_\tau, gl_1 \KR]^{C_2}$, and is homotopic to the composite 
$$\Sigma^{\rho}{\C P^{\infty}_{\tau}}_{+} \xrightarrow{\mu} {\C P^{\infty}_{\tau}}  \xrightarrow{\zeta} B^{\rho}GL_1(K\mathbb{R}),$$ 
where $\mu$ is induced by the $C_2$-equivariant multiplication on ${\C P^{\infty}_{\tau}}$ as the composition
	$$\Sigma^{\rho}{\C P^{\infty}_{\tau}}_{+} \simeq S^{\rho} \wedge {\C P^{\infty}_{\tau}}_{+} \xrightarrow{\text{(inclusion)}\wedge id} {\C P^{\infty}_{\tau}} \wedge {\C P^{\infty}_{\tau}}_{+} \xrightarrow{\text{multiplication}} {\C P^{\infty}_{\tau}}.$$
\end{prop}

\begin{proof}
Consider the following $C_2$-equivariant commutative diagram:
	\[\begin{tikzcd}[cramped]
		{S^{\sigma}\times L^{\sigma}\C P^{\infty}_{\tau}} &&& {S^{\sigma} \times L^{\sigma}B^{\rho}GL_1(K\mathbb{R})} \\
		{\C P^{\infty}_{\tau}} &&& {B^{\rho}GL_1(K\mathbb{R})}.
		\arrow["{S^{\sigma} \times L^{\sigma}\zeta}", from=1-1, to=1-4]
		\arrow["ev"', from=1-1, to=2-1]
		\arrow["ev", from=1-4, to=2-4]
		\arrow["\zeta"', from=2-1, to=2-4]
	\end{tikzcd}\]
	Under the identifications 
$$L^{\sigma}\C P^{\infty}_{\tau} \simeq S^1 \times \C P^{\infty}_{\tau}, \mbox{ and } L^{\sigma}B^{\rho}GL_1(K\mathbb{R}) \simeq BGL_1(K\mathbb{R}) \times B^{\rho}GL_1(K\mathbb{R}),$$
 we get the $\sigma$-fold suspension of the unit $\Sigma^{\sigma}u$ of Proposition \ref{LES of THR} fits into the following diagram. 	
	\[\begin{tikzcd}[cramped]
		{(S^{\sigma }\times  1 \times  \C P^{\infty}_{\tau}) \,  \vee \, (* \times S^1 \times  \C P^{\infty}_{\tau})}\\
		{S^{\sigma}\times S^1 \times \C P^{\infty}_{\tau}} &&& {S^{\sigma} \times BGL_1(K\mathbb{R}) \times B^{\rho}GL_1(K\mathbb{R})} \\
		{S^\rho \wedge {\C P^\infty_\tau}_+ \simeq S^{\sigma} \wedge \Sigma {\C P^{\infty}_{\tau} }_{+}} &&& {(S^{\sigma} \wedge BGL_1(K\mathbb{R}) )\wedge (B^{\rho}GL_1(K\mathbb{R}))_{+}} \\
		&&& {S^{\sigma} \wedge BGL_1(K\mathbb{R})}
		\arrow[hook, from=1-1, to=2-1]
		\arrow["{{S^{\sigma} \times L^{\sigma}\zeta}}", from=2-1, to=2-4]
		\arrow[from=2-1, to=3-1]
		\arrow["{\bar{\zeta}}", from=2-1, to=4-4]
		\arrow[from=2-4, to=3-4]
		\arrow["{{\Sigma^{\sigma}u}}"', from=3-1, to=4-4]
		\arrow["{{\pi_1}}", from=3-4, to=4-4]
	\end{tikzcd}\]
The map $\bar{\zeta}$ is clearly the pull back of $S^\sigma \wedge \hat{\zeta}$ under the quotient 
\[S^{\sigma }\times  S^1 \times  \C P^{\infty}_{\tau}  \to S^\sigma \wedge( S^1 \times  \C P^{\infty}_{\tau}).\]
The left vertical column above is a homotopy cofibration sequence, in which the cofiber is viewed as the strict quotient
\[	 S^{\sigma} \wedge \Sigma {\C P^{\infty}_{\tau} }_{+} \cong \frac{S^{\sigma}\times S^1 \times \C P^{\infty}_{\tau}}{(S^{\sigma }\times  1 \times  \C P^{\infty}_{\tau}  \vee \, * \times S^1 \times  \C P^{\infty}_{\tau})}.\]
The restriction of $\bar{\zeta}$ to $	S^{\sigma }\times  1 \times  \C P^{\infty}_{\tau}  \vee \, * \times S^1 \times  \C P^{\infty}_{\tau}$ is the constant map to the base-point, which identifies the map $\Sigma^\sigma u$ on the quotient space.

We now have the following commutative diagram, where all except the map $(1)$ are maps of spaces. The map (1) is the inclusion of the factor $S^{\rho} \wedge {\C P^{\infty}_{\tau}}_{+}$ in the stable splitting 
$$S^{\sigma} \wedge (S^1 \times {\C P^{\infty}_{\tau}}) \simeq S^\rho \wedge \C P^\infty_\tau \vee S^\sigma \wedge ( 1 \times \C P^\infty_\tau). $$
We denote the evaluation map by $\mbox{ ev}$.		
\[\begin{adjustbox}{max width=\linewidth}
\begin{tikzcd}[cramped]
		& {S^{\rho} \wedge {\C P^{\infty}_{\tau}}_+} &&&& {S^{\sigma} \wedge BGL_1(K\mathbb{R})} \\
		\\
		{S^{\sigma} \wedge (S^1 \times {\C P^{\infty}_{\tau}})} & {S^{\sigma} \wedge L^{\sigma}{\C P^{\infty}_{\tau}}} &&& {S^{\sigma} \wedge L^{\sigma}B^{\rho}GL_1(K\mathbb{R})} \\
		\\
		& {{\C P^{\infty}_{\tau}}} &&&& {B^{\rho}GL_1(K\mathbb{R})}
		\arrow["{\Sigma^{\sigma}u}", from=1-2, to=1-6]
		\arrow["{(1)}"', from=1-2, to=3-1]
		\arrow["{S^{\sigma} \wedge i}"', from=1-6, to=3-5]
		\arrow["\chi", from=1-6, to=5-6]
		\arrow["\simeq", from=3-1, to=3-2]
		\arrow[from=3-2, to=1-2]
		\arrow["{S^{\sigma} \wedge L^{\sigma}\zeta}", from=3-2, to=3-5]
		\arrow["{\mbox{\tiny ev}}", from=3-2, to=5-2]
		\arrow["{\mbox{\tiny ev}}", from=3-5, to=5-6]
		\arrow["\zeta", from=5-2, to=5-6]
	\end{tikzcd}
\end{adjustbox}\]
	In the diagram above $i: BGL_1(K\mathbb{R}) \rightarrow L^{\sigma}B^{\sigma}GL_1{K\mathbb{R}})$ is the inclusion of based loops and $\chi$ is the adjoint to the $C_2$-equivalence $BGL_1(K\mathbb{R}) \simeq \Omega^{\sigma}B^{\rho}GL_1(K\mathbb{R})$. 
	Note that the adjoint of $u$ is $(\chi \circ \Sigma^{\sigma}u)$, where 
$$\chi: \Sigma^{\sigma}BGL_1(K\mathbb{R}) \xrightarrow{\simeq} B^{\rho}GL_1(K\mathbb{R})$$ 
is essentially the evaluation map. Therefore, the result follows from the observation that under stable splitting, the composition 
$$S^{\sigma} \wedge \Sigma{ \C P^{\infty}_{\tau} }_{+} \hookrightarrow S^{\sigma}\times S^1 \times \C P^{\infty}_{\tau} \xleftarrow{\simeq} S^{\sigma}\times L^{\sigma} \C P^{\infty}_{\tau} \xrightarrow{ev} \C P^{\infty}_{\tau}$$ 
is homotopic to $\mu$. 	
\end{proof}

\vspace*{0.2cm}

\subsection{The homotopy type of $\THR^{\KR}(\KR/2)$}
We now compute $\THR^{\KR}(\KR/2)$ following the reductions in \S \ref{thrthomkr}. Recall some well-known computations for $C_2$-Eilenberg MacLane spectra. The notation $a_\sigma \in \pi^{C_2}_{-\sigma}(S^0)$ denotes the Euler class \cite[Definition 3.11]{HHR16}, and $u_{2\sigma} \in \pi^{C_2}_{2-2\sigma}(H\uZ)$ denotes the orientation class \cite[Definition 3.12]{HHR16}.  
\begin{myeq}\label{F2comp}
			\pi_\bigstar^{C_2} H\underline{\mathbb{F}_2} \cong \mathbb{F}_2[a_{\sigma}, u_{\sigma}] \oplus \Sigma^{-1} \mathbb{F}_2 \{a_{\sigma}^{-k} u_{\sigma}^{-l}\} \text{ for } k,l > 0 \mbox{ \cite[Proposition 3.5]{BG21}}.
\end{myeq}
\begin{myeq}\label{HZcomp}
\pi_\bigstar^{C_2} H\uZ \cong  \Z[a_{\sigma}, u_{2\sigma}]/(2a_\sigma) \oplus 2\Z[u_{2\sigma}^{-1}] \oplus \Sigma^{-1} \F_2 \{a_{\sigma}^{-k} u_{2\sigma}^{-l}\}  \text{ for } k,l > 0  \mbox{ \cite[Proposition 6.5]{Zen18}}.
\end{myeq}

The cohomology of $\C P^n_\tau$ with $\uZ$ and $\underline{\F_2}$ coefficients are readily computed from these formulas.
\begin{prop} \label{groups}
	For constant Mackey functors $\underline{\mathbb{Z}}$ and $\underline{\mathbb{F}_2}$, 
	$$H\uZ^{\bigstar}(\C P^n_{\tau}) \cong \pi_{-\bigstar}H\uZ[t] / (t^{n+1}),\ H\uZ^{\bigstar}(\C P^{\infty}_{\tau})  \cong \pi_{-\bigstar}H\uZ[[t]], $$
$$H\underline{\mathbb{F}_2}^{\bigstar}(\C P^n_{\tau}) \cong \pi_{-\bigstar}H\underline{\mathbb{F}_2}[t] / (t^{n+1}), \ H\underline{\mathbb{F}_2}^{\bigstar}(\C P^\infty_\tau) \cong \pi_{-\bigstar}H\underline{\mathbb{F}_2}[[t]],$$
	where $|t| = (1+\sigma)$.
	In particular, for a real oriented $C_2$-spectrum $E\mathbb{R}$, we get $E\mathbb{R}^{\star}(\C P^n) \cong E\mathbb{R}^{\star}[t] / (t^{n+1})$ and $E\mathbb{R}^{\star}({\C P^{\infty}_{\tau}}) \cong E\mathbb{R}^{\star}[[t]]$.
\end{prop}

\begin{proof}
	As $\underline{\pi}_{k\rho - 1}(H\underline{\mathbb{Z}}) = 0$ \eqref{HZcomp} and $\underline{\pi}_{k\rho - 1}(H\underline{\mathbb{F}_2}) = 0$ \eqref{F2comp}, by \cite[Definition 3.1]{hillmeier}, $H\underline{\mathbb{Z}}$ and $H\underline{\mathbb{F}_2}$ are even $C_2$-spectra. They are real orientable by \cite[Lemma 3.3]{hillmeier}. Now, the proposition follows from \cite[Proposition 4.2]{Ara79}. 
\end{proof}
	
\vspace*{0.2cm}
		
As $\KR$ is real oriented \cite[Theorem 2.8]{hukriz}, we have the formulas	
	\begin{gather*}
		K\mathbb{R}^{\star}(\C P^n) \cong K\mathbb{R}^{\star}[t] / (t^{n+1}),\  \text{and}\ K\mathbb{R}^{\star}({\C P^{\infty}_{\tau}}) \cong K\mathbb{R}^{\star}[[t]].
	\end{gather*}
	Following Proposition \ref{identify u}, the unit $u \in (K\mathbb{R}^0({\C P^{\infty}_{\tau}}))^{\times}$ takes the form 
\begin{myeq}\label{unitform}
u\ =\ (-1)+\nu^{*}(\zeta)t + \text{(terms involving higher order of } t \mbox{)} 
\end{myeq}
where restricting to 2-skeleton we get the following maps:
		\begin{myeq}\label{cohomo groups}
\begin{tikzcd}[cramped]
		&& {S^{\rho}} \\
		{S^{\rho} \times S^{\rho}} && {\C P^{2}_\tau} && {\Sigma^{\rho} gl_1\KR} \\
		{S^{\rho} \vee S^{\rho} \vee S^{2\rho}} \\
		{S^{2\rho}}
		\arrow["{i}",from=1-3, to=2-3]
		\arrow["{(-1)}", from=1-3, to=2-5]
		\arrow["{\text{multiplication}}", from=2-1, to=2-3]
		\arrow["\zeta_{\C P^2}", from=2-3, to=2-5]
		\arrow["{\text{stable splitting}}", equals, from=3-1, to=2-1]
		\arrow["\nu"', from=4-1, to=2-5]
                \arrow["\chi"', from=4-1, to=2-3]
		\arrow[hook, from=4-1, to=3-1]
	\end{tikzcd}
\end{myeq}
We write $gl_1\KR \simeq E \vee \hat{K}$  by 	Theorem \ref{splitglkr}, where $E \simeq gl_1P^2\KR$. Then 
$$[\C P^2_\tau, \Sigma^\rho gl_1\KR]^{C_2} \cong  [\C P^2_\tau, \Sigma^\rho E]^{C_2},$$
which implies that in order to compute $\nu^\ast(\zeta_{\C P^2})$ it suffices to only consider $E$. 	
	The	next set of results are aimed at understanding $\nu^{*}(\zeta_{\C P^2})$.
	\begin{lem} \label{identify group 1}
		We have the following isomorphisms
		\[[{\C P^{\infty}_{\tau}},\Sigma^{\rho}E]^{C_2} \cong [\C P^2_\tau, \Sigma^\rho E]^{C_2}\cong \mathbb{Z}.\]
 Moreover, the image of $\zeta \in [{\C P^{\infty}_{\tau}},\Sigma^{\rho}E]^{C_2}$  via these isomorphisms is an odd integer.
	\end{lem}	
	
	\begin{proof}
	Note that $\C P^\infty_\tau/\C P^2_\tau$ has a $C_2$-cell complex structure from \eqref{cofcp} with cells of the form $\DD(n\rho)$ for $n\geq 3$. Then, these $S^{n\rho} $ are at least $2$-connected for $n\geq 3$, so that by applying the cofiber description of $E$ in Proposition \ref{ses of gl1P2kR} we observe 
\[ [\C P^\infty_\tau/\C P^2_\tau, E]^{C_2}=0, \mbox{ and } [\Sigma^{-1} \C P^\infty_\tau/\C P^2_\tau, \Sigma E]^{C_2}=0,\]
as 
\[\pi_{n\rho}^{C_2} H\uZ = 0 , \mbox{ for } n\neq 0, \mbox{ and } \pi_{n\rho -1}^{C_2} H\uZ = 0 \ \forall n \mbox{ \eqref{HZcomp}},\]
and 
\[\pi_{n\rho}^{C_2} H\uF = 0 , \mbox{ for } n\neq 0, \mbox{ and } \pi_{n\rho -1}^{C_2} H\uF = 0 \ \forall n \mbox{ \eqref{F2comp}}.\]
This proves the first isomorphism in the statement. We may construct the following diagram of cofiber sequences using Proposition \ref{ses of gl1P2kR}		
		\begin{myeq}\label{cofdiagE}
\begin{tikzcd}[cramped]
			{\Sigma^{\rho}H\underline{\mathbb{Z}}} && E && {H\underline{\mathbb{F}_2}} && {\Sigma^{\rho+1}H\underline{\mathbb{Z}}} \\
			{\Sigma^{\rho}H\underline{\mathbb{Z}}} && {\Sigma^{\rho}H\underline{\mathbb{Z}}} && {\Sigma^{\rho}H\underline{\mathbb{F}_2}} && {\Sigma^{\rho+1}H\underline{\mathbb{Z}}.}
			\arrow[from=1-1, to=1-3]
			\arrow[equals, from=1-1, to=2-1]
			\arrow[from=1-3, to=1-5]
			\arrow[from=1-3, to=2-3]
			\arrow["{\beta \circ Sq^2_{C_2}}", from=1-5, to=1-7]
			\arrow["{Sq^2_{C_2}}"', from=1-5, to=2-5]
			\arrow[equals, from=1-7, to=2-7]
			\arrow["{\times 2}", from=2-1, to=2-3]
			\arrow[from=2-3, to=2-5]
			\arrow["\beta", from=2-5, to=2-7]
		\end{tikzcd}
\end{myeq}
From Proposition \ref{groups}, and the formulas \eqref{HZcomp} and \eqref{F2comp}, we have 
\[ H\uZ^{2\rho}(\C P^2_\tau) \cong \Z,\ H\uZ^{2\rho+1}(\C P^2_\tau) =0,\  H\uF^\rho(\C P^2_\tau)\cong \F_2,\ H\uF^{2\rho}(\C P^2_\tau) \cong \F_2.\]		
Now we apply these calculations to the following diagram of exact sequences arising from \eqref{cofdiagE}		
		\[\begin{adjustbox}{max width=\linewidth}
\begin{tikzcd}[cramped]
			 & {H\underline{\mathbb{Z}}^{2\rho}(\C P^2_\tau) \cong \mathbb{Z}} & {[\C P^{2}_\tau,\Sigma^{\rho}E]^{C_2}} & {H\underline{\mathbb{F}_2}^{\rho}(\C P^2_\tau) \cong \mathbb{F}_2} & {H\underline{\mathbb{Z}}^{2\rho+1}(\C P^2_\tau) \cong 0} \\
			 & {H\underline{\mathbb{Z}}^{2\rho}(\C P^2_\tau) \cong \mathbb{Z}} & {H\underline{\mathbb{Z}}^{2\rho}(\C P^2_\tau) \cong \mathbb{Z}} & {H\underline{\mathbb{F}_2}^{2\rho}(\C P^2_\tau) \cong \mathbb{F}_2} & {H\underline{\mathbb{Z}}^{2\rho+1}(\C P^2_\tau) \cong 0}
			\arrow["{(\alpha)}", from=1-2, to=1-3]
			\arrow[equals, from=1-2, to=2-2]
			\arrow[from=1-3, to=1-4]
			\arrow[from=1-3, to=2-3]
			\arrow[from=1-4, to=1-5]
			\arrow["Sq^2_{C_2}"', from=1-4, to=2-4]
			\arrow[equals, from=1-5, to=2-5]
			\arrow["{\times 2}", from=2-2, to=2-3]
			\arrow[from=2-3, to=2-4]
			\arrow[from=2-4, to=2-5]
		\end{tikzcd}
\end{adjustbox}\]
We note that 
\[ Sq^2_{C_2} : H\uF^\rho(\C P^2_\tau) \to H\uF^{2\rho}(\C P^2_\tau)\]
is an isomorphism as 
\[ \res^{C_2}_e : H\uF^\rho(\C P^2_\tau) \stackrel{\cong}{\to} H^2(\C P^2;\F_2), \ 	\res^{C_2}_e : H\uF^{2\rho}(\C P^2_\tau) \stackrel{\cong}{\to} H^4(\C P^2;\F_2),\]
and 
\[ \res^{C_2}_e(Sq^2_{C_2}) = Sq^2 :  H^2(\C P^2; \F_2) \stackrel{\cong}{\to} H^4(\C P^2;\F_2).\]	
We also observe from Proposition \ref{groups} and \eqref{F2comp} that 
\[H\uF^{\rho-1}(\C P^2_\tau) \cong \F_2, \mbox{ and } H\uF^{2\rho-1}(\C P^2_\tau) \cong \F_2.\]
This implies that $(\alpha)$ is injective. Hence $[{\C P^2_{\tau}},\Sigma^{\rho}E]^{C_2} \cong \mathbb{Z}$ and the map $(\alpha)$ is $\mathbb{Z} \xrightarrow{\times 2} \mathbb{Z}$. Additionally, following  \eqref{cohomo groups}, 
\[i^{*}:\  \mathbb{Z} \cong [{\C P^2_{\tau}},\Sigma^{\rho}E]^{C_2} \rightarrow [S^{\rho}, \Sigma^{\rho}E]^{C_2} \cong (\mathbb{Z})^{\times}\] sends $\zeta_{\C P^2} \mapsto (-1)$. Therefore, the image of $\zeta \in [{\C P^{\infty}_{\tau}},\Sigma^{\rho}E]^{C_2} \cong \mathbb{Z}$ is an odd element.
	\end{proof}
	
\vspace*{0.2cm}

	\begin{lem} \label{identify group 2}
		Let $q:\C P^2_\tau \rightarrow S^{2\rho}$ be the quotient by the lower skeleton as in \eqref{cellcp}. Then, all the groups in the equation below are $\Z$, and both the maps are isomorphisms. 
\[[S^{2\rho},\Sigma^{2\rho}H\underline{\mathbb{Z}}]^{C_2} \xrightarrow{q^{*}} [\C P^2_\tau,\Sigma^{2\rho}H\underline{\mathbb{Z}}]^{C_2}, \ [S^{2\rho},\Sigma^{2\rho}H\underline{\mathbb{Z}}]^{C_2} \rightarrow [S^{2\rho},\Sigma^{\rho}E]^{C_2} \cong \mathbb{Z}.\]
	\end{lem}
	
	\begin{proof}
		From \eqref{attcp}, we get the cofiber 
		$$S^{\rho} \rightarrow \C P^2_\tau \xrightarrow{q} S^{2\rho}$$
		which gives the following long exact sequence on homotopy groups.
		\[\xymatrix{
			\cdots & [S^{\rho+1},\Sigma^{2\rho}H\uZ]^{C_2} \ar@{=}[d] \ar[r] & [S^{2\rho},\Sigma^{2\rho}H\uZ]^{C_2} \ar@{=}[d] \ar[r]  & [\C P^2_\tau,\Sigma^{2\rho}H\uZ]^{C_2} \ar[r] & [S^{\rho},\Sigma^{2\rho}H\uZ]^{C_2} \ar@{=}[d] \\
& \Z/2  & \Z & & 0. } \]			
The groups are identified by \eqref{HZcomp}. This proves the first statement. As $E\simeq gl_1P^2(\KR)$ we apply Proposition \ref{ses of gl1P2kR} to get the following exact sequence on homotopy groups		
		\[\begin{tikzcd}[cramped]
			{[S^{2\rho},\Sigma^{\rho -1}H\underline{\mathbb{F}_2}]^{C_2} } & {[S^{2\rho},\Sigma^{2\rho}H\underline{\mathbb{Z}}]^{C_2} } & {[S^{2\rho},\Sigma^{\rho}E]^{C_2} } & {[S^{2\rho},\Sigma^{\rho}H\underline{\mathbb{F}_2}]^{C_2} }
			\arrow[from=1-1, to=1-2]
			\arrow[from=1-2, to=1-3]
			\arrow[from=1-3, to=1-4]
		\end{tikzcd}\]
Now we have 
\[ {[S^{2\rho},\Sigma^{2\rho}H\underline{\mathbb{Z}}]^{C_2} } \cong \pi_0^{C_2} H\uZ \cong \Z,\]
\[	[S^{2\rho},\Sigma^{\sigma}H\underline{\mathbb{F}_2}]^{C_2} \cong \pi_{\rho+1}^{C_2} H \underline{\F_2} \cong 0 \mbox{ \eqref{F2comp}},\]	
and 
\[[S^{2\rho},\Sigma^{\rho}H\underline{\mathbb{F}_2}]^{C_2} \cong \pi_\rho^{C_2} H\underline{\F_2} \cong 0 \mbox{ as } \rho \mbox{ is } 0 \mbox{-connected}.\]
		Thus, we get the desired result.		
	\end{proof}

\vspace*{0.2cm}

The following proposition provides the final piece towards the calculation of $\nu^\ast(\zeta_{\C P^2})$.	
	\begin{prop} \label{identify tau*}
		The map $\chi^{*}: [\C P^2_\tau,\Sigma^{\rho}E]^{C_2} \cong \mathbb{Z} \longrightarrow [S^{2\rho},\Sigma^{\rho}E] \cong \mathbb{Z}$ is an isomorphism.
	\end{prop}
	
	\begin{proof}
		Consider the diagram (in stable category)
		\[\begin{tikzcd}[cramped]
			& {S^{\rho} \times S^{\rho}} \\
			{S^{2\rho}} && {\C P^2_\tau} && {S^{2\rho}.}
			\arrow["m", from=1-2, to=2-3]
			\arrow[hook, from=2-1, to=1-2]
			\arrow["\chi"', from=2-1, to=2-3]
			\arrow["{q \text{ (quotient)}}"', from=2-3, to=2-5]
		\end{tikzcd}\]
		 As non-equivariantly the map $q \circ \chi$ is of degree $2$, the induced map on Mackey functors take the form		 
		 \[\begin{tikzcd}[cramped]
		 	{C_2 / C_2} && {[S^{2\rho},\Sigma^{2\rho}H\underline{\mathbb{Z}}]^{C_2} \cong \mathbb{Z}} &&& {[S^{2\rho},\Sigma^{2\rho}H\underline{\mathbb{Z}}]^{C_2} \cong \mathbb{Z}} \\
		 	& {:} \\
		 	{C_2/e} && {[S^{2\rho},\Sigma^{2\rho}H\underline{\mathbb{Z}}]^{e} \cong \mathbb{Z}} &&& {[S^{2\rho},\Sigma^{2\rho}H\underline{\mathbb{Z}}]^{e} \cong \mathbb{Z}}.
		 	\arrow["{\tiny \res^{C_2}_e}"', shift right=2, from=1-1, to=3-1]
		 	\arrow["{\text{hence } \times 2}", from=1-3, to=1-6]
		 	\arrow["id"', shift right=2, from=1-3, to=3-3]
		 	\arrow["id"', shift right=2, from=1-6, to=3-6]
		 	\arrow["{\tiny \tr^{C_2}_e}"', shift right=2, from=3-1, to=1-1]
		 	\arrow["{C_2}", from=3-1, to=3-1, loop, in=235, out=305, distance=10mm]
		 	\arrow["{\times 2}"', shift right=2, from=3-3, to=1-3]
		 	\arrow[from=3-3, to=3-3, loop, in=235, out=305, distance=10mm]
		 	\arrow["\times2"', from=3-3, to=3-6]
		 	\arrow["{\times 2}"', shift right=2, from=3-6, to=1-6]
		 	\arrow[from=3-6, to=3-6, loop, in=235, out=305, distance=10mm]
		 \end{tikzcd}\]
		 
		 Therefore, we arrive at the following diagram, using Lemma \ref{identify group 1} where the map $\alpha$ is calculated along with isomorphisms along the middle column, and Lemma \ref{identify group 2}:
		 \[\begin{tikzcd}[cramped]
		 	{[S^{2\rho},\Sigma^{2\rho}H\mathbb{\underline{Z}}]^{C_2} \cong \mathbb{Z}} && {[\C P^2_\tau,\Sigma^{2\rho}H\mathbb{\underline{Z}}]^{C_2} \cong \mathbb{Z}} &&& {[S^{2\rho},\Sigma^{2\rho}H\mathbb{\underline{Z}}]^{C_2} \cong \mathbb{Z}} \\
		 	\\
		 	&& {[\C P^2_\tau,\Sigma^{\rho}E]^{C_2}  \cong \mathbb{Z}} &&& {[S^{2\rho},\Sigma^{\rho}E]^{C_2} \cong \mathbb{Z}}.
		 	\arrow["\cong", from=1-1, to=1-3]
		 	\arrow["{\times 2}", curve={height=-24pt}, from=1-1, to=1-6]
		 	\arrow[from=1-3, to=1-6]
		 	\arrow["{\alpha \equiv (\times 2)}", from=1-3, to=3-3]
		 	\arrow["\cong", from=1-6, to=3-6]
		 	\arrow["{\chi^*}", from=3-3, to=3-6]
		 \end{tikzcd}\]
		 		 Thus, $\chi^*$ is an isomorphism.		 
	\end{proof}
	
\vspace*{0.2cm}

We finally compute the homotopy type of $\THR^{\KR}(\KR/2)$ for Thomified algebra structures in the theorem below.
	\begin{thm} \label{pi* calculation}
The $RO(C_2)$-graded homotopy groups $\THR^{\KR}(\KR/2)$ are given by		
$$\pi _{k\rho +l} ^{C_2} (\THR^{\KR} (\KR /2)) \cong 
		\begin{cases}
			\Z /(2 ^{\infty}) \; & \text{ when } l \equiv 0,4 \, (\text{mod }8)\\
			\Z / {2} \;  \; & \text{ when } l \equiv 2,3 \, (\text{mod }8)\\
			0 \;  \; & \text{ when } l \equiv 1,5,6,7 \,(\text{mod }8).
		\end{cases}$$
			In fact, as a $\KR$-module  $\THR^{\KR} (\KR /2) \simeq \KR / (2^{\infty})$, where $\KR / (2^{\infty})$ is defined by the cofiber sequence
		$$\KR \rightarrow \KR [2^{-1}] \rightarrow \KR / (2^{\infty})$$
		with $\KR [2^{-1}] \simeq \text{colim}(\KR \xrightarrow{2} \KR \xrightarrow{2} \cdots)$.
	\end{thm}
	
	\begin{proof}
	Using \eqref{unitform} and \eqref{cohomo groups}, the unit $u \in (K\mathbb{R}^0({\C P^{\infty}_{\tau}}))^{\times}$ takes the form 
		$$u\ =\ (-1)+(\text{an odd integer})\times t + \text{(terms involving higher orders of } t \mbox{)}$$
		as by Lemma \ref{identify group 1} and Proposition \ref{identify tau*}, $\nu^* (\zeta_{\C P^2})$ is an odd integer. Now, we use Proposition \ref{groups}, and calculate the map  
$$u-1:K\mathbb{R}_*({\C P^{\infty}_{\tau}}) \rightarrow K\mathbb{R}_*({\C P^{\infty}_{\tau}})$$ 
 of Proposition \ref{LES of THR} using cap products. This takes the from
		\begin{myeq} \label{unitcap}
		\begin{aligned}
			&K\mathbb{R}_{\star} \{\beta_0, \beta_1, \beta_2,...\} \xrightarrow{u-1} K\mathbb{R}_* \{\beta_0, \beta_1, \beta_2,...\}\;  \text{ where } |\beta_i| = i\rho \\
			&\beta_0 \mapsto -2\beta_0\\
			&\beta_1 \mapsto -2\beta_1 + t _{(0)} ^{(1)} \times \beta_0 \; \text{    ,where } t _{(0)} ^{(1)} \text{ is odd integer}\\
			&\beta_2 \mapsto -2\beta_2 + t_{(1)} ^{(2)} \times \beta_1\  + t_{(0)} ^{(2)} \times \beta_0  \; \text{    ,where } t_{(1)} ^{(2)}  \text{ is odd integer}\\
			&\beta_k \mapsto -2\beta_k + t_{(k-1)} ^{(k)}  \times \beta_{k-1}\  + (\text{terms involving }\beta_0, \beta_1, ..., \beta_{k-2})  \; \text{    ,where } t_{(k-1)} ^{(k)}  \text{ is odd integer}
		\end{aligned}
		\end{myeq}
	Therefore, using periodicity for $\pi _{\bigstar} ^{C_2} (\KR)$, the long exact sequence of Proposition \ref{LES of THR} breaks into parts for calculation of $\THR^{\KR} _{k\rho + l'} (\KR /2)$ as follows: (for $l \equiv l' \, (\text{mod} 8)$, with $0 \leq l \leq 7$)
		\begin{align*}
			l=0: \; \; &\KR _{k\rho} (\C P^\infty _\tau) \rightarrow \KR _{k\rho} (\C P^\infty _\tau) \rightarrow \THR^{\KR} _{k\rho} (\KR /2) \rightarrow 0 \\
			&\oplus\Z\hookrightarrow \oplus\Z \rightarrow \Z / ({2^{\infty}}) \rightarrow 0 \\
			&\beta_0 \mapsto -2\beta_0\\
			&\beta_1 \mapsto -2\beta_1 + (\text{odd integer}) \times \beta_0 \\
			&\beta_2 \mapsto -2\beta_2 + (\text{odd integer}) \times \beta_1\  +\ \text{term involving }\beta_0 \\
			&\beta_k \mapsto -2\beta_k + (\text{odd integer}) \times \beta_{k-1}\  +\ \text{terms involving }\beta_0, \beta_1, ..., \beta_{k-2}  \eqref{unitcap}\\
                        & \implies \THR^{\KR} _{k\rho} (\KR /2) \cong \Z / ({2^{\infty}}).
		\end{align*}		
\begin{align*}
			l=1: \; \; &\resizebox{14cm}{!}{$\KR _{k\rho +1} (\C P^\infty _\tau) \rightarrow \KR _{k\rho+1} (\C P^\infty _\tau) \rightarrow \THR^{\KR} _{k\rho+1} (\KR /2) \rightarrow \text{ker} \Bigl( \KR _{k\rho} (\C P^\infty _\tau) \rightarrow \KR _{k\rho} (\C P^\infty _\tau) \Bigr) $}\\
			& \mbox{ker}=\Z/2 \subset \oplus\Z/2 \twoheadrightarrow \oplus\Z/2 \rightarrow  0  \rightarrow 0 \\
			&\;\beta_0 \mapsto 0 \\
			&\;\beta_1 \mapsto \beta_0 \\
			&\;\beta_2 \mapsto \beta_1 + \text { terms involving } \beta_0 \\
			&\; ... \eqref{unitcap} \\ 
                       & \implies \THR^{\KR} _{k\rho+1} (\KR /2) \cong 0.
		\end{align*}		
\begin{align*}
			l =2: \; \; &\resizebox{14cm}{!}{$\KR _{k\rho +2} (\C P^\infty _\tau) \rightarrow \KR _{k\rho+2} (\C P^\infty _\tau) \rightarrow \THR^{\KR} _{k\rho+2} (\KR /2) \rightarrow \text{ker} \Bigl( \KR _{k\rho +1} (\C P^\infty _\tau) \rightarrow \KR _{k\rho+1} (\C P^\infty _\tau) \Bigr)$} \\
			&\mbox{ker}=\Z/2 \subset \oplus\Z/2 \twoheadrightarrow \oplus\Z/2 \rightarrow  \Z/2 \rightarrow \Z / 2 \\
			&\;\beta_0 \mapsto 0 \\
			&\;\beta_1 \mapsto \beta_0 \\
			&\;\beta_2 \mapsto \beta_1 + \text { terms involving } \beta_0 \\
			&\; ... \eqref{unitcap} \\
                        & \implies  \THR^{\KR} _{k\rho+2} (\KR /2) \cong \Z / 2.
		\end{align*}		

\begin{align*}
			l =3: \; \; &\resizebox{14cm}{!}{$\KR _{k\rho +3} (\C P^\infty _\tau) \rightarrow \KR _{k\rho+3} (\C P^\infty _\tau) \rightarrow \THR^{\KR} _{k\rho+3} (\KR /2) \rightarrow \text{ker} \Bigl( \KR _{k\rho +3} (\C P^\infty _\tau) \rightarrow \KR _{k\rho+2} (\C P^\infty _\tau) \Bigr)$} \\
			&0\rightarrow 0 \rightarrow  \Z/2 \rightarrow \Z / 2  \\
                        & \implies \THR^{\KR} _{k\rho+3} (\KR /2) \cong \Z / 2.
		\end{align*}		
\begin{align*}
			l=4: \; \; &\KR _{k\rho+4} (\C P^\infty _\tau) \rightarrow \KR _{k\rho+4} (\C P^\infty _\tau) \rightarrow \THR^{\KR} _{k\rho+4} (\KR /2) \rightarrow 0 \\
			&\oplus\Z\hookrightarrow \oplus\Z \rightarrow \Z /(2^{\infty}) \rightarrow 0 \\
			&\beta_0 \mapsto -2\beta_0\\
			&\beta_1 \mapsto -2\beta_1 + (\text{odd integer}) \times \beta_0 \\
			&\beta_2 \mapsto -2\beta_2 + (\text{odd integer}) \times \beta_1\  +\ \text{term involving }\beta_0 \\
			&\beta_k \mapsto -2\beta_k + (\text{odd integer}) \times \beta_{k-1}\  +\ \text{terms involving }\beta_0, \beta_1, ..., \beta_{k-2} \eqref{unitcap} \\
&\implies \THR^{\KR} _{k\rho+4} (\KR /2) \cong \Z /( 2^\infty).\\
			l= 5,6,7: \; \; & \THR^{\KR} _{k\rho + l'} (\KR /2) \cong 0 .
		\end{align*}
		
		A similar calculation of the $C_2$-equivariant homotopy groups of $\KR / (2^{\infty})$ shows they are abstractly isomorphic to that of $\THR^{\KR} (\KR /2)$. Now, we arrive at the following map of cofiber sequences of $\KR$-modules:
		\[\begin{adjustbox}{max width=\linewidth}
\begin{tikzcd}[cramped]
			{\KR \wedge {\C P^\infty_{\tau} }_{+} \simeq \bigvee _{n \geq 0} \Sigma^{n\rho} \KR } && {\KR \wedge {\C P^\infty_{\tau} }_{+} \simeq \bigvee _{n \geq 0} \Sigma^{n\rho} \KR } && {\THR^{\KR} (\KR /2)} \\
			\KR && {\KR [2^{-1}]} && {\KR / 2^{\infty}}
			\arrow["{u - 1}", from=1-1, to=1-3]
			\arrow["{(\gamma)}"', from=1-1, to=2-1]
			\arrow[from=1-3, to=1-5]
			\arrow["{(\delta)}"', from=1-3, to=2-3]
			\arrow["{(\eta)}"', from=1-5, to=2-5]
			\arrow[from=2-1, to=2-3]
			\arrow[from=2-3, to=2-5]
		\end{tikzcd}
\end{adjustbox}\]
		with the top cofiber sequence coming from Proposition \ref{LES of THR}. The maps $(\gamma)$ and $(\delta)$ comes from the consequent set of isomorphisms
		$$\left[ \bigvee _{n \geq 0} \Sigma^{n\rho} \KR, \KR \right]^{\KRmod} \cong \prod _{n \geq 0} \pi ^{C_2} _{n\rho} (\KR) \cong \prod_{n \geq 0} \Z,$$
		$$\left[ \bigvee _{n \geq 0} \Sigma^{n\rho} \KR, \KR [2^{-1}] \right] ^{\KRmod} \cong \prod _{n \geq 0} \pi ^{C_2} _{n\rho} (\KR [2^{-1}]) \cong \prod_{n \geq 0} \Z [2^{-1}],$$
		where we have the following correspondence
		\begin{align*}
			(\gamma) \longleftrightarrow & \left( 1,1,1, \dots \right) \in \Pi_{n \geq 0} \Z \\
			(\delta) \longleftrightarrow & \left( -1,\, -1 - t_{(0)} ^{(1)} /2, \, -1 -t_{(1)} ^{(2)} /2  -t_{(0)} ^{(2)} /2 , \,  \dots \right) \in \prod_{n \geq 0} \Z [2^{-1}]
		\end{align*}
		The calculations above show that $(\eta)$ induces an isomorphism on $\pi ^{C_2} _{\star}$. 
	\end{proof}
	
	\vspace*{0.2cm}

\bibliographystyle{siam}
\bibliography{ETRA.bib}

\end{document}